\documentclass[a4paper]{article}
\usepackage{amsthm,amsmath,amssymb, bbm}
\usepackage{color}
\usepackage{enumitem}

\newtheorem{teor}{Theorem}[section]
\newtheorem{defin}[teor]{Definition}
\newtheorem{lemm}[teor]{Lemma}
\newtheorem{osse}[teor]{Remark}
\newtheorem{prop}[teor]{Proposition}
\newtheorem{defi}[teor]{Definition}
\newtheorem{coro}[teor]{Corollary}
\newtheorem{prob}[teor]{Problem}
\newtheorem{assu}[teor]{Assumption}

\newcommand{\bele}{\begin{lemm}\begin{sl}}
\newcommand{\enle}{\end{sl}\end{lemm}}
\newcommand{\bedef}{\begin{defi}\begin{sl}}
\newcommand{\eddef}{\end{sl}\end{defi}}
\newcommand{\bete}{\begin{teor}\begin{sl}}
\newcommand{\ente}{\end{sl}\end{teor}}
\newcommand{\beos}{\begin{osse}\begin{rm}}
\newcommand{\eddos}{\end{rm}\end{osse}}
\newcommand{\beas}{\begin{assu}\begin{rm}}
\newcommand{\eddas}{\end{rm}\end{assu}}
\newcommand{\bepr}{\begin{prop}\begin{sl}}
\newcommand{\empr}{\end{sl}\end{prop}}
\newcommand{\bepro}{\begin{prob}\begin{rm}}
\newcommand{\empro}{\end{rm}\end{prob}}
\newcommand{\bede}{\begin{defin}\begin{sl}}
\newcommand{\edde}{\end{sl}\end{defin}}
\newcommand{\beco}{\begin{coro}\begin{sl}}
\newcommand{\enco}{\end{sl}\end{coro}}

\newcommand{\quext}{\quad\text}
\newcommand{\qquext}{\qquad\text}
\newcommand{\de}{\partial}

\newcommand{\RR}{\mathbb{R}}

\newcommand{\NN}{\mathbb{N}}

\newcommand{\N}{{\cal N}}

\newcommand{\beeq}[1]{\begin{equation}\label{#1}}
\newcommand{\eddeq}{\end{equation}}

\newcommand{\beeqa}[1]{\begin{eqnarray}\label{#1}}
\newcommand{\eddeqa}{\end{eqnarray}}

\newcommand{\beal}[1]{\begin{align}\label{#1}}
\newcommand{\eddal}{\end{align}}

\newcommand{\bespl}[1]{\begin{split}\label{#1}}
\newcommand{\edspl}{\end{split}}

\newcommand{\bega}[1]{\begin{gather}\label{#1}}
\newcommand{\edga}{\end{gather}}

\newcommand{\beeqax}{\begin{eqnarray*}}
\newcommand{\eddeqax}{\end{eqnarray*}}

\def\qed{\ifmmode % if math mode, assume display: omit penalty etc.
  \else \leavevmode\unskip\penalty9999 \hbox{}\nobreak\hfill
  \fi
  \quad\hbox{\hskip.5em\vrule width.4em height.6em depth.05em\hskip.1em}}
\def\endproofsym{\qed}
\renewenvironment{proof}[1][Proof]{\trivlist\item[\hskip\labelsep{\hskip0pt
    {\normalfont\scshape#1.}\hskip .321429\parindent}]\ignorespaces}
{\endproofsym\endtrivlist}
\def\endnobox{\def\endproofsym{}\end{proof}\def\endproofsym{\qed}}

\newcommand{\no}{\nonumber}

\newcommand{\beeqao}{\begin{eqnarray}\no}
\newcommand{\bealo}{\begin{align}\no}
\newcommand{\besplo}{\begin{split}\no}
\newcommand{\begao}{\begin{gather}\no}
\def\<#1>{\mathopen\langle #1\mathclose\rangle}
\def\Vp{{V^*}}

\def\genspazio #1#2#3#4#5{#1^{#2}(#5,#4;#3)}
\def\spazio #1#2#3{\genspazio {#1}{#2}{#3}T0}

\def\L {\spazio L}
\def\H {\spazio H}
\def\W {\spazio W}

\def\Lx #1{L^{#1}(\Omega)}
\def\Hx #1{H^{#1}(\Omega)}

\newcommand{\eps}{\varepsilon}

\newcommand{\duav}[1]{\langle{#1}\rangle}

\newcommand{\duavg}[1]{\left\langle{#1}\right\rangle}

\newcommand{\itt}{\int_0^t}

\newcommand{\io}{\int_\Omega}

\newcommand{\epsi}{\varepsilon}

\newcommand{\OO}{_{\Omega}}
\newcommand{\oo}{_{\Omega}}

\def\R{\mathbb R}

\newcommand{\bn}{\boldsymbol{n}}

\newcommand{\dn}{\partial_{\bn}}
\newcommand{\fhi}{\varphi}

\newcommand{\lhs}{left-hand side}
\newcommand{\rhs}{right-hand side}

\DeclareMathOperator{\dive}{div}
\DeclareMathOperator{\deriv}{d}

\DeclareMathOperator{\sign}{sign}

\newcommand{\LDH}{L^2(0,T;H)}
\newcommand{\LDV}{L^2(0,T;V)}

\newcommand{\LIH}{L^\infty(0,T;H)}
\newcommand{\LIV}{L^\infty(0,T;V)}

\newcommand{\LDHD}{L^2(0,T;H^2(\Omega))}

\let\TeXchi\chi
\def\chi{{\setbox0 \hbox{\mathsurround0pt
$\TeXchi$}\hbox{\raise\dp0 \copy0 }}}

\newcommand{\calH}{{\mathcal H}}

\newcommand{\calF}{{\mathcal F}}
\newcommand{\calE}{{\mathcal E}}
\newcommand{\calJ}{{\mathcal J}}

\newcommand{\calN}{{\mathcal N}}

\newcommand{\calV}{{\mathcal V}}
\newcommand{\calZ}{{\mathcal Z}}

\newcommand{\dit}{\deriv\!t}
\newcommand{\dis}{\deriv\!s}

\newcommand{\dir}{\deriv\!r}

\newcommand{\ddt}{\frac{\deriv\!{}}{\dit}}

\newenvironment{giuliorev}{\color{magenta}}{\color{black}}
\newcommand{\III}{\begin{giuliorev}}
\newcommand{\EEE}{\end{giuliorev}}
 
\definecolor{rosso}{rgb}{0.85,0,0}

\def\norma #1{\mathopen \| #1\mathclose \|}
\def\II {\mathbb{I}}

\def\cd{c_{\d}}
\def\d{\delta}

\def\Wn {{\cal W_{\bn}}}

\def\Accorpa #1#2 #3 {\gdef #1{\eqref{#2}-\eqref{#3}}%
  \wlog{}\wlog{\string #1 -> #2 - #3}\wlog{}}

\def\mobm{{\mathbbm{m}}}
\def\mobn{{\mathbbm{n}}}

\numberwithin{equation}{section}
\begin{document}

\title{On a Cahn--Hilliard--Keller--Segel model with generalized logistic source 
describing tumor growth}

\author{Elisabetta Rocca\\
Dipartimento di Matematica, Universit\`a di Pavia {and IMATI - C.N.R.},\\
Via Ferrata~5, I-27100 Pavia, Italy\\
E-mail: {\tt elisabetta.rocca@unipv.it}
\and
Giulio Schimperna\\
Dipartimento di Matematica, Universit\`a di Pavia {and IMATI - C.N.R.},\\
Via Ferrata~5, I-27100 Pavia, Italy\\
E-mail: {\tt giusch04@unipv.it}
\and
Andrea Signori\\
Dipartimento di Matematica, Universit\`a di Pavia,\\
Via Ferrata~5, I-27100 Pavia, Italy\\
E-mail: {\tt andrea.signori01@unipv.it}
}

\date{}

\maketitle
\begin{abstract}
 \noindent
 We propose a new type of diffuse interface model describing the evolution of a tumor mass
 under the effects of a chemical substance (e.g., a nutrient or a drug). The process is 
 described by utilizing the variables $\fhi$, an order parameter representing the local proportion of 
 tumor cells, and $\sigma$, representing the concentration of the chemical. 
 The order parameter $\fhi$ is assumed to satisfy a suitable form of the Cahn--Hilliard equation
 with mass source and logarithmic potential of Flory--Huggins type (or generalizations
 of it). The chemical concentration $\sigma$ satisfies a reaction-diffusion 
 equation where the cross-diffusion term has the same expression
 as in the celebrated Keller--Segel model. In this respect, the model we propose represents a 
 new coupling between the Cahn--Hilliard equation and a subsystem of the Keller--Segel model.
 We believe that, compared to other models, this choice is more effective in capturing the 
 chemotactic effects that may occur in tumor growth dynamics (chemically induced
 tumor evolution and consumption of nutrient/drug by tumor cells). Note that, in order 
 to prevent finite time blowup of $\sigma$, we assume a chemical source term of 
 logistic type. 
 Our main mathematical result is devoted to proving existence of weak solutions 
 in a rather general setting that covers both the two- and three- dimensional cases. 
 Under more restrictive assumptions on coefficient and data, and in some cases on the
 spatial dimension, we prove various regularity results. Finally, 
 in a proper class of smooth solutions we show uniqueness 
 and continuous dependence on the initial data in a number of significant cases. 
\end{abstract}

\noindent {\bf {Keywords:}}~~Cahn--Hilliard--Keller--Segel, chemotaxis, cross-diffusion, tumor growth, 
singular potential, well-posedness.

\vspace{2mm}

\noindent {\bf AMS (MOS) subject clas\-si\-fi\-ca\-tion:}  35D30, %Weak solutions to PDEs 
	    35K35, %Partial differential equations
%	    35K57, %Reaction-diffusion equations
	    35K86, %Parabolic equations and systems
	    35Q92, %PDEs in connection with biology, chemistry and other natural sciences
        92C17, %Biology and other natural sciences...Cell movement (chemotaxis, etc.) 
	    92C50. %Medical applications (general) 

\vspace{2mm}

%%%%%%%%%%%%%%%%%%%%%%%%%%%%%%%%%%%%%%%%%%%%%%%%%%%%%%%%%%%%%%%%%%%%%%%%%%%%%%%%%%
%%%%%%%%%%%%%%%%%%%%%%%%%%%%%%%%%%%%%%%%%%%%%%%%%%%%%%%%%%%%%%%%%%%%%%%%%%%%%%%%%%

\section{Introduction}
\label{sec:intro}

Let $\Omega \subset \R^d$, $d \in \{2,3\}$, be a smooth and bounded domain,
and let $T>0$ be an assigned final time. 
In this paper, we consider the following Cahn--Hilliard-Keller--Segel (CHKS) 
model aimed at describing some classes of tumor growth processes:
\begin{alignat}{2}
 \label{CH1}
  & \fhi_t - \dive\big(\mobm(\fhi,\sigma) \nabla \mu\big) 
    = S(\fhi,\sigma) && \qquad \text{ in }\, Q:= \Omega \times (0,T),\\
 \label{CH2}
  & \mu = - \eps \Delta \fhi + \eps^{-1}  f(\fhi) - \chi\sigma && \qquad \text{ in }\, Q,\\
 \label{nutr}
  & \sigma_t -\dive \big(\sigma \mobn (\fhi,\sigma)  \nabla (\ln \sigma + \chi (1-\fhi) )\big) 
   = b (\fhi, \sigma) && \qquad \text{ in }\, Q,\\
 \label{BC}
  & \dn \fhi = (\mobm(\fhi,\sigma) \nabla \sigma) \cdot \bn
    = (\mobn (\fhi,\sigma)  \nabla (\ln \sigma + \chi (1-\fhi) ))  
    \cdot \bn = 0 && \qquad \text{ on }\, \Sigma := \partial\Omega \times (0,T),\\
  \label{init}
  & \fhi|_{t=0} = \fhi_0, \quad
  \sigma|_{t=0} = \sigma_0 
  && \qquad \text{ in }\, \Omega.
\end{alignat}
\Accorpa\SYS {CH1} {init}
Equations \eqref{CH1}-\eqref{CH2} correspond to a generalized 
version of the Cahn--Hilliard (CH)
system with mass source for the two unknown variables $\fhi$ and $\mu$. 
Here, $\fhi$ denotes an order parameter, or phase-field, representing
the difference between the tumor cells and healthy cells volume fractions,
and is normalized in such a way that, at least in principle,
the level sets $\{\fhi=1\}:= \{x \in \Omega : \fhi (x) = 1\}$ 
and $\{\fhi = -1\}$ describe the regions occupied by the 
pure (``tumor'' and ``healthy'') phases,
respectively. These regions are separated by a narrow transition layer
of thickness scaling as $\eps\in(0,1)$, in which $\{-1< \fhi< 1\}$. 
As we will specify below, the fact that $\fhi$ takes value in the reference
interval $[-1,1]$ is enforced by the occurrence of the function
$f$ in \eqref{CH2}, which represents the derivative of what, in the Cahn--Hilliard
terminology, is generally noted as a ``singular (configuration) potential''.
The variable $\mu$ is an auxiliary quantity denoting the chemical potential 
of the phase separation process. Since in tumor growth processes the total mass
of the tumor is not conserved, we also assume the occurrence 
of a volumic source term $S$ on the \rhs\ of \eqref{CH1}.
We shall comment on the precise expression of $S$ later on.

The Cahn--Hilliard system \eqref{CH1}-\eqref{CH2}  (cf.~\cite{CH}) is coupled with 
the reaction-diffusion equation \eqref{nutr} describing the effects of a 
chemical substance on the evolution of the tumor. This may be 
a nutrient like oxygen or glucose which constitutes the primary source 
of nourishment for the tumor cells, as well as a drug or a medicine 
preventing the tumor to grow. In either case, the concentration of 
such a substance is represented by the variable $\sigma$. 
We shall extensively comment below on the expression of equation \eqref{nutr}.
The functions $\mobm(\fhi,\sigma)$ and $\mobn(\fhi,\sigma)$ in \eqref{CH1} 
and \eqref{nutr} are nonnegative mobility functions related to the phase-field 
and the nutrient concentration, respectively. The system 
is complemented with the Cauchy conditions \eqref{init}
and with the no-flux (i.e., homogeneous Neumann)
boundary conditions \eqref{BC}, where $\bn$ is the outer unit normal vector 
to $\de\Omega$.

Diffuse interface models for tumor growth are now receiving a notable
attention among the scientific community and the recent mathematical literature
is very vast (we may quote, with no claim of completeness
\cite{CGH, CGRS1, CGRS2, Ciarletta, DFRSS17, FGR, FLR,  GLDirichlet, GLNeumann, GLSS, HZO, HMV, HKNZ15, H1, H2, GLRS, GLS, SS},
see also the reference therein). Actually, most of the models considered in these papers turn out to couple
a Cahn--Hilliard relation for the tumor cell proportion (which may be of 
multi-phase type if more than two types of cells are considered, 
cf., e.g., \cite{FLRS, KS2, GAR}) 
with other equations describing the behavior of further significant 
quantities, like nutrient concentration (as in our case), macroscopic
velocity, or even temperature \cite{Ipo}. 

Compared to previous tumor growth models of the same type (i.e., based
on the coupling of the Cahn--Hilliard system with a reaction-diffusion equation),
the main novelty in our system~\SYS\ is represented by the 
expression of the reaction-diffusion equation \eqref{nutr}, which is also
what led us to use the terminology ``Cahn--Hilliard--Keller--Segel model''.
In this direction, we are aware of the recent contribution \cite{EPP}, where 
a connection between a generalized form of the Keller--Segel system and a 
relaxed version of the Cahn--Hilliard system is rigorously shown 
through a suitable limiting procedure).
In a sense, the biological effect we would like to represent is {\sl chemotaxis}, 
basically corresponding to the active movement, in a biological sense,
of the tumor cells towards regions of high nutrient concentration. Considering for 
simplicity the case of a constant mobility $\mobn\equiv1$, in previous 
models (see, e.g., \cite{GLSS}), this ``active transport'' effect was described
utilizing a relation of the form
\begin{equation}\label{nutr:old}
  \sigma_t - \Delta \sigma + \chi \Delta \fhi = b (\fhi, \sigma),
\end{equation}
where $b$ is, as in our case, a volumic nutrient source. However, relation 
\eqref{nutr:old}, which is mathematically simpler compared to 
\eqref{nutr}, in our view seems to present several drawbacks from a modeling
perspective. First of all, in view of the fact that the term $\chi \Delta\fhi$ 
has no sign properties, \eqref{nutr:old} does not obey the minimum principle;
hence, one cannot exclude, at least in principle, that the variable
$\sigma$ might somewhere assume strictly negative values 
conflicting with the physical interpretation of $\sigma$ as a concentration.
A further issue can be observed if one integrates \eqref{nutr:old} on 
a reference volume $V\subset\Omega$. Indeed, applying the Gauss--Green
formula, one then obtains
\begin{equation}\label{nutr:V}
  \ddt \int_V \sigma = \int_{\partial V} \dn \sigma
   + \int_V b(\fhi,\sigma) - \chi \int_{\partial V} \dn \fhi,
\end{equation}
and we may notice that the last integral prescribes that 
the variation of $\sigma$ in $V$ depends on the flux
of tumor cells across $\de V$, {\sl independently of}\/ the value of $\sigma$. 
For instance, if many tumor cells ($\fhi\sim 1$) are present outside $V$
and fewer ones ($\fhi\sim-1$) occur inside $V$ (so that $\dn \fhi$ is positive),
then there is a nutrient flux from the inside to the outside of $V$, 
but this flux is in fact independent of the actual nutrient concentration.

On the other hand, if \eqref{nutr:old} is replaced by our \eqref{nutr}, then
(still in the case $\mobn\equiv1$), \eqref{nutr:V} assumes the different form
\begin{equation}\label{nutr:V2}
  \ddt \int_V \sigma = \int_{\partial V} \dn \sigma
   + \int_V b(\fhi,\sigma) - \chi \int_{\partial V} \sigma \dn \fhi,
\end{equation}
where, as physically expected, the nutrient flux across $\de V$ driven by
consumption by tumor cells is proportional to the actual value of 
$\sigma$: the more nutrient is present, the more it flows away.
This is, indeed, the main reason that led us to consider the present expression
for the equation \eqref{nutr}.

It is clear that the above choice, corresponding in the constant
mobility case to the equation
\begin{equation}\label{nutr:simpl}
  \sigma_t - \Delta \sigma + \chi \dive (\sigma \nabla \fhi) = b (\fhi, \sigma),
\end{equation}
gives rise to a number of mathematical complications mainly due to the 
quadratic behavior of the cross-diffusion term. This is, indeed, one of 
the main sources 
of difficulty in the mathematical analysis of the Keller--Segel (KS) model \cite{KS}.
Despite the vastness of the mathematical literature dealing 
with the KS model (cf., e.g., \cite{BPCCZ,FLP, JL, Wri1,Wri2, Wri3}),
it is worth noting that, up to our knowledge, this
is the first paper where the coupling between a ``Keller--Segel-like'' expression
of the form \eqref{nutr} (or \eqref{nutr:simpl}) with the Cahn--Hilliard
system is considered. From a modeling perspective, while in true 
Keller--Segel models, a relation like \eqref{nutr} is combined 
with a {\sl second order}\/ reaction-diffusion equation describing the 
evolution of a further {\sl concentration}, in the present coupling,
relation \eqref{nutr} is coupled with a {\sl fourth order}\/ equation 
describing the evolution of a {\sl proportion}, i.e., of a {\sl normalized}\/
variable, the order parameter $\fhi$. This new type of coupling has some 
implications both on the regularity of solutions and on the mathematical
techniques we use to address the system. For instance, we may notice that,
compared to the case when the coupling variable $\fhi$ satisfies a second
order relation (like in the true KS model), here $\fhi$ enjoys 
{\sl more regularity}\/ in space, but {\sl less regularity}\/ in time. 
This leads to some modifications of the regularity scenario and 
of the expected properties of solutions compared to the standard
KS case.

It is worth noting that, as also happens in the KS model, the regularity obtained 
by the a-priori estimate corresponding to the energy balance 
principle (the variational formulation of the model starting from the free energy 
balance is presented below) seems not sufficient to prevent finite time blowup of the solution,
unless the mass source term $b$ in \eqref{nutr} is suitably designed.
In particular, as is habitual in the Keller--Segel context, we have to assume 
$b$ to present a ``generalized logistic growth'' property (see the next section
for the precise assumption); namely, it goes like $\sigma$ for $\sigma\sim 0$
(so to preserve the minimum principle), while it behaves as $-\sigma^p$ 
(for suitable $p>1$, with the reference case given by $p=2$ corresponding
to a ``true'' logistic growth) for large $\sigma$ (see \cite{HH, Wri1, Wri2} 
for examples of Keller--Segel models with logistic growth). 
With this choice, relation \eqref{nutr:V2}
prescribes that, if the nutrient concentration is high, then there 
occurs a volumic effect leading it to decrease. We believe this property be
biologically reasonable, in addition to being probably unavoidable mathematically.

As anticipated above, system~\SYS\ could be variationally derived from 
the free energy functional
\begin{align}\label{freeenergy}
  {\cal F} (\fhi, \sigma)
    = \underbrace{\frac \eps2 \int_\Omega |\nabla \fhi |^2 + \frac 1 \eps\io F(\fhi)}_{=:{\cal E}(\fhi)}  
      + \underbrace{\io \big( \sigma (\ln \sigma - 1) 
      + \chi  \sigma {(1-\fhi)} \big) }_{=: {\cal M}(\fhi,\sigma)},
\end{align}
where $F$ is an antiderivative of $f$. In particular, equation \eqref{CH1} is obtained as a balance
law by setting
\begin{align*}
  \fhi_t + \dive {\bf J}_\fhi = S(\fhi,\sigma),
\end{align*}
where, as is typical for the Cahn--Hilliard equation, the flux ${\bf J}_\fhi$ 
is prescribed as ${\bf J}_\fhi= - \mobm(\fhi,\sigma) \nabla \mu$ for a mobility 
function $\mobm(\fhi,\sigma)$, and where the chemical potential $\mu$ is 
defined as the variational derivative of the free energy with respect to the order parameter,
namely $\mu:= \delta {\cal F}/\delta \fhi$.
Note that also equation \eqref{nutr} can be obtained as a 
balance law for the nutrient flux ${\bf J}_\sigma $, i.e.,
\begin{align*}
   \sigma_t + \dive {\bf J}_\sigma = b(\fhi,\sigma), \quad 
     \text{with }\,{\bf J}_\sigma := - \sigma \mobn(\fhi,\sigma) \nabla \mu_\sigma, \quad 
      \mu_\sigma  :=  \frac {\delta {\cal F}}{\delta \sigma}
                = \frac {\delta {{\cal M}}}{\delta \sigma} = \ln \sigma + \chi (1-\fhi),
\end{align*}
where the mobility function has the expression $\sigma \mobn(\fhi,\sigma)$, hence,
in particular, degenerates (in fact linearly) as $\sigma\searrow 0$
(so guaranteeing the minimum principle).

The above expression \eqref{freeenergy} for the free energy permits us 
to remark a further peculiarity of the present model. This is related 
to the coercivity of $\calF$, which is linked to the choice of 
a ``singular potential'' $F$, with the 
most usual choice in the Cahn--Hilliard literature being given by the 
Flory--Huggins ``logarithmic potential'' given by
\begin{equation}\label{Flog}
  F(r) = (1+r)\log(1+r)+(1-r)\log(1-r)-\frac\lambda2 r^2, \quad r\in[-1,1],
  \quad \lambda \ge 0.
\end{equation}
For the standard Cahn--Hilliard model the expression
\eqref{Flog} represents a source of mathematical difficulties 
(cf., e.g., \cite{MAIMS}), due to its singular character,
and, for this reason, it is often replaced by a double well potential of controlled
growth like, e.g., $F_{{\rm reg}}(r) = (r^2 - 1)^2$. Here, instead,
the singular character of $F$ {\sl helps}\/ us to get coercivity of the energy
functional, and in particular to control the coupling term 
(i.e., the last summand in \eqref{freeenergy}).

Notice also that such a difficulty does not occur when the nutrient
equation has the form~\eqref{nutr:old}. Indeed, in that case the 
free energy takes the expression
\begin{align*}
%\label{free2}
  {\cal F}_2 (\fhi, \sigma)
    = \frac \eps2 \int_\Omega |\nabla \fhi |^2 + \frac 1 \eps\io F(\fhi)
      + \io \Big( \frac12 \sigma^2 + \chi  \sigma (1-\fhi) \Big),
\end{align*}
which keeps its coercivity because of the contribution of $\sigma^2$ (note also
that, in this case, the variational derivation of the model is similar, but
one has to consider a mobility of the form $\mobn(\sigma,\fhi)$
rather than $\sigma \mobn(\sigma,\fhi)$).

We also have to observe a further difficulty occurring in Cahn--Hilliard models
with mass source and singular potentials like \eqref{Flog}. Namely, the 
forcing term $S$ in \eqref{CH1} has to be designed in such a way to prevent
the spatial average of $\fhi$ to become larger than $1$ or smaller than $-1$,
which would be inconsistent with \eqref{CH2}. Indeed, the mass balance
(i.e., the evolutionary law ruling the spatial average of $\fhi$)
only depends on \eqref{CH1}, but at the same time
its outcome must be consistent with \eqref{CH2}.
Following the lines of \cite{FLRS}, we actually assume $S(\fhi,\sigma) 
= -m \fhi + h(\fhi,\sigma)$, where $m>0$ is ``large'' compared to the 
$L^\infty$-norm of the (bounded) function $h$, which is readily 
seen to be an appropriate choice (see Subsec.~\ref{sub:mass} below
for details). Note also that, for constant $h$, \eqref{CH1}-\eqref{CH2} reduces
to the well-known Cahn--Hilliard--Oono system (see, e.g., \cite{GGW,Mi,OP1,OP2}).

\smallskip

Our main mathematical results are devoted to proving existence of weak
solutions under mild conditions on parameters and data as well as regularity and 
uniqueness results holding in more restrictive settings.
In particular, under the sole ``energy regularity'' conditions
on the initial data (basically corresponding to the finiteness of the functional
$\calF$ at the initial time), we can prove existence of weak solutions
for nonconstant, bounded and nondegenerate mobilities $\mobm$, $\mobn$,
and for a wide class of logistic terms.
In particular, we provide, depending 
on the space dimension $d$, sufficient conditions on the growth of $b$ at
infinity in order to exclude the occurrence of blowup. This result
is proved by a-priori estimates and weak compactness methods. A possible
approximation scheme compatible with the a-priori estimates is also sketched.

In the case of true logistic growth, i.e., for $b$ behaving like 
$- \sigma^2$ at infinity, we can also present a number of regularity results
holding under additional hypotheses on the mobilities and on the other coefficients
and data. As is customary for the CH system, some regularity results
are only valid in spatial dimension $d=2$, for reasons depending 
both on the structure of equation \eqref{nutr} (and, in particular, on
the quadratic behavior of the cross-diffusion term), and on the occurrence
of the singular potential, which gives rise, in the three-dimensional
case, to an upper regularity threshold (see, e.g., \cite{MZ}).
In some cases we can also prove uniqueness; in fact, this is presented 
as a conditional result stating that two weak solutions starting 
from the same initial data and obeying some additional regularity
properties must coincide. Then, it is observed that these regularity
conditions are fulfilled for proper classes of strong solutions,
also depending on the regularity of data and on the space dimension.

\smallskip

The plan of the paper is as follows: in the next section, we introduce our precise 
assumptions and present the statements of all our mathematical results. Then, in 
Section~\ref{sec:well}, we prove existence of weak solutions, while in Section~\ref{sec:rego}
we move to the regularity results. Finally, Section~\ref{SEC:UQ} is devoted to 
uniqueness of ``strong'' solutions.

%%%%%%%%%%%%%%%%%%%%%%%%%%%%%%%%%%%%%%%%%%%%%%%%%%%%%%%%%%%%%%%%%%%%%%%%%%%%%%%%%
%%%%%%%%%%%%%%%%%%%%%%%%%%%%%%%%%%%%%%%%%%%%%%%%%%%%%%%%%%%%%%%%%%%%%%%%%%%%%%%%%

\section{Mathematical preliminaries and main results}
\label{sec:main}

%%%%%%%%%%%%%%%%%%%%%%%%%%%%%%%%%%%%%%%%%%%%%%%%%%%%%%%%%%%%%%%%%%%%%%%%%%%%%%%%%

\subsection{Notation}\label{SEC:NOT}

Before diving into the mathematical details, let us introduce the notation employed in the paper.
Letting $X$ be a Banach space, we denote by $\norma{\cdot}_X$ the corresponding norm, by $X^*$ the 
topological dual of $X$, and by $\< \cdot , \cdot >_X$ the related duality pairing between $X^*$ and $X$.
Standard Lebesgue and Sobolev spaces defined on $\Omega$, for every $1 \leq p \leq \infty$ and $k \geq 0$,
are indicated by $L^p(\Omega)$ and $W^{k,p}(\Omega)$, with associated norms
$\norma{\cdot}_{L^p(\Omega)}=\norma{\cdot}_{p}$ and $\norma{\cdot}_{W^{k,p}(\Omega)}$, respectively. 
When $p = 2$, these become Hilbert spaces and we use $\norma{\cdot}=\norma{\cdot}_2$ for the norm of 
$\Lx2$ and set $H^k(\Omega):= W^{k,2}(\Omega)$. 
Moreover, for brevity we introduce the following notation:
\begin{align*}
  & H := \Lx2, 
  \qquad  
  V := \Hx1, 
  \qquad 
  H_{\bf n}^2(\Omega) := \{ v \in \Hx2 : \dn v = 0 \,\, \text{on $\Gamma$} \},
%  \label{def:HV}
\end{align*}
where we denote by $\Gamma$ the boundary of $\Omega$, that is $\Gamma= \partial \Omega.$

For every $v\in V^*$, we use 
$v\OO:=\frac1{|\Omega|}\<{v},{1}>_V$ for the 
generalized mean value of $v$. 
Let us also point out a version of the celebrated Poincar\'e--Wirtinger inequality:
\begin{align}
  \norma{v - v\OO}  \leq c_{\Omega} \norma{\nabla v},
  \quad  v\in V,
  \label{poincare}
\end{align}
where the constant $c_\Omega>0$ depends only on $\Omega$ and the spatial dimension $d$.
The norm in $V^*$ will be simply denoted by $\| \cdot \|_*$.
Identifying $H$ with $H^*$ by employing the scalar product of $H$, we obtain the 
chain of continuous and dense embeddings $V\subset H \subset V^*$. Moreover, 
we may denote as $V_0$, $H_0$, $V_0^*$ the (closed) subspaces respectively
of $V$, $H$, and $V^*$, consisting of functions (or functionals) with zero spatial
mean. Then, we observe that the weak version of the operator $-\Delta$
with homogeneous Neumann boundary conditions, i.e.,
\begin{equation}\label{neum:lap}
   (-\Delta) : V \to V^*, \qquad 
    \duavg{(-\Delta) v, z}:= \io \nabla v\cdot \nabla z,
\end{equation}
for $v,z\in V$, is invertible when it is restricted to the functions
with zero spatial mean (i.e., when it operates from $V_0$ to $V_0^*$).
Its inverse operator will be denoted by $\calN:V_0^* \to V_0$. 

Finally, we remark that, for any $v \in \Vp$ 
there exists a positive constant $c$ such that
\begin{align*}
	| v\OO |= \Big| \frac 1 {|\Omega|} \<v,1>_{V} \Big|\leq c \norma{v}_*,
\end{align*}
whence the Poincar\'e--Wirtinger inequality \eqref{poincare} yields 
\begin{align*}
	\norma{v}_{V} \leq c (\norma{\nabla v}+|v\OO| )
	\leq c (\norma{\nabla v}+\norma{ v}_* ), \quad v \in V.
\end{align*}
From now onward, we convey that the small-case symbol $c$ denotes every constant
that only relates to structural data of the problem and the norms of the involved functions; 
thus, its meaning may vary from line to line.
When an additional positive constant $\delta$ also enters the computation, we use $\cd$ to 
stress the dependency of $c$ on $\d$.

%%%%%%%%%%%%%%%%%%%%%%%%%%%%%%%%%%%%%%%%%%%%%%%%%%%%%%%%%%%%%%%%%%%%%%%%%%%%%%%%%

\subsection{Main results}

We describe here our basic assumptions on coefficients and data, which will
be kept for the remainder of the paper. Each assumption will be presented with
a number of comments aimed at outlining its meaningfulness in the light
of our specific application to tumor growth processes.

Moreover, we observe that more restrictive conditions, needed
for the regularity and uniqueness results, will be specified on occurrence.

\smallskip
\noindent%
{\bf (A1) - Assumptions on the potential.}~~%
We assume $F$ to be decomposed as $F=F_1 + F_2$, with 
$F_1$ denoting the ``singular'' convex part and $F_2$ the ``smooth'' nonconvex part.
The latter is simply given by $F_2(r):= - \lambda r^2/2$, $r\in\RR$, with $\lambda \ge 0$ (so including
the case $F_2\equiv0$ corresponding to a convex potential $F$). The properties
of $F_1$ are better described by using some basic notions from the theory of 
subdifferential
operators. Namely, we assume $F_1:\RR\to(-\infty,+\infty]$ be convex and lower semicontinuous
with the set $\{r\in\RR:F_1(r)<+\infty\}$ (usually indicated as {\sl domain}\/ of $F_1$
in the convex analysis terminology) coinciding either with $[-1,1]$ or with $(-1,1)$.
In such a situation it is well-known that the {\sl subdifferential} $f_1=\de F_1$ is a 
maximal monotone, possibly multivalued, operator in $\RR$ such that $\{f_1(r)\}$
is nonempty at least for $r\in(-1,1)$ and at most for $r\in[-1,1]$. Here, we are not 
interested in considering nonsmooth operators; for this reason we will also 
assume $F_1\in C^2(-1,1)$ so that $f(r)=f_1(r) +f_2(r) = F_1'(r)+F_2'(r)$ for $r\in (-1,1)$. 
Moreover, just for the sake of simplicity, we assume $F_1$ so normalized that 
$F_1'(0)=0$, which implies in particular that $F_1'(r)\ge 0$ for
$r\ge 0$ and $F_1'(r)\le 0$ for $r\le 0$. Notice that
this includes both the case of the Flory--Huggins potential \eqref{Flog} (whose 
{\sl domain}\/ is $[-1,1]$) as well as the case of ``more singular''
potentials like that considered in \cite{Sc}, i.e.,
\begin{equation}\label{Fsingsing}
  F_1(r) = - \log(1-r^2), \quad r\in(-1,1).
\end{equation}
Notice however that nonsmooth potentials, like the so-called {\sl double obstacle potential}\/
$F_1(r) = I_{[-1,1]}(r)$, with $I_{[-1,1]}$ denoting the {\sl indicator function}\/
of the interval $[-1,1]$ (cf., e.g., \cite{Br}) may be considered as well,
at least for what concerns existence of weak solutions.

\smallskip
\noindent%
{\bf (A2) - Assumptions on the mass source term.}~~%
We assume $S$ to be given by
\begin{equation}\label{hp:S}
  S(\fhi,\sigma) = - m \fhi + h(\fhi,\sigma), \quad (\fhi, \sigma) \in \RR^2,
\end{equation}
where $m > 0$ is a constant. Moreover, we assume $h$ to be uniformly bounded and Lipschitz
continuous with respect to the complex of its variables. Finally, 
the following compatibility condition is assumed to hold
\begin{align}\label{hp:h1}
  & \frac{H}{m} < 1, \quad  \text{where }\, H:= \| h \|_{L^\infty(\RR\times\RR)}.
\end{align} 
Notice that, in principle, only the behavior of $h$ over the physical reference set 
$\calH = [-1,1]\times[0,+\infty)$ is significant. On the other hand, it is worth assuming
$h$ be defined for every value of its arguments because, for instance, in an approximation,
it may happen $\fhi$ to take values outside $[-1,1]$ (cf. Subsec.\ref{subsec:appro}).

It is worth observing that, if $h$ is a constant function (still indicated as $h$ for 
notational simplicity), the expression of $S$ corresponds to that occurring in the  
so-called Cahn--Hilliard--Oono equation (see, e.g., \cite{GGW,Mi,OP1,OP2} and 
the references therein), i.e.
\begin{equation*}
%\label{S:oono}
  S(\fhi,\sigma) = - m \fhi + h,
  \quad 
  h \in (-1,1).
\end{equation*}
Moreover, we observe that the case $m\equiv h\equiv 0$, corresponding to the conservation
of total tumor mass, is admissible too, and in fact simpler to deal with. The 
variations needed to consider the situation with no mass source will be 
outlined on occurrence.

\smallskip
\noindent%
{\bf (A3) - Assumptions on the chemical source term.}~~%
We assume $b$ has a generalized logistic expression of the form
\begin{equation} \label{b:logi2}
   b(\fhi,\sigma) =  \beta(\fhi) (\kappa_0\sigma - \kappa_\infty \sigma^p),  
    \quad \fhi\in\RR,~~\sigma\ge 0,
\end{equation}
where $p \in (1,2]$ is a given exponent, and $\kappa_0> 0,\kappa_\infty > 0$ are positive
constants. Note that, in view of the minimum principle holding for equation \eqref{nutr}
(and preserved in the approximation) it is sufficient to specify the 
above expression for $\sigma\ge 0$. Here, the function $\beta:\RR \to \RR$ 
is assumed to be Lipschitz continuous and to satisfy 
\begin{alignat}{2}
 \label{beta:bou}
   & 0 \le \beta(r) \le B < + \infty &&\quext{for every }\,r\in\RR,\\
 \label{beta:bou2}
   & 0 < b_0 \le \beta(r) &&\quext{for every }\,r\in[-3/2,3/2],\\
 \label{beta:bou3}
   & \beta(r) \equiv 0 &&\quext{for every }\,r\not\in(-2,2),
\end{alignat}
where $b_0,B>0$ are given constants. In fact, in the limit, the only 
significant values of $\beta(r)$ will be those assumed as $r\in[-1,1]$. 
However, as in the case of $h$, it is necessary to extend $\beta$ also
outside that interval in view of an approximation.
We finally observe that the motivations underlying the choice of a logistic behavior for the 
chemical source have been extensively detailed in the introduction.

\smallskip
\noindent%
{\bf (A4) - Assumptions on the mobility functions.}~~%
We assume $\mobm\in C^0(\RR\times[0,+\infty))$ and 
$\mobn\in C^1(\RR\times[0,+\infty))$ to be globally Lipschitz continuous 
in the complex of their arguments,
and to satisfy
\begin{align} \label{mob:bou}
  & 0 < m_0 \le \mobm(\fhi,\sigma),\mobn(\fhi,\sigma) \le M < + \infty,
    \quext{for every }\,\fhi\in\RR,~~\sigma\ge 0,\\
  \label{mob:bou2}
  & | \partial_\fhi \mobn(\fhi,\sigma) | \le M < + \infty,
    \quext{for every }\,\fhi\in\RR,~~\sigma\ge 0,
\end{align}
where, again, $m_0,M>0$ are given constants. In order to properly state a weak formulation
of the system, we also set
\begin{equation*}
% \label{defi:N}
   N(\fhi,\sigma) := \int_0^\sigma \mobn(\fhi,s)\,\dis,
\end{equation*}
and we notice that, thanks to \eqref{mob:bou}, $N$ satisfies
\begin{equation} \label{prop:N}
   m_0 \sigma \le N(\fhi,\sigma) \le M\sigma \quext{for every }\,
    \fhi\in\RR,~~\sigma\ge 0.
\end{equation}
Moreover, it is not difficult to prove that 
\begin{align} \no
  | N(\fhi_1,\sigma_1) - N(\fhi_2,\sigma_2) | 
   & \le | N(\fhi_1,\sigma_1) - N(\fhi_1,\sigma_2) | 
     + | N(\fhi_1,\sigma_2) - N(\fhi_2,\sigma_2) | \\
 \label{Lip:N}
  & \leq L | \sigma_1 - \sigma_2 | + L \sigma_2 | \fhi_1 - \fhi_2 | 
\end{align}
where $L>0$ is a Lipschitz constant. We also need to define 
\begin{equation*}
% \label{defi:n1}
   \mobn_1(\fhi,\sigma) := \partial_\fhi N(\fhi,\sigma)
   = \int_0^\sigma \partial_\fhi \mobn(\fhi,s)\,\dis,
\end{equation*}
whence there holds the identity 
\begin{align}	\label{nablaN:identity}
  \nabla N(\fhi,\sigma) = \mobn (\fhi,\sigma) \nabla \sigma + \mobn_1 (\fhi,\sigma) \nabla \fhi.
\end{align}
Moreover, since $\mobn$ is assumed to be $C^1$, $\mobn_1$ turns out to 
be continuous and to satisfy
\begin{equation} \label{cresc:n1}
   | \mobn_1(\fhi,\sigma) | \le M \sigma 
    \quext{for every }\, \fhi\in\RR,~~\sigma\ge 0,
\end{equation}
as a direct check shows.

\medskip

In addition to the above assumptions, we take the chemotaxis 
sensitivity $\chi$ appearing in \eqref{CH2}-\eqref{nutr} 
to be a strictly positive constant. We keep its value explicit because its magnitude will play a role 
in part of the results. On the other hand, the magnitude of the interfacial energy coefficient
$\eps>0$ has no importance for the mathematical analysis. Hence, for the sake of simplicity,
we will directly take $\eps=1$ from now onward, without further reference.

The above choices lead us to rewrite system \eqref{CH1}-\eqref{nutr} in the following
form, where, for the sake of clarity, some expressions of the source terms have
been expanded:
\begin{alignat}{2}
\label{CH1b}
  & \fhi_t -\dive \big(\mobm(\fhi,\sigma) \nabla \mu\big) = - m \fhi + h(\fhi,\sigma)
  && \quad \text{in $Q$,}
   \\
  \label{CH2b}
  & \mu = - \Delta \fhi + F_1'(\fhi) - \lambda\fhi - \chi\sigma
    && \quad \text{in $Q$,}
   \\
 \label{nutrb}
  & \sigma_t - \dive \big(\mobn (\fhi,\sigma)  \nabla \sigma \big)
  - \chi \dive \big(\sigma \mobn (\fhi,\sigma)  \nabla(1-\fhi) \big) 
    = \beta(\fhi) (\kappa_0\sigma - \kappa_\infty \sigma^p)
      && \quad \text{in $Q$.}
\end{alignat}
In particular, we have written the equation for $\sigma$ in the ``decoupled''
form \eqref{nutrb} where the {cross-diffusion} term is split between two distinct
components. Indeed, this is a necessary step in order to deal with a mathematically 
tractable weak formulation. On the other hand, it is also worth recalling the following
``coupled'' version of the equation for $\sigma$ which is more suitable for 
the derivation of the a-priori estimates:
\begin{equation}\label{nutrb2}
  \sigma_t - \dive \big(\sigma \mobn (\fhi,\sigma)  
     \nabla ( \ln \sigma + \chi (1-\fhi) ) \big)
    = \beta(\fhi) (\kappa_0\sigma - \kappa_\infty \sigma^p)
     \quad \text{in }\,Q.
\end{equation}
It is clear that, as far as ``smooth'' solutions are considered, relations
\eqref{nutrb} and \eqref{nutrb2} may be interpreted as equivalent. In particular, 
this may happen in the approximation thanks to additional regularity available
at that level. 

\smallskip

We can now present our first result for the chemotaxis system \SYS\ 
concerning the existence of weak solutions in dimensions two and three holding
under the assumptions detailed above. We observe in particular that, in
order to pass to the limit in the cross-diffusion term (and in particular
to decouple its components as expressed by equation \eqref{nutrb}), we will
be forced to restrict the admissible range of the exponents $p$ 
in~\eqref{b:logi2} in a way depending on the space dimension $d$. 
In the sequel, functions of the form $g(r)=r \ln r$, or similar,
are implicitly intended to be extended, by continuity, to $r=0$ by setting $g(0)=0$.
\begin{teor}[Existence of weak solutions, $d\in\{2,3\}$]\label{THM:WEAK}
 Suppose that\/ {\rm Assumptions {\bf (A1)-(A4)}} are satisfied, let $\chi>0$ and let $d \in \{2,3\}$.
 Moreover, assume that the initial data satisfy
 \begin{align}\label{hp:fhi0}
   & \fhi_0 \in V, 
    \qquad F(\fhi_0) \in L^1(\Omega),
    \qquad  (\fhi_0)\OO \in(-1, 1),\\
  \label{hp:sigma0}
   & \sigma_0\ge 0~~\text{a.e.~in }\,\Omega,
    \qquad \sigma_0 \ln \sigma_0 \in L^1(\Omega).
 \end{align} 
 Moreover, assume that the exponent $p$ in\/ \eqref{b:logi2} satisfies
 $p \in [3/2,2]$ for $d=2$ and $p \in [8/5,2]$ for $d=3$.
 Then, system\/ \SYS\ admits at least one weak solution; namely
 there exists a triplet $(\fhi, \mu, \sigma)$ satisfying the regularity properties
 \begin{align} 
   \label{reg:1}
   & \fhi \in {\H1 \Vp \cap L^\infty(0,T;V)} \cap  L^p(0,T;W^{2,p}(\Omega)),\\
  \label{reg:1bis} 
   & \fhi \in L^\infty(Q): \, - 1 \le \fhi(x,t) \le 1 
    \quext{for a.e.~}\,(x,t)\in Q, \\
  \label{reg:2:nonneg}
   & \sigma(x,t) \ge 0
    \quext{for a.e.~}\,(x,t)\in Q,\\
  \label{reg:2:dtsigma}
   & \sigma \in C^0([0,T];\Wn\!\!^*) \cap \L\infty {\Lx1},  \\
   \label{reg:2}
   & \sigma^p \ln \sigma  \in L^1(0,T;L^1(\Omega)),
     \quad \sigma \ln \sigma  \in L^\infty(0,T;L^1(\Omega)),\\
  \label{reg:3}
   &  \sigma^{1/2} \nabla (\ln \sigma + \chi(1-\fhi)) \in {L^2(0,T;H)},\\
   \label{reg:4}
   &  \mu \in L^2(0,T;V),\\
   \label{reg:5}
   &  F(\fhi) \in {L^\infty(0,T;L^1(\Omega))}, \quad f(\fhi) \in {L^p(0,T;\Lx p)},
 \end{align}
 together with the ``pointwise'' formulation
 \begin{equation}\label{wf:mu} 
    \mu = - \Delta \fhi + f(\fhi) - \chi\sigma \quad \text{a.e. in }\, Q,
 \end{equation}
 the boundary condition
 \begin{equation}\label{wf:mu:bc} 
   \dn \fhi = 0 \quext{in the sense of traces on }\,\Gamma\times(0,T),
 \end{equation}
 and the weak variational formulations
 \begin{align}\label{wf:phi}
   & \<\fhi_t,v>_V + \io \mobm(\fhi,\sigma) \nabla \mu\cdot \nabla v = \io S(\fhi,\sigma) v,
    \quext{a.e.~in }\,(0,T), \\ 
   \no
   & \<\sigma(t), w(t) >_{{\Wn}} - \itt\io N(\sigma,\fhi) \Delta w 
      - \itt\io \mobn_1(\fhi,\sigma) \nabla \fhi \cdot \nabla w
      - \chi \itt\io \sigma \mobn(\fhi,\sigma) \nabla \fhi\cdot \nabla w  \\ 
   \label{wf:sigma}  
   & \qquad = \<\sigma_0, w(0) >_{{\Wn}} + \itt \<\sigma, w_t >_{{\Wn}}
    + \itt \io b(\fhi,\sigma) w, \quext{for every }\, t\in[0,T], 
 \end{align}
 for all test functions $v \in V$, $w \in C^1([0,T];\Wn)$, 
 where we have set
 \begin{align*}
   \Wn : = \big\{ w \in W^{1,\infty}(\Omega) \cap W^{2,p'}(\Omega):~~\dn w  = 0~~\text{on }\, \Gamma \big\},
 \end{align*}
 with $p' $ being the conjugate exponent of $p$, i.e., the exponent such that $1/p + 1/p'= 1$.
 The space $\Wn$ is naturally endowed with the graph norm, which turns it into a Banach space. 
 Besides, the initial conditions are satisfied in the sense that
 \begin{align}\label{ini:1}
   \fhi|_{t=0} & = \fhi_0 \quad \text{a.e.~in }\, \Omega,\\ 
  \label{ini:2}
   \sigma|_{t=0} & = \sigma_0 \quad \text{in }\, \Wn\!\!^*.
 \end{align}
 Furthermore, if the source term $b$ has a standard logistic growth, i.e., $b$ 
 fulfills  \eqref{b:logi2} with $p=2$, then the solution 
 $(\fhi, \mu, \sigma)$ obtained before
 satisfies the additional regularity property
\begin{equation}
\label{rego:add}
%   \fhi \in L^2(0,T; H^2(\Omega)), \quad 
       \sigma^{1/2}  \in L^2(0,T; V).
 \end{equation}
\end{teor}
\noindent%
It is worth providing some further comments on the above statement. First of all, we 
notice that, due to {\bf (A1)}, the second condition in \eqref{hp:fhi0} implies in particular
that $\fhi_0\in \Lx\infty$ with $-1\leq \fhi_0\leq 1$ almost everywhere in $\Omega$.
We also observe that relations \eqref{wf:phi}-\eqref{wf:sigma} conveniently incorporate the boundary
conditions. Finally, we {observe} that there may be proved the additional regularity
property $\sigma \in BV(0,T;\Wn\!\!^*)$. 

\smallskip

The above result may be improved as soon as the source term is pure logistic, 
i.e., $b$ verifies \eqref{b:logi2} with $p=2$.
Specifically, that additional assumption, along with natural conditions on the initial data, 
suffices to improve the regularity of the weak solutions for $d=2$
without any further restriction. In the three-dimensional case, a similar property holds
provided that the chemotactic coefficient $\chi$ is assumed small enough and 
the mobility $\mobn$ is taken as a constant function (however this condition may be partially
relaxed, see Remark~\ref{rem:mobn} below). 
\begin{teor}[Regularity properties of weak solutions]\label{THM:WEAK:2d3d}
 Suppose that Assumptions\/ {\bf (A1)-(A4)} and \eqref{hp:fhi0}-\eqref{hp:sigma0} 
 hold with $p=2$ in \eqref{b:logi2},
 and assume that the initial datum $\sigma_0$ additionally satisfies
 \begin{align}\label{ass:initialsigma:2d}
   \sigma_0 \in H.
 \end{align}
 Moreover, if $d=3$, suppose also that
 \begin{align}\label{smallness}
   & \chi < \sqrt{2 \kappa_\infty b_0 },\\
  \label{n:cost}
   & \mobn(\fhi,\sigma) \equiv 1.
 \end{align} 
 Then, the weak solution $(\fhi,\mu,\sigma)$ provided by Theorem~\ref{THM:WEAK}
 satisfies the following additional regularity properties: 
 \begin{align}\label{reg:2d:1}
   \fhi & \in H^1 (0,T; \Vp) \cap L^4(0,T; H^2_{\bn}(\Omega)) \cap L^2(0,T; W^{2,q}(\Omega)),\\ 
   \label{reg:2d:2}
   \sigma &  \in H^1 (0,T; \Vp) \cap C^0([0,T]; H) \cap L^2 (0,T; V),
 \end{align}
 where $q=6$ in~\eqref{reg:2d:1} if $d=3$, whereas one can take any $q \in [1,\infty)$
 if $d=2$.  
 Moreover, the system equations are satisfied in the following sense:
 \eqref{wf:mu}-\eqref{wf:mu:bc} hold together with the variational equalities
 \begin{align}\label{wf:1}
   & \<\fhi_t,v>_{V} + \io \mobm(\fhi,\sigma) \nabla \mu\cdot \nabla v 
      = \io S(\fhi,\sigma) v,\\ 
  \label{wf:3} 
   & \<\sigma_t,v>_{V} + \io \mobn(\fhi,\sigma) \nabla \sigma \cdot \nabla v 
      - \chi \io \sigma \mobn(\fhi,\sigma) \nabla \fhi \cdot \nabla v
     =  \io b (\fhi, \sigma)v,
 \end{align}
 for every test function $v \in V$ and almost everywhere in $(0,T)$.
 Finally, the initial conditions\/ \eqref{ini:1}-\eqref{ini:2} 
 are now both satisfied almost everywhere in $\Omega$.
\end{teor}
\noindent%
Under the assumptions of the previous theorem (including in particular\/
{\rm \eqref{smallness}} in the three-dimensional case), we can prove
additional regularity of solutions for constant {mobilities}
provided that also the initial data are smoother. This is stated 
in the following theorem.
\begin{teor}\label{THM:REG:1}
 Suppose that {Assumptions}\/ {\bf (A1)-(A4)} hold,
 with $p=2$ in \eqref{b:logi2}, together with \eqref{hp:fhi0}-\eqref{hp:sigma0}. 
 Moreover, assume 
 \begin{equation}\label{mn:cost}
   \mobm(\fhi,\sigma) \equiv \mobn(\fhi,\sigma) \equiv 1,
 \end{equation}
 and, if $d=3$, assume also\/ \eqref{smallness}. 
 If the initial data satisfy the additional conditions
 \begin{equation}\label{init:reg}
     \fhi_0 \in   H_{\bf n}^2(\Omega), 
      \qquad \mu_0 := - \Delta \fhi_0 + f(\fhi_0) - \chi\sigma_0 \in V,
   \qquad \sigma_0 \in V,
 \end{equation}
 then, the weak solution $(\fhi,\mu,\sigma)$ provided by\/ {\rm Theorem~\ref{THM:WEAK}} 
 satisfies the following additional regularity properties:
 \begin{align}\label{reg:strong:2d:1}
   & \fhi \in W^{1,\infty}(0,T;V^*) \cap H^1 (0,T;V) 
     \cap L^\infty(0,T; W^{2,q}(\Omega)),\\
  \label{reg:strong:2d:1b}
   & F_1'(\fhi) \in L^\infty(0,T;L^{q}(\Omega)),\\
  \label{reg:strong:2d:2}
   & \mu \in \LIV,\\
  \label{reg:strong:2d:3}
   & \sigma  \in H^1 (0,T; H) \cap C^0([0,T]; V) \cap \L2 {\Hx2},
\end{align}
where $q=6$ if $d=3$ and $q\in[1,\infty)$ if $d=2$. Morover, 
\eqref{wf:1}-\eqref{wf:3} can be interpreted as equations holding a.e.~in~$\Omega$
with the boundary conditions holding in the sense of traces.
\end{teor}
\noindent%
The next result, valid only in the two-dimensional case, extends to the 
present system a regularity property holding for the Cahn--Hilliard equation 
for those singular potentials whose convex part 
fulfills the growth condition
\begin{align}\label{F:growth}
  |F_1''(r )| \leq e ^{C_F(|F_1'(r)|+1)}, \quext{for every }\, r \in (-1,1),
\end{align}
for some positive constant $C_F$. 
It is well-known that \eqref{F:growth} is satisfied by the logarithmic potential in \eqref{Flog};
as one can directly check, it also holds for ``more singular'' potentials,
like \eqref{Fsingsing}, such that $|F_1'(r)|$
behaves like a negative power of $1-|r|$ as $|r|\nearrow 1$. It does not hold, instead, in the 
case of the {\it double obstacle}\/ potential.
Nevertheless, whenever \eqref{F:growth} holds,
we can prove that, for smoother initial data, the solution $\fhi$ is ``separated''
from the singular values $\pm1$ in the uniform norm. 
This is stated in the following theorem.
\begin{teor}\label{THM:REG:2}
 Suppose that Assumptions\/ {\bf (A1)-(A4)} hold with $p=2$ in \eqref{b:logi2},
 together with \eqref{hp:fhi0}-\eqref{hp:sigma0}. 
 Moreover, assume that $d=2$, the potential fulfills\/ \eqref{F:growth},
 and\/ \eqref{mn:cost} holds. If the initial data satisfy the 
 additional conditions 
 \begin{equation}\label{init:reg:2}
     \fhi_0 \in  H_{\bf n}^2(\Omega), 
      \qquad \mu_0 := - \Delta \fhi_0 + f(\fhi_0) - \chi\sigma_0 \in H_{\bf n}^2(\Omega), 
      \qquad \sigma_0 \in V,
 \end{equation}
 then the weak solution $(\fhi,\mu,\sigma)$ provided by\/ {\rm Theorem~\ref{THM:WEAK}},
 in addition to the regularity stated in Theorem~\ref{THM:REG:1},
 satisfies the following additional properties:
 \begin{align}\label{reg:2d:3}
   &\fhi  \in \W{1,\infty} H \cap H^1 (0,T; {\Hx2}) 
       \cap L^\infty(0,T; \Hx4 \cap W^{2,q}(\Omega)), \quad q \in [2,\infty),\\
  \label{reg:2d:4}
   &\mu  \in \L\infty {\Hx2}  \cap \L2 {\Hx3},\\
   \label{reg:2d:5}
    & F''_1(\fhi) \in \L\infty {\Lx{q}}.
 \end{align}
 Moreover, if the initial data also satisfy
 \begin{equation}\label{init:sigmareg}
   \sigma_0 \in L^\infty(\Omega),
 \end{equation}
 then one also has
 \begin{equation}\label{sig:infty}
    \sigma \in L^\infty(Q)
 \end{equation}
 and there exists a computable constant $\delta\in(0,1)$ only depending on the 
 problem data such that the following ``separation property'' holds:
 \begin{equation}\label{separation}
   -1+\delta \le \fhi \leq 1-\delta \quad \text{a.e.~in }\,Q.
 \end{equation}
\end{teor}
\beos\label{rem:separ0}
 A direct check shows that, if \eqref{init:reg:2} and \eqref{sig:infty} hold,
 then the separation property \eqref{separation} holds at the initial time
 (i.e., its analogue is satisfied by $\fhi_0$). Hence, \eqref{separation}
 is fully compatible with \eqref{init:reg:2}. 
\eddos
\beos\label{rem:separ}
 The separation property \eqref{separation} is  extremely important for
 singular potentials like \eqref{Flog}.  Indeed, if \eqref{separation} holds, 
 then, the singularity of $F$ is no longer an obstacle 
 for the analysis as, actually, $\fhi$ is limited to 
 range in a closed subinterval of $(-1,1)$ where $F$ has controlled growth.
\eddos
\beos\label{rem:para}
 Given the parabolic nature of system~\SYS,
 most of our regularity results could be seen as smoothing properties of weak
 solutions (i.e., of solutions starting from ``energy regular'' initial data
 as those constructed in Theorem~\ref{THM:WEAK}), holding for strictly positive 
 times (provided that the required additional assumptions on
 coefficients, like for instance constant mobilities, hold). Of course,
 since uniqueness is not known to hold for weak solutions, we can assert that
 from energy regular initial data starts at least one weak solution that
 enjoys parabolic smoothing properties. However, we cannot exclude that there
 might exist other weak solutions which {\sl do not}\/ regularize in time.
\eddos 

\noindent%
Our last result is devoted to establishing uniqueness of solutions
in the case of constant
mobility functions. We prefer to formulate the result in a general version holding both
for $d=2$ and for $d=3$ though in a conditional way.
\begin{teor}[Uniqueness]\label{THM:UNIQ:2d}
 Suppose that assumptions\/ {\bf (A1)-(A4)} hold. Moreover, let $\mobm,\mobn,\beta \equiv 1$
 and $p=2$ in\/ \eqref{b:logi2}.
 Let us consider a couple of weak solutions $\{(\fhi_i, \mu_i, \sigma_i)\}_{i=1,2}$
 additionally satisfying 
 \begin{align}\label{cond:un1}
   & \fhi_1 \in L^2(0,T;W^{2,6}(\Omega)),\\
  \label{cond:un2}
   & \sigma_1 \in L^4(0,T;H),\\
  \label{cond:un3}
   & \sigma_2 \in L^4(0,T;L^6(\Omega)),
 \end{align} 
 associated to initial data $\{(\fhi_{0,i},\sigma_{0,i})\}_{i=1,2}$ fulfilling 
 \eqref{hp:fhi0}-\eqref{hp:sigma0} and \eqref{init:reg}.
 Let us also assume that either $h$ is a constant function, or 
 $F\in C^2(-1,1)$ and there hold the additional conditions
 \begin{equation}\label{cond:un4}
    F''(\fhi_1),F''(\fhi_2) \in L^2(0,T;H)
 \end{equation}
 as well as $\{(\fhi_{0,i},\sigma_{0,i})\}_{i=1,2}$ also fulfill \eqref{init:reg:2}.
 Then, $(\fhi_{1}, \mu_{1}, \sigma_{1}) \equiv (\fhi_2, \mu_2, \sigma_2)$ almost everywhere 
 in $Q$. 
\end{teor}
\noindent%
\beos\label{oss:quanduniq}
We notice that conditions \eqref{cond:un1}-\eqref{cond:un2} 
are verified under the regularity setting of Theorem~\ref{THM:WEAK:2d3d}. 
The main obstacle is represented by \eqref{cond:un3}, which holds
only under the more restrictive conditions in Theorem~\ref{THM:REG:1}. 
Finally, the validity of \eqref{cond:un4} is limited to the two-dimensional
case under assumption \eqref{F:growth} and in the regularity
setting of Theorem~\ref{THM:REG:2}. 
\eddos
\beos\label{oss:dipcon}
It is worth noticing that, under the assumptions of the above theorem, a continuous
dependence estimate also holds. For instance, in the case of $h$ constant, one has
\begin{align}\notag
  & \big\| \fhi_1-\fhi_2-((\fhi_1)\OO - (\fhi_2)\OO) \big\|_{L^\infty(0,T;V^*)}^2
   + \| (\fhi_1)\OO - (\fhi_2)\OO \|_{L^\infty(0,T)}^2
   + \| (\fhi_1)\OO - (\fhi_2)\OO \|_{L^\infty(0,T)}\\
 \no  
  & \quad\qquad + \big\| \sigma_1-\sigma_2-((\sigma_1)\OO - (\sigma_2)\OO) \big\|_{L^\infty(0,T;V^*)}^2
   + \| (\sigma_1)\OO - (\sigma_2)\OO \|_{L^\infty(0,T)}^2\\
 \no
  & \quad\qquad
  + \| \fhi_1-\fhi_2 \|_{L^2(0,T;V)}^2
  +\| \sigma_1-\sigma_2 \|_{L^2(0,T;H)}^2 \\
 \no
  & \quad \leq K \Big( \big\| \fhi_{0,1} - \fhi_{0,2}- ((\fhi_{0,1})\OO - (\fhi_{0,2})\OO) \big\|_{V^*}^2
  + | (\fhi_{0,1})\OO - (\fhi_{0,2})\OO |^2
  + | (\fhi_{0,1})\OO - (\fhi_{0,2})\OO |\\
 \label{cont:dep1}
  & \quad\qquad
  + \big\| \sigma_{0,1} - \sigma_{0,2}- ((\sigma_{0,1})\OO - (\sigma_{0,2})\OO) \big\|_{V^*}^2
  + | (\sigma_{0,1})\OO - (\sigma_{0,2})\OO |^2 \Big),
\end{align} 
for some $K$ depending only on the known data, including the 
norms in \eqref{cond:un1}-\eqref{cond:un2}. 
\eddos
\beos\label{oss:altuniq}
Aiming at reducing the technical burden, Theorem~\ref{THM:UNIQ:2d} is proved by
considering the two- and three-dimensional cases together. As a consequence,  
it is worth noticing that conditions \eqref{cond:un1}-\eqref{cond:un3}
are unlikely to be optimal, especially in dimension two where better inequalities hold,
and may be in fact replaced by other similar assumptions. For instance, it will
be noted in the proof that \eqref{cond:un3} might be replaced by 
\begin{equation*}
   \sigma_2 \in L^\infty(0,T;L^{3+\epsi}(\Omega)) 
    \quext{for some }\,\epsi>0. 
\end{equation*}
\eddos

%%%%%%%%%%%%%%%%%%%%%%%%%%%%%%%%%%%%%%%%%%%%%%%%%%%%%%%%%%%%%%%%%%%%%%%%%%%%%%%%%%%%%%%%%%%%%%%%%%%
%%%%%%%%%%%%%%%%%%%%%%%%%%%%%%%%%%%%%%%%%%%%%%%%%%%%%%%%%%%%%%%%%%%%%%%%%%%%%%%%%%%%%%%%%%%%%%%%%%%

\section{Well-posedness}
\label{sec:well}
This section is devoted to the proof of Theorem \ref{THM:WEAK}, which will be split
into several parts presented in separate subsections.

%%%%%%%%%%%%%%%%%%%%%%%%%%%%%%%%%%%%%%%%%%%%%%%%%%%%%%%%%%%%%%%%%%%%%%%%%%%%%%%%%%%%%%%%%%%%%%%%%%%

\subsection{Mass dynamics}

\label{sub:mass}

The main tool in the existence proof consists in the derivation of suitable 
a-priori estimates. For the sake of simplicity, these will be presented 
by working on a triplet $(\fhi,\mu,\sigma)$ solving the original 
system~\eqref{CH1b}-\eqref{nutrb} plus the initial and boundary conditions, 
without referring to any explicit approximation or regularization of it. 
In Subsection~\ref{subsec:appro} below we will propose a regularization
of the system and explain how the formal estimates derived here may
be adapted to the rigorous framework.

In this respect, it is worth observing from the very beginning that 
a crucial point stands in the fact that the coercivity of the energy functional
$\calF$ (cf.~\eqref{freeenergy}) is tied to the choice of a ``singular''
potential $F$. Hence, dealing with the original (i.e., non-regularized
system) and assuming in particular 
that the component $\fhi$ of the solution
satisfies the a-priori information
\begin{equation}\label{constr}
  - 1 \le \fhi(x,t) \le 1 
   \quext{for a.e.~}\,(x,t)\in Q
\end{equation}
represents a real simplification at this level. 
Indeed, let us recall the expression of the energy functional $\calF$,
namely (recall that $\eps=1$)
\begin{equation*}
%\label{defi:F}
  {\cal F}(\fhi, \sigma) = \frac 12 \int_\Omega |\nabla \fhi |^2 
   + \io F(\fhi)
   + \io \sigma (\ln \sigma - 1) 
   + \chi \io \sigma (1-\fhi).
\end{equation*}
Then, it is clear that, as far as \eqref{constr} holds, 
the last term is nonnegative due to the expected positive sign of the variable $\sigma$. 
As a consequence, $\calF$ turns out to be
coercive: it is easy to check that there exists a computable 
constant $C>0$ depending only on the known data such that
\begin{equation}\label{co:1a}
 \big( C + {\cal F}(\fhi, \sigma) \big)
   \ge \frac 12 \| \fhi \|_V^2 + \frac12 \| \sigma \ln \sigma \|_{{1}}.
\end{equation}
Notice that the above still holds when $\fhi\in L^\infty$, even if \eqref{constr}
is not known to hold. Indeed, the coupling term can then be controlled as follows
\begin{equation}\label{coerc:F}
  \chi \bigg| \io \sigma (1-\fhi) \bigg|
   \le \chi \| \sigma \|_{{1}} \big( 1 + \| \fhi \|_{{\infty}} \big)
   \le \frac12 \| \sigma \ln \sigma \|_{{1}} - c,   
\end{equation}
where the last $c$ also depends on the $L^\infty$-norm of $\fhi$.

On the other hand, if the singular potential $F$ is replaced by a function
of controlled growth, then the boundedness of $\fhi$ is lost, \eqref{coerc:F}
cannot be used, and consequently the energy functional loses its coercivity.
This is the main issue we will need to fix when detailing the 
approximation in Subsection~\ref{subsec:appro}.

That said, we first show that, under assumptions \eqref{hp:S}-\eqref{hp:h1} the spatial
mean of $\fhi$ is constrained to take values in the  physical interval $(-1,1)$ 
for every $t\ge 0$. Actually, testing \eqref{CH1} by $|\Omega|^{-1}$ and setting
for simplicity $y=\fhi\OO$, we deduce the ODE-like relation
\begin{equation*}
  y ' + m y = \frac1{|\Omega|} \io h(\fhi,\sigma),
\end{equation*}
whence, using \eqref{hp:h1}, we obtain the differential inequalities
\begin{equation*}
  - H \le y ' + m y \le H.
\end{equation*}
Consequently, it holds that, for every $t \in [0,T]$,
\begin{equation*}
  y(0) e^{-mt} + ( 1 - e^{-mt}) \Big( - \frac{H}{m} \Big)
   \le y(t) 
   \le y(0) e^{-mt} + ( 1 - e^{-mt}) \frac{H}{m}.
\end{equation*}
Using again \eqref{hp:h1} and recalling the last assumption in \eqref{hp:fhi0}, 
we then deduce that, for some $\delta > 0$ depending only on $\fhi_0$, $H$ and $m$, there holds
\begin{equation}\label{medie}
  \big|(\fhi(t))\OO\big|\le 1 - \delta \quext{for every }\,t\in[0,T],
\end{equation}
which entails that the total mass of $\fhi$ is prevented to reach the critical values $\pm1$.
Of course the same property \eqref{medie} holds also when one takes $S\equiv 0$
because the spatial mean of $\fhi$ is conserved in that case.

%%%%%%%%%%%%%%%%%%%%%%%%%%%%%%%%%%%%%%%%%%%%%%%%%%%%%%%%%%%%%%%%%%%%%%%%%%%%%%%%%%

\subsection{Formal energy estimate}
\label{subsec:ene}

First of all, we observe that, using assumption \eqref{hp:sigma0} and applying 
a standard minimum principle argument, there follows that $\sigma\ge 0$ almost
everywhere in~$Q$.
Then, testing \eqref{CH1b} by $\mu$, \eqref{CH2b} by $\fhi_t$, and taking
the difference, we infer
\begin{equation}\label{co:11}
  \ddt \calE(\fhi) - \chi\io \sigma \fhi_t + \io \mobm(\fhi,\sigma) | \nabla \mu |^2 
   = \io S(\fhi,\sigma) \mu,
\end{equation}
where we recall that $\calE$ denotes the standard Ginzburg--Landau
energy introduced in \eqref{freeenergy} (with $\eps=1$).
Next, testing \eqref{nutrb2} by $\ln \sigma + \chi (1-\fhi)$, 
and setting for the sake of simplicity $L(\sigma):= \sigma (\ln \sigma - 1)$, we obtain 
\begin{align}\no
  & \ddt \io L(\sigma)
   + \chi \io \sigma_t (1-\fhi)
   + \io \sigma \mobn(\fhi,\sigma) \big| \nabla ( \ln \sigma + \chi (1-\fhi) ) \big|^2 \\
 \label{co:12}   
  & \mbox{}~~~~~  
   = \io \beta(\fhi) (\kappa_0\sigma - \kappa_\infty \sigma^p)   
   \big( \ln \sigma + \chi (1-\fhi) \big).
\end{align}
Adding \eqref{co:11} with \eqref{co:12} and rearranging, 
using also \eqref{mob:bou}, we deduce
\begin{align}\no
  & \ddt \bigg[ \calE(\fhi) 
  + \io \big( L(\sigma) + \chi \sigma (1-\fhi) \big) \bigg]\\ 
  & \notag \qquad
   + m_0 \io \sigma \big| \nabla ( \ln \sigma {+ \chi (1-\fhi)} ) \big|^2 
   + m_0 \| \nabla \mu \|^2 
   + \kappa_\infty \io \beta(\fhi) \sigma^p \ln \sigma \\ 
 \label{co:13}
  & \mbox{}~~~~~ \leq \io S(\fhi,\sigma) \mu
   + \kappa_0 \io \beta (\fhi)\sigma \ln \sigma 
     + \chi \io \beta(\fhi) (\kappa_0\sigma - \kappa_\infty \sigma^p)   (1-\fhi).
\end{align}
To control the first term on the \rhs, we observe that, replacing the 
expression of $\mu$ given by equation \eqref{CH2b} and using~{\bf (A2)} 
along with \eqref{constr} and the Poincar\'e--Wirtinger inequality \eqref{poincare}, there follows
\begin{equation}\label{co:14}
  \io S(\fhi,\sigma) \mu 
   = \io S(\fhi,\sigma) ( \mu - \mu\OO ) 
    + \mu\OO \io S(\fhi,\sigma)
    \le (m + H) c_\Omega \| \nabla \mu \|
    + | \Omega | (m + H) | \mu\OO |,
\end{equation}
where $c\OO>0$ is a Poincar\'e constant.
Now, integrating \eqref{CH2b} over $\Omega$ and recalling that $\sigma \ge 0$
almost everywhere in $Q$ and \eqref{BC}, we have
\begin{equation}\label{co:15}
  | \Omega |  | \mu\OO | 
   \le \| f(\fhi) \|_{1}
   + \chi \io \sigma.
\end{equation}
Replacing \eqref{co:14} and \eqref{co:15} into
\eqref{co:13}, using \eqref{beta:bou} and \eqref{beta:bou2}
with the fact $1-\fhi\ge 0$ (which is, in turn,
a consequence of \eqref{constr}),
it is not difficult to obtain
\begin{align}\no
  & \ddt {\cal F}(\fhi, \sigma)
   + m_0 \io \sigma \big| \nabla ( \ln \sigma {+ \chi (1-\fhi)} ) \big|^2
   + m_0 \| \nabla \mu \|^2
   + \kappa_\infty b_0 \io ( \sigma^p \ln \sigma + 2) \\
 \no
   & \mbox{}~~~~~ 
   \le c 
    + c \io \sigma 
    + c \| \nabla \mu \| 
    + (m + H) \| f(\fhi) \|_{1}\\
 \no
   & \mbox{}~~~~~~~~~~  
    + \kappa_0 \io \beta(\fhi) \sigma \ln \sigma 
    + \chi \io \beta(\fhi) (\kappa_0\sigma - \kappa_\infty \sigma^p)   (1-\fhi)\\
 \label{co:19b}
   & \mbox{}~~~~~ \le 
    c 
    + c \io \sigma 
    + \frac{m_0}4 \| \nabla \mu \|^2 
    + (m + H) \| f(\fhi) \|_{{1}} 
    + \kappa_0 B \io | \sigma \ln \sigma |
     + \chi \kappa_0 B \io \sigma (1-\fhi),
\end{align}
where we recall that $\cal F$ was defined in \eqref{freeenergy}
(with $\eps=1$).
In order to control the norm of $f(\fhi)$ on the \rhs, we 
test \eqref{CH2b} by $\fhi - \fhi\OO$ to get
\begin{equation}\label{co:16}
   \io f(\fhi) (\fhi - \fhi\OO)
    + \| \nabla \fhi \|^2
    = \io \mu(\fhi-\fhi\OO) 
    + \chi \io \sigma(\fhi-\fhi\OO).
\end{equation}
Using the mass property \eqref{medie} and proceeding similarly as in \cite{MZ} 
(see also, e.g., \cite{FG,KNP}) we then deduce 
that there exist two constants $\alpha>0$ (small) and $c>0$ (large),
both depending on the constant $\delta$ in \eqref{medie}, such that
\begin{equation}\label{co:17}
   \io f(\fhi) (\fhi - \fhi\OO)
   \ge \alpha\| f(\fhi) \|_{{1}} - c.
\end{equation}
Hence, recalling \eqref{constr} and operating 
straightforward manipulations in \eqref{co:16} leads to
\begin{align}\no
  \alpha\| f(\fhi) \|_{1} 
   & \le c \| \nabla \mu \| \| \nabla \fhi \|
    + 2\chi  \io \sigma 
    + c\\
 \label{co:18}   
   & \le c_{\eps} \| \nabla \fhi \|^2
    + \eps \| \nabla \mu \|^2
    + 2\chi  \io \sigma 
    + c,
\end{align}
where $\eps >0$ can be taken arbitrarily small, in a way that will be 
specified later on, and $c_\eps>0$ depends on the choice of $\eps$.

Next, we multiply \eqref{co:18} by $M>0$ large enough such that 
$M\alpha\ge m + H + 1$; 
then, we choose $\eps$ small enough such that 
$M\eps\le m_0/4$. Finally, we add the result of this
operation to \eqref{co:19b} deducing that
\begin{align*}\no
  & \ddt 
    {\cal F}(\fhi, \sigma)
   + m_0 \io \sigma \big| \nabla ( \ln \sigma {+ \chi (1-\fhi)}) \big|^2
   + \frac{m_0}2 \| \nabla \mu \|^2
   + {\alpha}M \| f(\fhi) \|_{1}
      + \kappa_\infty b_0 \io (\sigma^p \ln \sigma + 2) \\
% \label{co:19}
   & \mbox{}~~~~~ \le c \| \nabla \fhi \|^2
   + c
   + c \io \sigma
   + \kappa_0 B \io | \sigma \ln \sigma |
   + \chi \kappa_0 B \io \sigma (1-\fhi).
\end{align*}
In order to get an estimate from the above relation, we observe that,
since $p>1$, we have
\begin{equation*}
%\label{co:1b}
  c \io \sigma
    + \kappa_0 B \io | \sigma \ln \sigma |
    + \chi \kappa_0 B \io \sigma (1-\fhi)
   \le c + \frac{\kappa_\infty b_0}2 \io (\sigma^p \ln \sigma + 2).
\end{equation*}
Then, we may set 
\begin{equation*}
%\label{co:1c}
  \calV :=  C + 
   {\cal F}(\fhi, \sigma),
\end{equation*}
where $C>0$ is chosen such that the coercivity property \eqref{co:1a}
holds. As noted above, here we are using the constraint
\eqref{constr} in an essential way. With these choices, 
we arrive at the differential inequality 
\begin{equation}\label{en:es}
  \ddt\calV   
   + m_0 \io \sigma \big| \nabla ( \ln \sigma {+ \chi (1-\fhi)} ) \big|^2
   + \frac{m_0}2 \| \nabla \mu \|^2
   + {\alpha}M \| f(\fhi) \|_{1}
   + \frac{\kappa_\infty b_0}2 \io( \sigma^p \ln \sigma + 2)
   \le c \calV.
\end{equation}
Noting that $\calV(0)<\infty$ thanks to \eqref{hp:fhi0}-\eqref{hp:sigma0}, 
we can then apply Gr\"onwall's lemma to deduce the following 
a-priori bounds:
\begin{align} \notag
  & \norma{\fhi}_{\LIV}
   + \| \sigma \ln \sigma \|_{L^\infty(0,T;L^1(\Omega))} 
   + \norma{ F(\fhi) }_{\L\infty {\Lx1}}
   + \norma{ \sigma^{1/2} \nabla ( \ln \sigma {+ \chi (1-\fhi)} ) }_{\L2 H}\\ 
 \label{en:es:formal}
  & \quad 
   + \norma{\nabla \mu}_{\L2 {H}}
   + \norma{f(\fhi)}_{\L1 {\Lx1}}
   + \norma{\sigma^p \ln \sigma}_{\L1 {\Lx1}}
  \le c.
\end{align}
Here and below, it is intended that the constant $c>0$ on the \rhs,
whose explicit value may vary on occurrence, depends only 
on the known data of the problem, including the initial data, but is independent
of any hypothetical approximation parameter.

Next, going back to \eqref{co:18}, squaring its first row, and integrating in
time, using also the information resulting from \eqref{en:es:formal},
we infer
\begin{equation}\label{st:11}
  \norma{f(\fhi)}_{L^2(0,T;L^1(\Omega))} \le c.
\end{equation}
Then, we go back to \eqref{co:15}: squaring and integrating in time, 
using \eqref{en:es:formal} and \eqref{st:11} lead us to $\norma{\mu_{\Omega}}_{L^2(0,T)}\leq c$, 
and applying once more the Poincar\'e--Wirtinger inequality, we arrive at 
\begin{equation*}
%\label{st:12}
  \norma{\mu}_{L^2(0,T;V)} \le c.
\end{equation*}
Next, comparing terms in \eqref{CH1b} and using in particular assumptions {\bf (A2)}
and \eqref{mob:bou}, it is not difficult to deduce 
\begin{equation}\label{st:13}
  \norma{\fhi_t}_{L^2(0,T;V^*)} \le c.
\end{equation}
Finally, we observe that \eqref{CH2b}, complemented with the no-flux boundary condition,
can be interpreted as a family of time-dependent elliptic problems
with maximal monotone perturbations of the form
\begin{equation}\label{CH2c}
  - \Delta \fhi + F_1'(\fhi) = \mu + \lambda\fhi + \chi\sigma.
\end{equation}
Here, we observe that, since we assumed $p\le 2$, the maximal summability
available for the \rhs\ is exactly the $L^p$-one. Hence, applying standard tools
(which basically correspond to testing \eqref{CH2c} 
by $|F_1'(\fhi)|^{p-1}\sign(F_1'(\fhi))$ and exploiting the monotonicity of $F_1'$),
we deduce the additional estimates
\begin{equation}\label{st:14}
  \norma{F_1'(\fhi)}_{L^p(0,T;L^p(\Omega))} 
  + \norma{\Delta\fhi}_{L^p(0,T;L^p(\Omega))} 
  \le c.
\end{equation}
Suppose now that $ b$ has a true logistic growth, that is \eqref{b:logi2} 
is fulfilled with $p=2$. Then, repeating the energy estimate in a decoupled fashion, 
i.e., by testing \eqref{CH2b} by $-\Delta \fhi$ and \eqref{nutrb}
by $\ln\sigma$ it is not difficult to arrive at
\begin{align*}\no
  & \ddt\io \sigma (\ln \sigma - 1)
   + 4 m_0 \| \nabla \sigma^{1/2} \|^2
   + b_0 \kappa_\infty\io \sigma^2 \ln \sigma
   + \| \Delta \fhi \|^2
   + \io f_1'(\fhi)|\nabla \fhi|^2 \\
 \no
  & \mbox{}~~~~~ \le -2\chi \io \sigma \Delta \fhi 
   + \kappa_0\io  \beta(\fhi) \sigma\ln\sigma
   + \io \nabla\mu\cdot\nabla \fhi
   - \lambda \io  |\nabla \fhi|^2\\
 \no
  & \mbox{}~~~~~ \le \frac12 \| \Delta \fhi \|^2
   + c \io \sigma^2 
   + c
   + c \| \nabla\mu \|\\
% \label{co:51old}
  & \mbox{}~~~~~ \le \frac12 \| \Delta \fhi \|^2
   + \frac {b_0 \kappa_\infty}2 \io \sigma^2 \ln \sigma 
   + c.
%   (\kappa_\infty)
%   + c \| \nabla\mu \|.
\end{align*}
Hence, the cross-diffusion terms are decoupled, both in dimension two
and in dimension three, and we obtain the additional estimate
\begin{align}
\label{st:62old}
  & \| \sigma^{1/2} \|_{L^2(0,T;V)} \le c.
\end{align}

%%%%%%%%%%%%%%%%%%%%%%%%%%%%%%%%%%%%%%%%%%%%%%%%%%%%%%%%%%%%%%%%%%%%%%%%%%%%%%%%%%

\subsection{Approximation scheme}
\label{subsec:appro}

In this part, we outline a possible regularization scheme for system~\SYS. 
Usually, in Cahn--Hilliard-based models, approximation is provided by
smoothing out the singular term (here represented by $F_1'$) and replacing
it with a Lipschitz continuous function. In this way, at least locally in time,
existence of approximate solutions may be proved for instance by 
using a Faedo--Galerkin or time discretization scheme.
Here, however, a further difficulty arises because the coercivity
of the free energy~\eqref{freeenergy} is tied to the presence of the singular
potential $F_1$. In other words, if $F_1$ is smoothed out, it is
also necessary to intervene on the coupling term $\chi\sigma (1-\fhi)$ by suitably
truncating it; indeed, approximating $F_1$, the property $|\fhi|\le 1$ is 
lost and the coupling term becomes supercritical if it is not
smoothed out. This is why, similarly, e.g., to \cite{HH}, 
we propose a regularized scheme where also $\sigma$ is properly 
truncated. Namely, we consider the system
\begin{alignat}{2}\label{CH1n}
  & \fhi_t -\dive\big(\mobm(\fhi,\sigma) \nabla \mu\big) = S(\fhi,\sigma) \quad && \text{in $Q$,} \\
  \label{CH2n}
  & \mu = - \Delta \fhi + F_n'(\fhi) - \lambda\fhi - \chi T_n(\sigma) \quad && \text{in $Q$,}\\
 \label{nutrn}
  & T_n(\sigma)_t - \dive \big(\mobn (\fhi,\sigma)  \nabla \sigma \big)
  - \chi \dive \big(\sigma \mobn (\fhi,\sigma)  \nabla(1-\fhi) \big) 
    = \beta(\fhi) (\kappa_0\sigma - \kappa_\infty \sigma^p) \quad && \text{in $Q$,}
\end{alignat}
with $n\in\NN$ denoting the regularization parameter, 
intended to go to infinity in the limit. 
As for the boundary and initial conditions,
we consider the same as in \eqref{BC} and \eqref{init}.
In this part we will
avoid employing the subscript $n$ to denote 
the approximate solution for the sake of notational clarity.
Here, we assume the following properties. First of all, $\{ F_n\}$, 
with $F_n:\RR \to \RR$, 
is a family of convex and regular functions such that, as $n\to \infty$, $F_n$ 
tends to $F_1$ in the sense of {Mosco}. We refer to \cite[Chap.~3]{At}
for the necessary background in convex analysis; we just observe that 
a simple condition ensuring this property holds when, for every fixed
$r\in\RR$, $F_n(r)$ is increasingly monotone with respect to $n$ 
and converges to the limit value $F_1(r)$ (which is intended to be $+\infty$
as far $|r|$ is larger than $1$, cf.~{\bf (A1)}).
We also assume the normalization property
\begin{equation}\label{ap:11}
   F_n'(r)\sign r \ge n^3 (|r|-1) \quext{for every }\, |r|\ge 1. 
\end{equation}
Indeed, it is apparent that, for every potential $F_1$ compatible with
assumption~{\bf (A1)}, an approximation $F_n$ satisfying the above conditions
can be constructed by standard methods. For instance one could take the 
Yosida approximation (see \cite{Ba,Br}) of $F_1$ of order $n^{-1}$ and 
add to it $n^3 (|r|-1)_+\sign r$. 

Concerning the truncation operator $T_n$ we assume the following properties:
\begin{align}\label{Tn11}
  & T_n \in C^{1,1}(\RR); \qquad T_n(r)=r \quext{for every }\,r\le n,\\  
 \label{Tn12}
  & T_n~~\text{is strictly monotone and concave},\\
 \label{Tn13}
  & T_n(r) < n+1 \quext{for every }\,r\in\RR,
   \qquad \lim_{r\to\infty} T_n(r) = n+1. 
\end{align}
Explicit forms of $T_n$ can be also constructed very easily. 

We also need to introduce its inverse function, that is, $\gamma_n:=T_n^{-1}$. Then,
$\gamma_n\in C^1((-\infty,n+1);\RR)$ and $\gamma_n$ can also be interpreted 
as a maximal monotone graph in $\RR\times\RR$ so to apply the usual machinery
of maximal monotone operator theory. Then, we also set 
\begin{equation*}
%\label{ap:31}
   s:= T_n(\sigma) \qquext{so that }\,\sigma=\gamma_n(s)
\end{equation*}
and equation \eqref{nutrn} can be consequently restated in terms 
of the new variable $s$ in the following equivalent way:
\begin{equation}\label{nutrns}
  s_t - \dive \big(\mobn (\fhi,\gamma_n(s))  \nabla \gamma_n(s) \big)
  - \chi \dive \big(\gamma_n(s) \mobn (\fhi,\gamma_n(s))  \nabla(1-\fhi) \big) 
    = \beta(\fhi) (\kappa_0\gamma_n(s) - \kappa_\infty \gamma_n(s)^p).
\end{equation}
Of course, also relations \eqref{CH1n}-\eqref{CH2n} could be equivalently
reformulated in terms of $s$. 
Besides, by using the variable $s$, the initial conditions
may be expressed as follows:
\begin{equation}\label{initns}
  s|_{t=0} = s_0 = T_n(\sigma_0), \qquad 
  \fhi|_{t=0} = \fhi_0.
\end{equation}
Dealing with the regularized system \eqref{CH1n}-\eqref{nutrn},
complemented with the initial conditions \eqref{initns} and with the no-flux
boundary conditions, may still be nontrivial. 
Indeed, equation \eqref{nutrns} contains the {\sl singular}\/
function $\gamma_n$. A strategy {that} could be used in order to obtain 
at leas local in time existence can be sketched as follows:
\begin{itemize}
 \item[(A)] Smoothing out the function $\gamma_n$ (for instance by replacing in with its 
 Yosida regularization $\gamma_{n,\epsilon}$ for a
 regularization parameter $\epsilon$ intended to go 
 to zero in the limit); one may also need to add further regularizing terms
 to get better properties of approximating solutions;
 \item[(B)] Proving local in time existence to the obtained system {through} the 
 Faedo--Galerkin method;
 \item[(C)] Getting a-priori estimates uniform with respect to $\epsilon$ and,
 exploiting these, passing to the limit with respect to the approximation parameter 
 $\epsilon$ so to obtain a solution to \eqref{CH1n}-\eqref{nutrn} with the initial
 and boundary conditions. 
\end{itemize}
As said, even if the above strategy (A)-(C) may be not trivial, we believe that
the main difficulties are just of technical nature. Indeed, equations of the form
\begin{equation}\label{stefan}
  s_t - \Delta( \gamma(s) ) = f,
\end{equation}
with $\gamma$ maximal monotone graph (possibly, as here, of singular nature),
have been extensively studied in the literature and the proposed strategy in order 
to get local existence (i.e., smoothing $\gamma$, discretizing by Faedo--Galerkin,
then going back to the original $\gamma$) is very well established. Of course,
in our setting, equation \eqref{nutrns} is more complicated than \eqref{stefan}
and we also have the additional difficulties resulting from the coupling
with the CH system. Nevertheless, to reduce technical details,
we assume to have accomplished the above strategy and we just focus on what we
believe to be the main difficulty to get an existence theorem,
i.e., the passage to the limit $n\to\infty$.

For this purpose, we first have to reproduce the energy estimate by 
working on the regularized system~\eqref{CH1n}-\eqref{nutrn}.
Of course, we can use either the original nutrient variable $\sigma$
or the transformed (truncated) variable $s=T_n(\sigma)$ since the formulations
in terms of $\sigma$ and $s$ are equivalent at this level. 
Notice also that the proposed approximation is devised as to preserve the minimum
principle property; hence we can freely assume $s$ and $\sigma$
to be nonnegative.

That said, using the variable $\sigma$ so to get an estimate more similar
to that obtained in the previous section, we have to test
\eqref{CH1n} by $\mu$, \eqref{CH2n} by $\fhi_t$, and
\eqref{nutrn} by $\ln\sigma + \chi(1-\fhi)$. Then, by proceeding as before,
we arrive at the analogue of \eqref{co:13}, which takes now the form
\begin{align}\no
  & \ddt \calF_n(\fhi,\sigma)
   + m_0 \io \sigma \big| \nabla ( \ln \sigma {+ \chi (1-\fhi)} ) \big|^2 
   + m_0 \| \nabla \mu \|^2 
   + \kappa_\infty \io \beta(\fhi) \sigma^p \ln \sigma\\ 
 \label{co:13n}
  & \mbox{}~~~~~ \leq \io S(\fhi,\sigma) \mu
   + \kappa_0 \io \beta(\fhi)\sigma \ln \sigma 
   + \kappa_0 \chi \io \beta(\fhi)\sigma ( 1 -\fhi)
   + \kappa_\infty \chi \io \beta(\fhi) \sigma^p (\fhi - 1),
\end{align}
and where the approximated energy takes the expression
\begin{equation*}
%\label{energyn}
  {\cal F}_n(\fhi, \sigma) = 
    \io \Big( \frac12 |\nabla \fhi |^2 + F_n(\fhi) 
   - \frac\lambda2 \fhi^2
   + L_n(\sigma)
   + \chi T_n(\sigma) (1-\fhi) \Big).
\end{equation*}
The function $L_n$ is defined as follows:
\begin{equation*}
%\label{defiLn}
  L_n(\sigma) 
   = \int_0^\sigma T_n'(r) \ln r \,\dir,
\end{equation*}
so that we have in particular
\begin{equation}\label{propLn1}
  L_n(\sigma) = \sigma (\ln\sigma - 1)~~\text{for }\,\sigma \le n
   \qquext{and }\,L_n(\sigma) \ge n (\ln n - 1)~~\text{for }\,\sigma >n.
\end{equation}
Let us now verify that the coupling term $\chi T_n(\sigma) (1-\fhi)$ can be controlled,
uniformly in $n$, by using the other integrands. We actually notice that 
\begin{align*}\no
   \chi \io T_n(\sigma) (1-\fhi)
  & = \chi \io T_n(\sigma) - \chi \io T_n(\sigma) \fhi\\
 \no
   & =  \chi \io T_n(\sigma)
   - \chi \int_{\{|\fhi|\le 2\}} T_n(\sigma) \fhi
    - \chi \int_{\{|\fhi|> 2\}} T_n(\sigma) \fhi\\
 \no
   & \ge 
    \chi \io T_n(\sigma)
    -2 \chi \int_{\{|\fhi|\le 2\}} T_n(\sigma)
     - (n+1) \chi \int_{\{{|\fhi| > 2}\}} \fhi\\
% \label{ap:13}
   & \ge - \chi \io T_n(\sigma)
     - 2 (n+1) \chi \int_{\{{|\fhi| > 2}\}} ( \fhi - 1).
\end{align*}
To control the \rhs, one can first verify that 
\begin{equation}\label{ap:14}
  \chi T_n(\sigma) \le \frac12 L_n(\sigma) + c,
\end{equation}
for every $\sigma>0$ and a suitable constant $c\ge 0$ independent of $n$.
Analogously, owing to \eqref{ap:11}, it is clear that,
for $c\ge 0$ as above, we have
\begin{equation}\label{ap:15}
   F_n(r) \ge \frac{n^3}2 (|r|-1)^2 - c \quext{for every }\,|r|\ge 1.
\end{equation}
Consequently, thanks also to Young's inequality, we have in particular
\begin{equation*}
  2 (n+1) \chi ( \fhi - 1) \le \frac14 F_n(\fhi) + c 
   \quext{for every }\,\fhi\ge 2. 
\end{equation*}
Based on the above considerations, and noting also that, by \eqref{ap:11},
\begin{equation*}
%\label{ap:17}
  \frac\lambda2 \fhi^2 \le \frac14 F_n(\fhi) + c,
\end{equation*}
again for $c\ge 0$ independent of $n$,
we conclude that there exists a constant 
$C$ independent of $n$ such that 
\begin{equation}\label{propenn}
  {\cal F}_n(\fhi, \sigma) 
    \ge \io \Big({\frac 12}|\nabla \fhi |^2 + \frac12 F_n(\fhi) + \frac12 L_n(\sigma) - C\Big).
\end{equation}
Namely, coercivity of the energy is preserved at the approximate level. 

Then, in order to deduce from relation \eqref{co:13n} an analogue of the 
energy estimate \eqref{en:es}, we need to {check} that we can still control the 
\rhs. To this aim, we start observing that the first integral can be managed 
similarly with \eqref{co:14}-\eqref{co:17}. We only notice that, when the 
analogue of \eqref{co:16} is performed, we can no longer use the uniform 
boundedness of $\fhi$. On the other hand, the mean property \eqref{medie} is preserved
also in the approximation. 
Then, the contribution corresponding to the last integral in \eqref{co:16}
is now managed as follows:
\begin{align*}\no
  \chi \io T_n(\sigma) (\fhi - \fhi\OO) 
   & = \chi \int_{ \{|\fhi|\le3/2\} } T_n(\sigma) ( \fhi - \fhi\OO )
   + \chi \int_{ \{|\fhi|> 3/2\} } T_n(\sigma) ( \fhi - \fhi\OO )\\
 \no
 & \le c \int_{\{|\fhi|\le3/2\}} | T_n(\sigma) |
  + c (n + 1) \int_{\{|\fhi|>3/2\}} ( |\fhi| - 1 )\\
% \label{ap:61}
 & \le c \io | L_n(\sigma) | 
  + c + c \io F_n(\fhi),
\end{align*} 
where we used \eqref{ap:14}, \eqref{ap:15} and the control \eqref{medie} on the spatial average
of $\fhi$, which is not affected by the approximation. 
Then we notice that the last integral
on the \rhs\ can be controlled by Gr\"onwall's lemma. In order to estimate
the remaining terms in \eqref{co:13n},
we first observe that, using {\bf (A3)} and in
particular \eqref{beta:bou3},
\begin{align*}\no
  & \kappa_\infty \chi \io \beta(\fhi) \sigma^p (\fhi - 1)
   = \kappa_\infty \chi \int_{\{|\fhi|\le2\}} \beta(\fhi) \sigma^p (\fhi - 1)\\
% \label{ap:62}
  & \qquad = \kappa_\infty \chi \int_{\{|\fhi|\le3/2\}} \beta(\fhi) \sigma^p (\fhi - 1)
   + \kappa_\infty \chi \int_{\{3/2<|\fhi|\le2\}} \beta(\fhi) \sigma^p (\fhi - 1)
   =: \II_1 + \II_2.
\end{align*} 
Now, using \eqref{beta:bou} with a generalized form of Young's inequality,
we have
\begin{equation}\label{ap:63}
  \II_1 \le c \int_{\{|\fhi|\le3/2\}} \sigma^p
   \le c + \frac{\kappa_\infty b_0}4 \int_{\{|\fhi|\le3/2\}} |\sigma^p \ln \sigma|.
\end{equation}
Similarly, we have
\begin{equation}\label{ap:64}
 \II_2 \le c \int_{\{3/2<|\fhi|\le2\}} \beta(\fhi) \sigma^p
   \le c + \frac{\kappa_\infty}4 \int_{\{3/2<|\fhi|\le2\}} 
   \beta(\fhi) |\sigma^p \ln \sigma|.
\end{equation}
The integrals on the \rhs s of \eqref{ap:63} and \eqref{ap:64} can be 
estimated by noting that the last term
on the \lhs\ of \eqref{co:13n} gives
\begin{equation}\label{ap:63b}
  \kappa_\infty \io \beta(\fhi) \sigma^p \ln \sigma
   \ge \kappa_\infty b_0 \int_{\{|\fhi|\le3/2\}} | \sigma^p \ln \sigma | 
    + \kappa_\infty \int_{\{3/2<|\fhi|\le 2\}} \beta(\fhi) | \sigma^p \ln \sigma | - c,
\end{equation}
as a straightforward check shows. The remaining two integrals on the \rhs\ 
of \eqref{co:13n} have a slower growth with respect to $\sigma$; hence they
can be controlled in a similar but in fact easier way. 

As a consequence, it is not difficult to obtain from \eqref{co:13n} the following
inequality:
\begin{align*}\no
  & \ddt \calF_n(\fhi,\sigma)
   + m_0 \io \sigma \big| \nabla ( \ln \sigma {+ \chi (1-\fhi)} ) \big|^2 
   +\frac{ m_0}{2} \| \nabla \mu \|^2 
   + \frac{\kappa_\infty b_0}2 \int_{\{|\fhi|\le3/2\}} | \sigma^p \ln \sigma | \\
% \label{co:13n2}
  & \mbox{}~~~~~ 
   + \frac{\kappa_\infty}2 \int_{\{3/2<|\fhi|\le 2\}} \beta(\fhi) | \sigma^p \ln \sigma |
  \le c \| \nabla \fhi \|^2
   + c 
   + c \io F_n(\fhi)
   + c \io | L_n(\sigma) |.
\end{align*}
Estimating the last term similarly with \eqref{ap:63} we end up with
\begin{align}\no
  & \ddt \big( C_0 + \calF_n(\fhi,\sigma) \big)
   + m_0 \io \sigma \big| \nabla ( \ln \sigma {+ \chi (1-\fhi)} ) \big|^2 
   + \frac{m_0}{2} \| \nabla \mu \|^2 
   + \frac{\kappa_\infty b_0}4 \int_{\{|\fhi|\le3/2\}} | \sigma^p \ln \sigma | \\
 \label{co:13n3}
  & \mbox{}~~~~~ 
   + \frac{\kappa_\infty}2 \int_{\{3/2<|\fhi|\le 2\}} \beta(\fhi) | \sigma^p \ln \sigma |
  \le c \| \nabla \fhi \|^2
   + c 
   + c \io F_n(\fhi)
   + c \io | L_n(\sigma) |,
\end{align}
where $C_0>0$ is large enough so that $C_0 + \calF_n$ is coercive uniformly with respect 
to $n$ (cf.~\eqref{propenn}). 
Then, we can apply Gr\"onwall's lemma to the above relation
so to deduce the following bounds which are independent of the approximation parameter $n$:
\begin{align}\label{st:n11}
  & \| \fhi \|_{\LIV} \le c,\\
 \label{st:n12}
  & \| \nabla \mu \|_{\LDH} \le c,\\
 \label{st:n13}
  & \| F_n(\fhi) \|_{L^\infty(0,T;L^1(\Omega))} 
   + \| L_n(\sigma) \|_{L^\infty(0,T;L^1(\Omega))} \le c,\\
 \label{st:n13b}
  & \| \sigma^{1/2} \nabla ( \ln \sigma {+ \chi (1-\fhi)} ) \|_{\LDH} \le c.
\end{align}
As in the previous part, the uniform control 
\eqref{medie} of the spatial average of $\fhi$ permits us to 
improve \eqref{st:n12} leading to
\begin{equation}\label{st:n14}
  \| \mu \|_{\LDV} \le c.
\end{equation}
Now, from the first of \eqref{st:n13} and \eqref{ap:15} it is not difficult to deduce that
\begin{equation*}
%\label{st:n15}
  \big| \{ |\fhi(\cdot,t)| \ge 3/2 \} \big| \le c n^{-3}
   \quext{for a.e.~}\,t\in(0,T).
\end{equation*}
As a consequence, we have 
\begin{align*}\no
  \itt \io | \sigma^p \ln \sigma|
   & = \itt \int_{\{|\fhi(\cdot,t)|\le 3/2\}} | \sigma^p \ln \sigma|
     + \itt \int_{\{|\fhi(\cdot,t)|> 3/2\}} | \sigma^p \ln \sigma|\\
% \label{st:n16}
   & \le c + c n^p \ln n \big| \{|\fhi(\cdot,t)|> 3/2\} \big| 
     \le c + c n^{p-3} \ln n \le c,
\end{align*}
where we used the control \eqref{co:13n3} and the fact $p\le 2$.
Hence, we end up with 
\begin{equation}\label{st:n17}
  \| \sigma^p \ln \sigma \|_{L^1(0,T;L^1(\Omega))} \le c.
\end{equation}
A similar procedure, combined with the second bound in \eqref{st:n13},
permits us to deduce 
\begin{equation}\label{st:n18}
  \| \sigma \ln \sigma \|_{L^\infty(0,T;L^1(\Omega))} \le c.
\end{equation}
Finally, we notice that the analogues of \eqref{st:13},  \eqref{st:14}, and \eqref{st:62old}
can be obtained reasoning as in the previous section.
%

%%%%%%%%%%%%%%%%%%%%%%%%%%%%%%%%%%%%%%%%%%%%%%%%%%%%%%%%%%%%%%%%%%%%%%%%%%%%%%%%%%%%%%%%%%%%%

\subsection{Passage to the limit} 
\label{subsec:lim}

In this part, we assume to have a sequence $\{(\fhi_n,\mu_n,\sigma_n)\}_n$ of solutions
complying with the a-priori estimates uniformly with respect to the parameter $n$.
Such a sequence may be an outcome of the ``strategy'' (A)-(C) outlined before. 
In particular, we will assume $(\fhi_n,\mu_n,\sigma_n)$ to solve, at least locally
in time, system \eqref{CH1n}-\eqref{nutrn} 
complemented with homogeneous Neumann boundary conditions and suitable initial conditions.
Moreover, from now on, the dependence of
the approximate solution on the parameter $n$ is stressed.

Moreover, since
the estimates derived in the previous part are uniform in time, by standard 
extension arguments the solution obtained in the limit will acquire a global in
time character. For this reason, and for the sake of simplicity too, we will directly
assume to have a global solution at the approximate level.
 
That said, we observe that the approximated version of
estimates \eqref{st:13} and \eqref{st:14}, 
\eqref{st:n11}-\eqref{st:n14}, \eqref{st:n17},
with standard weak and weak star compactness results, imply that 
there exist limit functions $\fhi, \sigma, \mu,$ and $\xi$ such that, as $n\to\infty$,
\begin{alignat}{2}\label{lim:11}
  & \fhi_n \to \fhi \quad && \text{weakly-star in }\,
    \H1 \Vp \cap \L\infty V\cap L^p(0,T;W^{2,p}(\Omega)),\\
 \label{lim:12}
  & \sigma_n  \to \sigma \quad 
  && \text{weakly in } \L{p} {\Lx p},\\
 \label{lim:13}
  & \mu_n \to \mu \quad  && \text{weakly in }\, \L2 V,\\
 \label{lim:14}
  & F_{n}'(\fhi_n) \to \xi \quad  && \text{weakly in }\,\L p {\Lx p}.
\end{alignat}

The above convergence relations, as well as the ones that follow, are intended
to hold up to the extraction of non-relabelled subsequences of $n\to\infty$.
Since \eqref{Tn11}-\eqref{Tn13} imply in particular
$|T_n'(r)|\le 1$ for every $n\in\NN$ and $r\in \RR$, we also have 
\begin{equation*}
%\label{lim:12b}
  s_n  \to s \quad \text{weakly in } \L{p} {\Lx p},
\end{equation*}
where, at this level, the functions $s$ and $\sigma$ need not be related to 
each other. 

Next, note that \eqref{lim:11}, applying the Aubin--Lions lemma, also gives
\begin{equation}\label{conv:01}
  \fhi_n \to \fhi \quad \text{strongly in }\, C^0([0,T];H^{1-\eps}(\Omega))
   ~~\text{for every }\,\eps>0.
\end{equation}

The above implies, in particular, pointwise (almost everywhere)
convergence. As we will see below, properties \eqref{lim:11}-\eqref{lim:14}
are sufficient to pass to the limit in the Cahn--Hilliard
system \eqref{CH1b}-\eqref{CH2b}. On the other hand,
as is common in Keller--Segel-type models,
the main difficulties arise when one considers the equation \eqref{nutrb} for the chemical
concentration. In particular, the key step stands in providing a suitable control
of the cross-diffusion term, which has a quadratic growth. To this aim, the choice
of a logistic source term is crucial and a suitable refined argument has to be devised. 

Before detailing the procedure to pass to the limit, we need some preparation.
In this direction, we set
$Z:=W^{1,2p/(p-1)}(\Omega)$ and we first notice that
\begin{align}\no
  & \big\| \sigma_n \mobn(\fhi_n,\sigma_n) \nabla (\ln \sigma_n + \chi (1-\fhi_n) ) 
       \big\|_{L^{2p/(p+1)}(Q)}\\
 \label{oss:11} 
  & \mbox{}~~~~~
   \le M \| \sigma_n^{1/2} \|_{L^{2p}(Q)}
      \big\| \sigma_n^{1/2} \nabla (\ln \sigma_n + \chi (1-\fhi_n) ) \big\|_{L^2(Q)}
   \le c,
\end{align}
the last inequality following from \eqref{st:n13b} and \eqref{st:n17}.

Next, we consider the second term in \eqref{nutrn} multiplied by $z\in Z$.
Integrating by parts and using the analogue on $\Omega$ of \eqref{oss:11}, 
we obtain
\begin{align}\no
  & \io \sigma_n \mobn(\fhi_n,\sigma_n)  \nabla (\ln \sigma_n + \chi (1-\fhi_n) ) \cdot \nabla z\\
 \no
  & \mbox{}~~~~~
   \le \big\| \sigma_n \mobn(\fhi_n,\sigma_n) 
      \nabla (\ln \sigma_n + \chi (1-\fhi_n) ) \big\|_{2p/(p+1)}
     \| \nabla z \|_{2p/(p-1)}\\
 \label{oss:12} 
  & \mbox{}~~~~~
   \le M \| \sigma_n \|_{p}^{1/2}
      \big\| \sigma_n^{1/2} \nabla (\ln \sigma_n + \chi (1-\fhi_n) ) \big\|
        \| z \|_{Z}.
\end{align}
Moreover, we observe that, if $p$ is as in the statement
of Theorem~\ref{THM:WEAK}, i.e., $p\in[3/2,2]$ if ${d}=2$ and $p\in[8/5,2]$ if ${d}=3$, 
then we also have $Z\subset L^\infty(\Omega)$ by Sobolev's embeddings. As a 
consequence, we have
\begin{align}
  \io \beta(\fhi_n) (\kappa_0\sigma_n - \kappa_\infty \sigma_n^p) z
   & \le c \big( 1 + \| \sigma_n \|^p_p \big) \| z \|_{\infty}
 \label{oss:13}  
   \le c \big( 1 + \| \sigma_n \|^p_p \big) \| z \|_{Z}.
\end{align}
Hence, recalling \eqref{st:n17} and \eqref{st:n13b}, it is not difficult to deduce
from \eqref{oss:12} and \eqref{oss:13} that
\begin{equation}\label{oss:14}
  \| s_{n,t} \|_{L^1(0,T;Z^*)} \le c.
\end{equation}
Now, to apply once again the Aubin--Lions lemma, we also need an estimate of
the gradient of $s_n$. To this aim, we need to decouple the information 
on the cross-diffusion term resulting from the energy estimate. In order to achieve
this goal, the constraints on the exponent $p$ in \eqref{b:logi2}
are essential. 

Indeed, we start noticing that, by the second in
\eqref{st:14} (or the corresponding
convergence \eqref{lim:11}), there holds
\begin{equation*} 
%\label{st:32}
  \| \nabla\fhi_n \|_{L^p(0,T;L^{dp/(d-p)}(\Omega))} \le c,
\end{equation*}
where, in the critical case $p=d=2$, $dp/(d-p)$ is intended to be
replaced by any ${\mathfrak p}\in[1,\infty)$. 

Interpolating the above information with the $L^\infty(0,T;H)$-bound resulting from
the second in \eqref{lim:11}, it is then not difficult to obtain
\begin{equation} \label{st:33}
  \| \nabla\fhi_n \|_{L^{p(d+2)/d}(Q)} \le c,
\end{equation}
which holds also in the critical case $p=d=2$ thanks to Lady\v{z}enskaja's inequality
and the last of \eqref{lim:11}.

Combining this fact with the uniform $L^p$-bound for $\sigma_n$, we then deduce that
\begin{equation*} 
%\label{st:33b}
  \| \sigma_n \nabla\fhi_n \|_{L^{1}(Q)} \le c \quext{provided that }\,
   \frac1p + \frac{d}{p(d+2)}\le 1,
\end{equation*}
which corresponds exactly to the conditions $p\ge 8/5$ for $d=3$ and $p\ge 3/2$ for $d=2$ 
stated in Assumption~{\bf (A3)}. More precisely, under such conditions, we have
\begin{equation} \label{st:33b2}
  \| \sigma_n \nabla\fhi_n \|_{L^{\frac{p(d+2)}{2d+2}}(Q)} \le c. 
\end{equation}
We now compare \eqref{st:33b2} with \eqref{oss:11} and notice that, both 
for $d=2$ and for $d=3$, in the admissible range for $p$ we always 
have
\begin{equation*} 
%\label{st:33b3}
  \frac{2p}{p+1} \ge \frac{p(d+2)}{2d+2},
\end{equation*}
whence we arrive at 
\begin{equation*} 
%\label{st:a1}
  \| \nabla \sigma_n \|_{L^{\frac{p(d+2)}{2d+2}}(Q)} 
   = \| \nabla \sigma_n \|_{L^{\frac{p(d+2)}{2d+2}}(0,T;L^{\frac{p(d+2)}{2d+2}}(\Omega))}
  \le c.
\end{equation*}
By $|T_n'|\le 1$ we then also have the corresponding estimate
\begin{equation} \label{st:a1b}
  \| \nabla s_n \|_{L^{\frac{p(d+2)}{2d+2}}(Q)} 
   = \| \nabla s_n \|_{L^{\frac{p(d+2)}{2d+2}}(0,T;L^{\frac{p(d+2)}{2d+2}}(\Omega))}
  \le c.
\end{equation}
Conditions \eqref{oss:14} and \eqref{st:a1b} allow us to apply the generalized
Aubin--Lions in the form \cite[Cor.~4, Sec.~8]{Si}, so obtaining, also thanks
to Sobolev's embeddings,
\begin{equation} \label{conv:21}
  s_n \to s \quad \text{strongly in }\,
   L^{\frac{p(d+2)}{2d+2}}(0,T;L^{ {\mathfrak p} }(\Omega))
   ~~\text{for every }\, \mathfrak p\in\Big[1,\frac{dp(d+2)}{2d(d+1)-p(d+2)}\Big).
\end{equation}
The complicated exponents above are not essential and in fact they can be improved.

Indeed, \eqref{conv:21} implies in particular convergence almost everywhere.
Moreover, the control of the last summand in \eqref{st:n17} 
provides the following uniform integrability estimate:
\begin{equation} \label{uni:int}
  \big\| \sigma_n \ln^{1/p} (1 + \sigma_n) \big\|_{L^p(Q)} \le c
\end{equation}
and, as before, the same bound holds for $s_n$. By Vitali's theorem \cite{Vi},
we then have
\begin{equation} \label{conv:22}
  s_n \to s \quad \text{strongly in }\, L^p(Q).
\end{equation}
We now show that, in fact, the functions $s$ and $\sigma$ do coincide. To this aim,
%we first observe that, thanks to the energy estimate 
%(and in particular to the structure of the modified energy functional,
%cf.~\eqref{propenn}), we have
%
%\begin{equation} \label{conv:72}
%  \| L_n(\sigma_n) \|_{L^\infty(0,T;L^1(\Omega))} \le c,
%\end{equation}
%
%with, as usual, $c>0$ independent of $n$. Now, 
%
setting $\Omega_n^+=\Omega_n^+(t):=\{x\in \Omega:\sigma_n(x,t)\ge n\}$,  
thanks to \eqref{propLn1} there follows $|\Omega_n^+(t)|\le c/(n\ln n)$
for almost every $t\in(0,T)$. As a consequence,
\begin{align}\no
  \| \sigma_n - s \|_{L^1(Q)} 
   & \le \| \sigma_n - s_n \|_{L^1(Q)} + \| s_n - s \|_{L^1(Q)}\\
 \no
  & = \big\| \sigma_n - T_n(\sigma_n) \|_{L^1(Q)} + \| s_n - s \|_{L^1(Q)}\\
 \no
  & \le \int_0^T \int_{\Omega_n^+(t)} \sigma_n(\cdot,t)\,\dit 
   + \| s_n - s \|_{L^1(Q)}\\
 \no
  & \le \int_0^T \Big( |\Omega_n^+(t)|^{\frac{p-1}{p}} \| \sigma_n \|_{L^p(\Omega_n^+(t))} \Big)\,\dit
   + \| s_n - s \|_{L^1(Q)}\\
    \label{conv:71} 
  & \le \frac{c}{n^{\frac{p-1}{p}}} \| \sigma_n \|_{L^1(0,T;L^p(\Omega))}
   + \| s_n - s \|_{L^1(Q)},
\end{align}
and it is readily seen that the \rhs\ tends to zero in view of \eqref{uni:int}
and \eqref{conv:22} as $n\to\infty$. Comparing with \eqref{lim:12},
we then obtain the identification
$s\equiv \sigma$. In particular, the truncation operator $T_n$ disappears
in the limit. For this reason, we can drop the use of the letter $s$ in the 
limit and go back to the original variable $\sigma$. 
Notice also that, applying Vitali's theorem again,
we have
\begin{equation} \label{conv:22b}
  \sigma_n \to \sigma \quad \text{strongly in }\, L^p(Q).
\end{equation}
We are now ready to take the limit $n\to\infty$ in the Cahn--Hilliard system. To this aim, we first
notice that, by the facts that $\fhi_n\to\fhi$ and $\sigma_n\to\sigma$ almost
everywhere, combined with the boundedness and Lipschitz continuity of 
$h,\mobm,\mobn,$ and $\beta$, we may deduce, as $n\to\infty$,
\begin{equation} \label{conv:23}
  h(\fhi_n,\sigma_n),\mobm(\fhi_n,\sigma_n),\mobn(\fhi_n,\sigma_n),\beta(\fhi_n)
   \to h(\fhi,\sigma),\mobm(\fhi,\sigma),\mobn(\fhi,\sigma),\beta(\fhi)
   \quad \text{strongly in }\, L^{\mathfrak{p}}(Q),
\end{equation}
for every $\mathfrak{p}\in [1,\infty)$, thanks also to a generalized version
of Lebesgue's dominated convergence theorem.
Combining this with \eqref{lim:13} we have, by virtue of the 
weak-strong convergence principle,
\begin{equation} \label{conv:24}
  \mobm(\fhi_n,\sigma_n)\nabla \mu_n
   \to \mobm(\fhi,\sigma)\nabla \mu
   \quad \text{weakly in }\, L^2(0,T;H).
\end{equation}
Hence, in view of the above relations, testing \eqref{CH1n} by 
a generic test function $v\in V$ and integrating by parts, 
it is apparent that all terms pass to $n\to \infty$
so to obtain \eqref{wf:phi} in the limit. 

Concerning \eqref{wf:mu}, this is obtained by testing \eqref{CH2n}
by $v\in V$ and then letting $n\to \infty$. To check that 
this procedure work we just need to take care of the nonlinear term depending
on the configuration potential. In other words, going back to the weak convergence in \eqref{lim:14},
we need to identify the limit function $\xi$.
To this aim, we first notice that from \eqref{conv:01}
and Sobolev's embeddings there follows in particular
\begin{equation} \label{conv:25-0}
  \fhi_n \to \fhi
   \quad \text{strongly in }\, L^{{\mathfrak q}}(0,T;L^{{\mathfrak q}}(\Omega)),
\end{equation}
for any ${\mathfrak q}\in[1,6)$. Hence, under our assumptions on $p$, we have in particular
\begin{equation} \label{conv:25}
  \fhi_n \to \fhi
   \quad \text{strongly in }\, L^{{p'}}(0,T;L^{{p'}}(\Omega)),
\end{equation}
where $p'$ is the conjugate exponent to $p$.

Actually, the strong convergence \eqref{conv:25} combined with the weak convergence
\eqref{lim:14} guarantee the identification $\xi = F_1'(\fhi)$ by means of a suitable
version of the standard strong-weak compactness argument for maximal monotone operators.
Indeed, we recall that the assumed Mosco-convergence $F_{n} \to F_1$ implies
a convergence property on the maximal monotone operators induced by the 
derivatives $F_n'$. Referring once more to \cite[Chap.~3]{At} for the background,
what holds is the {\sl graph convergence}
\begin{equation*} 
  F_n' \to F_1' \quext{in }\,L^{p'}(Q) \times L^p(Q).
\end{equation*}
This corresponds to saying that, for every couple 
$[w,\eta]\in L^{p'}(Q) \times L^p(Q)$ such that $\eta=F_1'(w)$ a.e.~in~$Q$
there exists a sequence $\{[w_n,\eta_n]\}\subset L^{p'}(Q) \times L^p(Q)$, with 
$\eta_n=F_n'(w_n)$ a.e.~in~$Q$ and such that 
\begin{equation*} 
%\label{conv:graph}
  [w_n,\eta_n]\to [w,\eta]
   \quext{strongly in }\, L^{p'}(Q) \times L^p(Q).
\end{equation*}
Thanks to this property, an appropriate version of the usual monotonicity 
argument in (reflexive) Banach spaces (cf., e.g., \cite{Ba}) permits us to 
achieve that, as $n\to\infty$,
\begin{equation*}
  F_n'(\fhi_n) \to F_1'(\fhi) \quext{weakly in }\, L^p(Q).
\end{equation*}
Hence, we can pass to the limit $n\to \infty$ in \eqref{CH2n}
so to obtain \eqref{wf:mu}. 

Finally, we need to take the limit in the equation for $\sigma$, 
which is a bit trickier. First of all, we go back to \eqref{nutrn}, test it by
$w\in \Wn$ and integrate by parts. Using the condition
$\dn w = 0$ on the boundary and \eqref{nablaN:identity}, we then get
\begin{align*} \no
  \io \mobn(\fhi_n,\sigma_n) \nabla \sigma_n \cdot \nabla w
    & = \io ( \nabla N(\fhi_n,\sigma_n) - \mobn_1(\fhi_n,\sigma_n) \nabla \fhi_n)
     \cdot \nabla w\\ 
% \label{conv:26c}
   & = - \io N(\fhi_n,\sigma_n) \Delta w - \io \mobn_1(\fhi_n,\sigma_n) \nabla \fhi_n \cdot \nabla w,
\end{align*}
which leads us to the $n$-analogue of \eqref{wf:sigma}, namely
\begin{align} \no
  & \duav{s_{n,t}, w}_{{\Wn}} 
    - \int_\Omega N(\sigma_n,\fhi_n) \Delta w 
	- \int_\Omega \mobn_1(\fhi_n,\sigma_n)\nabla \fhi_n \cdot \nabla w 
	- \chi \int_\Omega \sigma_n \mobn(\fhi_n,\sigma_n) 
    \nabla \fhi_n\cdot \nabla w\\
 \label{sigma:wfn}
  & \mbox{}~~~~~
  = \int_\Omega \beta(\fhi_n) (\kappa_0 \sigma_n - \kappa_\infty \sigma_n^p) w. 
\end{align}
To take the limit in this relation, we first observe that,
by \eqref{conv:22b} and \eqref{conv:23},
\begin{equation} \label{conv:23b}
  \beta(\fhi_n)(\kappa_0 \sigma_n - \kappa_\infty \sigma_n^p)
   \to \beta(\fhi)(\kappa_0 \sigma - \kappa_\infty \sigma^p)
   \quad \text{weakly in }\, L^{1}(Q).
\end{equation}
Next, by \eqref{conv:22b}, \eqref{conv:25} and \eqref{Lip:N}, it turns out that
$N(\fhi_n,\sigma_n) \to N(\fhi,\sigma)$ almost everywhere. As a consequence of the 
generalized Lebesgue theorem we then deduce 
\begin{equation} \label{conv:31}
  N(\fhi_n,\sigma_n) \to N(\fhi,\sigma)
    \quext{strongly in }\,L^p(Q).
\end{equation}
Analogously, recalling \eqref{cresc:n1}, we infer
\begin{equation} \label{conv:31b}
  \mobn_1(\fhi_n,\sigma_n) \to \mobn_1(\fhi,\sigma)
    \quext{strongly in }\,L^{p}(Q).
\end{equation}
To deal with the cross-diffusion terms, for clarity we just consider the worst case,
corresponding, as said, to $d=3$ and $p=8/5$. In that case, the exponent of the space
in \eqref{st:33} reduces to $8/3$. 
Then, using \eqref{conv:23} with the uniform boundedness of $\mobn$ and \eqref{st:33},
we deduce
\begin{equation*} 
%\label{conv:33}
  \mobn(\fhi_n,\sigma_n) \nabla \fhi_n \to  \mobn(\fhi,\sigma) \nabla \fhi 
   \quext{weakly in }\,L^{8/3}(Q),       %  L^{\frac{p(d+2)}{d}}(Q),
\end{equation*}
whence, by virtue of \eqref{conv:22b}, as $n\to\infty$,
\begin{equation} \label{conv:33b}
  \sigma_n \mobn(\fhi_n,\sigma_n) \nabla \fhi_n \to \sigma \mobn(\fhi,\sigma) \nabla \fhi 
    \quext{weakly in }\,L^{1}(Q).
\end{equation}
Analogously, owing to \eqref{conv:31b}, we get, as $n\to\infty$,
\begin{equation} \label{conv:33c}
  \mobn_1(\fhi_n,\sigma_n) \nabla \fhi_n \to \mobn_1(\fhi,\sigma) \nabla \fhi 
    \quext{weakly in }\,L^{1}(Q).
\end{equation}
The above relations serve as a starting point to pass to the 
limit in \eqref{sigma:wfn}. Indeed, the diffusion terms are managed
by means of \eqref{conv:31} and \eqref{conv:33b}, whereas the \rhs\ goes to the desired
limit thanks to \eqref{conv:23b}. On the other hand, the best estimate we 
have on $s_{n,t}$ is given by \eqref{oss:14}. Hence, in order to take the limit
$n\to \infty$, we have to consider, as specified in the statement, 
$w\in C^1([0,T];\Wn)$ and integrate \eqref{sigma:wfn} with respect to
time between $0$ and $t\le T$ integrating by parts the first term.
In this way, the time derivative of $s_n$ disappears; nevertheless,
\eqref{oss:14} still 
does not suffice to take the desired limit, unless one uses a generalized
tool like Helly's selection principle. In particular, the limit function
$s$ is expected to be only $BV$ with respect to time, which would allow it
to have jumps with respect to the time variable. 

In order to exclude this fact, we need to refine a bit
the information on $s_{n,t}$ by exploiting in a suitable
way the uniform integrability property \eqref{uni:int}. This procedure will
allow us to recover also the initial datum in the sense \eqref{ini:2}. 
To this aim, we go back to \eqref{nutrn} and test it by $w\in\Wn$. Using, 
in particular, the fact $\Wn\subset W^{1,\infty}(\Omega)$, 
it is then not difficult to obtain
\begin{align} \no 
  \| s_{n,t}\|_{{{\cal W}^*_{\bn}}}
   & \le c \| \sigma_n^{1/2} \|
    \big\| \sigma_n^{1/2} \nabla (\ln \sigma_n + \chi (1-\fhi_n) ) \big\|
    + c \big( 1 + \| \sigma_n \|^p_p \big)\\
  \label{uni:11a}          
    & \le c \| \sigma_n^{1/2} \| \big\| \sigma_n^{1/2} \nabla (\ln \sigma_n + \chi (1-\fhi_n) ) \big\|
    + c + c \| \sigma_n \|^p_p 
    =: M_{1,n} + c + M_{2,n}.
\end{align}
Now, it is clear that $t\mapsto \| \sigma_n^{1/2}(t) \|$ is bounded in $L^{2p}(0,T)$ as a consequence
of \eqref{conv:22b} and that 
$t\mapsto \big\| \sigma_n^{1/2}(t) \nabla (\ln \sigma_n(t) + \chi (1-\fhi_n(t)) ) \big\|$
is bounded in $L^{2}(0,T)$ as a consequence of \eqref{st:n13b}. Combining these facts,
we readily obtain that
\begin{equation} \label{conv:33x}
  \| M_{1,n} \|_{L^{2p/(p+1)}(0,T)} \le c,
\end{equation}
with $c$ independent of $n$. Let us now set,
for $r>0$, $\Phi(r):=r \ln(e+r)$ and let us notice that $\Phi$ is 
convex and increasing.
Then, applying $\Phi$ to inequality \eqref{uni:11a} and integrating in time,
it is not difficult to check that 
\begin{align} \no
  \int_0^T \Phi\big( \| s_{n,t}\|_{{{\cal W}^*_{\bn}}} \big)
   & \le \int_0^T \Phi(M_{1,n} + c + M_{2,n})\\
 \no
   & \le c + c \int_0^T \Phi(M_{1,n}) + c \int_0^T \Phi(M_{2,n})\\
 \no
   & \le c + c \int_0^T \Phi\big(c \| \sigma_n \|^p_{p} \big)\\
 \label{uni:14}  
  & \le c + c \io \sigma_n^p \ln (e + \sigma_n) \le C,
\end{align}
where $C>0$ is a computable constant depending only on the known data 
of the problem. Notice also that, to estimate the first integral on the second row,
we used \eqref{conv:33x} with the fact $2p/(p+1)>1$.

Next, let $0\le \tau < t \le T$. Then, we have
\begin{equation*} 
%\label{uni:21}
 \frac{\| s_n(t) - s_n(\tau) \|_{{{\cal W}^*_{\bn}}}}{|t-\tau|}
   \le \int_\tau^t \frac1{|t-\tau|} \| s_{n,t}(r) \|_{{{\cal W}^*_{\bn}}} \,\dir.
\end{equation*}
Using that $\Phi$ is nondecreasing and convex, and applying
Jensen's inequality, we then deduce
\begin{align*} \no
  \Phi\bigg( \frac{\| s_n(t) - s_n(\tau) \|_{{{\cal W}^*_{\bn}}}}{|t-\tau|} \bigg)
   & \le \Phi\bigg( \int_\tau^t \frac1{|t-\tau|} \| s_{n,t}(r) \|_{{{\cal W}^*_{\bn}}} \,\dir \bigg)\\
 \no  
  & \le  \int_\tau^t \frac1{|t-\tau|} \Phi\big( \| s_{n,t}(r) \|_{{{\cal W}^*_{\bn}}} \big) \,\dir\\
% \label{uni:22} 
  & \le \frac1{|t-\tau|} \int_0^T \Phi\big( \| s_{n,t}(r) \|_{{{\cal W}^*_{\bn}}} \big) \,\dir
    \le \frac C{|t-\tau|},
\end{align*}
where $C>0$ is the constant introduced in \eqref{uni:14}.

Then, using again the strict monotonicity of $\Phi$, we deduce 
\begin{equation*} 
%\label{uni:23}
  \frac{\| s_n(t) - s_n(\tau) \|_{{{\cal W}^*_{\bn}}}}{|t-\tau|} 
   \le \Phi^{-1} \Big(\frac{C}{|t-\tau|} \Big),
\end{equation*}
or, in other words,
\begin{equation*} 
%\label{uni:24}
  \| s_n(t) - s_n(\tau) \|_{{{\cal W}^*_{\bn}}}
   \le | t - \tau |\Phi^{-1}\Big(\frac C{|t-\tau|}\Big).
\end{equation*}
Then, noting that $\Phi^{-1}$ is strictly sublinear at infinity, as a direct check
shows, we may observe that there holds the following equicontinuity property:
for every $\epsi>0$ there exists $\delta>0$ such that for every $n\in \NN$ 
and every $0\le \tau < t \le T$ with $|t-\tau|<\delta$ there holds
\begin{equation*} 
%\label{uni:25}
  \| s_n(t) - s_n(\tau) \|_{{{\cal W}^*_{\bn}}} < \epsi.
\end{equation*}
Now, using \eqref{st:n18} with \eqref{Tn11}-\eqref{Tn13},
%
%from the control 
%
%\begin{equation} \label{uni:31}
%  \| L_n(\sigma_n) \|_{L^\infty(0,T;L^1(\Omega))} \le c,
%\end{equation}
%
%which is a consequence of the approximate energy estimate, combined
%with the coercivity properties of the function $L_n$ (cf.~\eqref{propLn1})
%and the expression of the truncation operator $T_n$ (see in particular
%\eqref{Tn13}), 
%
it is easy to deduce
\begin{equation*} 
%\label{uni:32}
  \| s_n \|_{L^\infty(0,T;L^1(\Omega))} \le c.
\end{equation*}
Hence, observing that $L^1(\Omega) \subset {\cal W}^*_{\bn}$ with compact embedding,
if we take as $\calZ$ a generic (reflexive) Banach space (which, of 
course, will have a negative order as a Sobolev space) such that 
\begin{equation*}
% \label{uni:26}
  L^1(\Omega) \subset \subset \calZ \subset {{\cal W}^*_{\bn}},
\end{equation*}
using some interpolation it is not difficult to check that Ascoli's theorem
can be applied to the sequence $\{s_n\}$ in the space 
$C^0([0,T];\calZ)$ so to obtain
\begin{equation} \label{uni:27}
  s_n \to s = \sigma \quext{strongly in }\,C^0([0,T];\calZ)
\end{equation}
and, a fortiori, in $C^0([0,T];{{\cal W}^*_{\bn}})$.
In particular, since $s_n|_{t=0} = T_n(\sigma_0)$ and $T_n(\sigma_0)$
tends to $\sigma_0$ in $L^1(\Omega)$ thanks to Lebesgue's dominated
convergence theorem, we obtain that the initial condition $\sigma|_{t=0} = \sigma_0$
is satisfied in a standard sense,
which excludes the occurrence of jumps of $\sigma$ with respect
to the time variable. Moreover, \eqref{uni:27} allows us to pass
to the limit in the time-integrated version of \eqref{sigma:wfn} so to
obtain \eqref{wf:sigma} (which, we remark, also incorporates the boundary
conditions). 
We incidentally notice that \eqref{reg:2:dtsigma} also follows
from the above procedure. In particular, the second regularity 
in \eqref{reg:2} is a consequence
of \eqref{st:n18} and an equiintegrability argument. 
Finally, \eqref{rego:add} follows from the analogue of \eqref{st:62old}.
This concludes the proof of Theorem~\ref{THM:WEAK}.

%%%%%%%%%%%%%%%%%%%%%%%%%%%%%%%%%%%%%%%%%%%%%%%%%%%%%%%%%%%%%%%%%%%%%%%%%%%%%%%%%%%%%%%%%%%%%%%%%%%%%%%%%%%%%%%%%%
%%%%%%%%%%%%%%%%%%%%%%%%%%%%%%%%%%%%%%%%%%%%%%%%%%%%%%%%%%%%%%%%%%%%%%%%%%%%%%%%%%%%%%%%%%%%%%%%%%%%%%%%%%%%%%%%%%

\section{Proof of the regularity results}
\label{sec:rego}

The proofs of the regularity results are mainly based on the derivation
of higher-order additional sets of a-priori estimates. It is worth observing from the very
beginning that these estimates will be derived in a formal way by working
on the ``original'' system~\SYS. We believe that, at this regularity
level, obtaining the estimates in a fully rigorous way would require a very
lengthy and technical adaptation of the approximation scheme. Since such a 
procedure would, however, present a very limited mathematical interest,
we prefer to proceed formally.

\smallskip

\begin{proof}[Proof of Theorem \ref{THM:WEAK:2d3d}]
We distinguish between the 2D and the 3D cases, which have to be managed
by different methods.

\smallskip
\noindent%
{\bf Two dimensional case.}~~%
For convenience, let us start with the two dimensional case 
recalling that now $p=2$. In that setting, 
we test \eqref{nutrb} by $\sigma$ to obtain
\begin{equation} \label{co:52}
  \frac12 \ddt \| \sigma \|^2
   + m_0 \| \nabla \sigma \|^2
   + \kappa_\infty b_0\norma{\sigma}^3_3
\le c \| \sigma \|^2
   + \chi \io \sigma \mobn(\fhi,\sigma) \nabla \fhi \cdot \nabla \sigma.
\end{equation}
Then, we test \eqref{CH2b} by $-\Delta \fhi$ to infer that
\begin{equation} \label{co:53}
  \io F_1''(\fhi)|\nabla \fhi|^2 
   + \| \Delta \fhi \|^2
  \le - \chi\io \sigma \Delta \fhi
   + \lambda \| \nabla \fhi \|^2
   + \| \nabla \fhi \| \| \nabla \mu \|,
\end{equation}
whence, squaring and using the previous estimates
with the monotonicity of $F_1'$, standard manipulations lead us to 
\begin{equation} \label{co:54}
  \| \Delta \fhi \|^4
   \le c \big(1 + \| \sigma \|^4 +  \| \nabla \mu \|^2 \big).
\end{equation}
Then, to control the last term on the \rhs\ of \eqref{co:52}, 
we observe that, by H\"older's and standard
interpolation inequalities holding for $d=2$, 
\begin{align}\no
  \chi\io \mobn(\fhi,\sigma) \sigma \nabla \fhi\cdot \nabla  \sigma
   & \le c \| \sigma \|_4
   \| \nabla \fhi \|_4
   \| \nabla \sigma \|\\
 \no
   & \le c \| \sigma \|^{1/2} \| \sigma \|_V^{1/2} 
   \| \fhi \|_V^{1/2} \| \fhi \|_{H^2(\Omega)}^{1/2}
      \| \nabla \sigma \|\\
 \no
   & \le c \| \sigma \|^{1/2} \big( \| \sigma \|^{1/2} + \| \nabla \sigma \|^{1/2} \big) 
   \big( 1 + \| \Delta \fhi \|^{1/2} \big)
      \| \nabla \sigma \|\\
 \label{co:55}
   & \le \frac{m_0}2 \| \nabla \sigma \|^2
    + \frac12 \| \Delta \fhi \|^4
    + c ( \| \sigma \|^4 +1),
\end{align}
where we used in particular that $t \mapsto\norma{\fhi(t)}_V$ is $L^\infty (0,T)$ and Young's inequality.
Summing \eqref{co:52} with \eqref{co:54} 
and using \eqref{co:55}, we then arrive at 
\begin{align*} 
%\no
  & \frac12 \ddt \| \sigma \|^2
   + \frac{m_0}2 \| \nabla \sigma \|^2
   + \kappa_\infty b_0\norma{\sigma}^3_3
   + \frac12 \| \Delta \fhi \|^4
%   \\
% \label{co:52b}   
   \le c \big (1 + \| \sigma \|^4 + \| \nabla \mu \|^2 \big). 
\end{align*}
Next, recalling that \eqref{reg:2} holds with $p=2$,
an application of Gr\"onwall's lemma, along with elliptic regularity results,
yields the additional regularity bounds
\begin{align} \label{st:71}
  & \| \fhi \|_{L^4(0,T;H^2(\Omega))} \le c,\\
 \label{st:72}
  & \| \sigma\|_{L^\infty(0,T;H) \cap L^2(0,T;V)} \le c,
\end{align}
where \eqref{ass:initialsigma:2d} has also been used.

Using \eqref{st:71}-\eqref{st:72} and comparing terms in \eqref{nutrb}, 
it is then a standard matter to derive that
\begin{align*}
%\label{st:73}
  \| \sigma_t\|_{\L2 \Vp} \le c.
\end{align*}
This permits us to write the nutrient equation in the standard form
\eqref{wf:3} rather than in the integrated form \eqref{wf:sigma}; moreover, 
by classical results for second-order parabolic equations, this also gives
the continuity property in \eqref{reg:2d:2}.
By the above relations we also recover
the usual regularity scenario for the 
Cahn--Hilliard equation with singular potential under the ``energy regularity'' of 
the initial data in two space dimensions, i.e.,
\begin{equation*}
% \label{rego:fhi:en}
  \fhi \in H^1(0,T;\Vp) \cap L^\infty(0,T;V) \cap L^4(0,T;H^2(\Omega))
   \cap L^2(0,T;W^{2,q}(\Omega)),
\end{equation*}
for any $q\in[1,\infty)$, where the latter regularity property is obtained by 
considering once more \eqref{CH2b} as an elliptic equation with maximal monotone 
nonlinearity, namely
\begin{equation} \label{CH2:ell}
   -\Delta \fhi + F_1'(\fhi) = \lambda\fhi + \chi\sigma + \mu,
\end{equation}
and noting that the \rhs\ lies (or, in a suitable approximation, is uniformly bounded in)
$L^2(0,T;L^q(\Omega))$ thanks to continuity of the two-dimensional
embedding $V \hookrightarrow L^q(\Omega)$ for any $q\in[1,\infty)$.

\medskip
\noindent
{\bf Three dimensional case.}~~%
We now move to the three dimensional case. As said before, we proceed formally and,
to begin, we provide an auxiliary estimate which will play an important role in the sequel.
To this aim, we set $\gamma(\fhi) := - ({F'_1})_-^5(\fhi)$ (where $(\cdot)_-$ denotes the negative
part of a quantity). Then, noting that $\gamma$ is monotone and nonpositive, 
we test \eqref{CH2b} by $\gamma(\fhi)$ to obtain
\begin{equation}\label{co:seg:11}
  \io {F'_1}(\fhi) \gamma(\fhi) 
   + \io \gamma'(\fhi) |\nabla \fhi|^2 
   = \lambda \io \fhi \gamma(\fhi)
   + \chi \io \sigma \gamma(\fhi)
   + \io \mu \gamma(\fhi).
\end{equation}
Now, as $\gamma(\fhi) = -  {(F'_1)}_-^5(\fhi)=F_1'(\fhi)^5 \chi_{\{\fhi < 0\}}$
(recall the normalization property $F_1'(0)=0$, cf.~{\bf (A1)}), 
it is clear that 
\begin{equation}\label{co:seg:12}
  \io {F'_1}(\fhi) \gamma(\fhi) 
   = \| {(F'_1)}_-(\fhi) \|_6^6.
\end{equation}
Moreover, the second term on the \lhs\ of \eqref{co:seg:11}
is clearly nonnegative, while the second term on the 
\rhs\ is nonpositive due to \eqref{reg:2:nonneg}. By H\"older's 
and Young's inequalities we also have
\begin{align}\no
  \lambda \io \fhi \gamma(\fhi)
   + \io \mu \gamma(\fhi)
   & \le c \big( \| \fhi \|_6 + \| \mu \|_{6} \big) 
      \| \gamma(\fhi) \|_{{6/5}} 
   \le c \big( 1 + \| \mu \|_{6} \big)\| {(F'_1)}_-(\fhi) \|_{6}^5 \\
 \label{co:seg:13}
  & \mbox{}~~~~~
    \le c \big( 1 + \| \mu \|_{6}^6 \big)
    + \frac12 \| {(F'_1)}_-(\fhi) \|_{6}^6.
\end{align}
Hence, replacing \eqref{co:seg:12} and \eqref{co:seg:13} into \eqref{co:seg:11},
it is not difficult to deduce
\begin{equation*}
%\label{co:seg:14}
 \frac12 \| {(F'_1)}_-(\fhi) \|_{{6}}^6
    \le c \big( 1 + \| \mu \|_{{6}}^6 \big).
\end{equation*}
Taking the cubic root, using Sobolev's embeddings and recalling \eqref{reg:4} this implies 
\begin{equation}\label{re:14}
   {(F'_1)}_-(\fhi) \in L^2(0,T;L^6(\Omega)).
\end{equation}

\smallskip

\noindent%
Next, recalling assumptions \eqref{smallness}-\eqref{n:cost}, 
we may test \eqref{nutrb} (with $n\equiv 1$ and $p=2$) 
by $\sigma$ to obtain
\begin{equation}\label{co:seg:15}
 \frac12 \ddt \| \sigma \|^2
   + \| \nabla \sigma \|^2
   + \kappa_\infty b_0 \| \sigma \|_{3}^3
   \le c \| \sigma \|^2
  + \chi \io \sigma \nabla \fhi\cdot  \nabla \sigma
\end{equation}
and we need to properly manipulate the last term on the \rhs. 
To this aim, we integrate by parts and exploit the no-flux conditions with 
relation \eqref{CH2} to deduce
\begin{align*}\no
  \chi \io \sigma \nabla \fhi\cdot  \nabla \sigma
   & = \frac\chi2 \io \nabla \fhi \cdot \nabla (\sigma^2)
    = - \frac\chi2 \io \Delta \fhi \, \sigma^2 \\
 \no
  & = \frac\chi2 \io \mu\sigma^2 
   - \frac\chi2 \io {F'_1}(\fhi) \sigma^2
   + \frac{\lambda\chi}2 \io \fhi \sigma^2
   + \frac{\chi^2}2 \| \sigma \|_{3}^3\\
% \label{co:seg:16}
  & =: \II_1 + \II_2 + \II_3 + \II_4.
\end{align*}
We now provide a bound of the various terms on the \rhs. First of all, for every $\d>0$,
\begin{equation*}
%\label{co:seg:17}
  \II_1 \le \| \mu \|_{{6}} \| \sigma \|_{{3}} \| \sigma \|
   \le \d \| \sigma \|_{{3}}^2 + c_\d \| \mu \|_{V}^2 \| \sigma \|^2.
\end{equation*}
The second term is the key one. Using again that
$F_1'(\fhi)$ has the same sign as $\fhi$, we have
\begin{align*}\no
  \II_2 & = - \frac\chi2 \int_{\{\fhi{\geq}0\}} {F'_1}(\fhi) \sigma^2
   - \frac\chi2 \int_{\{\fhi<0\}} {F'_1}(\fhi) \sigma^2
   \le - \frac\chi2 \int_{\{\fhi<0\}} {F'_1}(\fhi) \sigma^2 \\
 \no
   & = \frac\chi2 \io |{(F'_1)}_-(\fhi)| \sigma^2 
   \le \frac\chi2 \| {(F'_1)}_-(\fhi) \|_{6} \| \sigma \|_{3} \| \sigma \|\\
%  \label{co:seg:18}
   & \le \d \| \sigma \|_{3}^2 
         + c_\d \| {(F'_1)}_-(\fhi) \|_{6}^2 \| \sigma \|^2.
\end{align*}
The control of $\II_3$ is immediate, while $\II_4$ has to be moved
to the \lhs. Collecting the above considerations, then \eqref{co:seg:15} gives, 
for every ``small'' $\d >0$ and correspondingly ``large'' $c_\delta>0$,
\begin{align*}\no
 & \frac12 \ddt \| \sigma \|^2
   + \| \nabla \sigma \|^2
   + \Big( \kappa_\infty b_0 - \frac{\chi^2}2 \Big) \| \sigma \|_{3}^3
%   \\
%\label{co:seg:15b}
%  & \mbox{}~~~~~
   \le c_\delta \Big( 1 +  \| {(F'_1)}_-(\fhi) \|_{6}^2 + \| \mu \|_{V}^2 \Big) \| \sigma \|^2
    + 3\d \| \sigma \|_{3}^2.
\end{align*}
Hence, under the compatibility condition \eqref{smallness},
recalling \eqref{reg:4} and the preliminary bound \eqref{re:14} we can adjust $\d \in (0,1)$ 
and apply Gr\"onwall's lemma to deduce
\begin{align*}
%\label{re:21}
  \sigma \in \LIH \cap \LDV \cap \L3{\Lx3}.
\end{align*}
Finally, with this property at hand, the regularity of $\fhi$ can be bootstrapped 
easily by arguing as done above for the two dimensional case.
\end{proof}
\beos\label{rem:mobn0}
 It is not difficult to check that, in the three-dimensional case, the ``smallness''
 condition \eqref{smallness} might be avoided if one takes a superquadratic
 logistic term on the \rhs\ of \eqref{nutr}, i.e., 
 $\beta(\fhi)(\kappa_0\sigma - \kappa_\infty\sigma^{2+\rho})$, where $\rho>0$
 may be arbitrarily small.
\eddos
\beos\label{rem:mobn}
In the three dimensional case, even when the mobility $\mobn$ is not taken as a constant function,
something could still be said. Indeed, \eqref{co:seg:15} would then be replaced by
\begin{equation*}
%\label{co:seg:15:2b}
 \frac12 \ddt \| \sigma \|^2
   + m_0 \| \nabla \sigma \|^2
   + \kappa_\infty b_0 \| \sigma \|_{3}^3
   \le c \| \sigma \|^2
  + \chi \io \mobn(\fhi,\sigma) \sigma \nabla \sigma \cdot \nabla \fhi
\end{equation*}
and the last term could be integrated by parts as follows:
\begin{equation}\label{n2:11}
  \chi \io \mobn(\fhi,\sigma) \sigma \nabla \sigma \cdot \nabla \fhi
    = - \chi \io N_2(\fhi,\sigma) \Delta \fhi 
     - \chi \io \mobn_2(\fhi,\sigma) | \nabla \fhi |^2,
\end{equation}
where we have set    
\begin{equation*}
%\label{defi:N2}
  N_2(\fhi,\sigma) := 
    \int_0^\sigma \mobn(\fhi,s)s\,\dis, \qquad
   \mobn_2(\fhi,\sigma) := \int_0^\sigma 
    \partial_\fhi \mobn(\fhi,s) s \,\dis.
\end{equation*}
Then, the procedure performed above can be adapted at least 
when either $\mobn$ depends only on $\sigma$ (so that the last integral
in \eqref{n2:11} disappears) or $\mobn$ satisfies proper structure assumptions
ensuring that the last integral in \eqref{n2:11} is nonnegative
(so that it can be moved to the \lhs\ and does not need to be controlled).
On the other hand, in the general case (i.e., for $\mobn$ depending both on $\fhi$ and 
$\sigma$ with no sign conditions), the last integrand in \eqref{n2:11}
behaves like $\sigma^2|\nabla \fhi|^2$, which appears to have a supercritical
behavior in space dimension $d=3$, as the interested reader can verify.
\eddos
\noindent%
\begin{proof}[Proof of Theorem~\ref{THM:REG:1}]
Again, we derive additional a-priori estimates in a formal way and 
without referring to the proposed approximation.
We also recall that, at this level, weak solutions are already known to enjoy the 
regularity in \eqref{reg:1}-\eqref{reg:5} and \eqref{reg:2d:1}-\eqref{reg:2d:2}.
That said, we formally differentiate equation \eqref{CH2b} with 
respect to time obtaining the identity
\begin{equation}\label{differ}
  \mu_t = - \Delta \fhi_t + F_1''(\fhi)\fhi_t - \lambda \fhi_t - \chi\sigma_t.
\end{equation}
Next, we multiply \eqref{CH1b} (where $\mobm\equiv 1$)
by $\mu_t$, the above expression \eqref{differ} 
by $\fhi_t$, and add to both sides the term $\norma{\fhi_t}^2$.
Then, summing the resulting equalities together and integrating by parts we infer
\begin{align*}\notag
& \frac 12 \ddt \norma{\nabla \mu}^2 
  + \norma{\fhi_t}^2_V	+ \io F''_1(\fhi)|\fhi_t|^2\\ 
 & \quad  \notag
  = \io S(\fhi,\sigma) \mu_t
   + (1+\lambda) \norma{\fhi_t}^2
   + \chi \io \sigma_t \fhi_t\\
 & \quad  
% \label{est:proof}
  =  \ddt \io S(\fhi,\sigma) \mu
  - \io \partial_t (S(\fhi,\sigma)) \mu
  + (1+\lambda) \norma{\fhi_t}^2
  + \chi \io \sigma_t \fhi_t.
\end{align*}
Note that here we assumed, just for simplicity, that $F_1$ is twice differentiable,
which is unnecessarily true under our assumption {\bf (A1)}. However,
it is easy to see that, using standard convex analysis tools,
the argument might be adapted to work for nonsmooth, but convex, 
$F_1$ (as in our case). That said, we add to the above relation the result 
of \eqref{nutrb} (where $\mobn\equiv1$)
tested by $\sigma_t$. Using the specific expression of the source 
term in \eqref{b:logi2} with $p=2$, after some rearrangements we obtain 
\begin{align}\no
  & \ddt \Big( \frac 12\norma{\nabla \mu}^2 
   - \io S(\fhi,\sigma) \mu 
   + \frac 12\norma{\nabla \sigma}^2 \Big)
   + \norma{\fhi_t}^2_V
   + \norma{\sigma_t}^2\\
  \no
 & \quad 
  \leq (1+\lambda) \norma{\fhi_t}^2
  - \io \partial_t (S(\fhi,\sigma)) \mu
  + \chi \io \sigma_t \fhi_t
  - \chi\io \big( \nabla \sigma \cdot \nabla \fhi + \sigma \Delta \fhi\big) \sigma_t \\
 \label{key:11}
  & \quad \qquad
   + \io  \beta(\fhi) ( \kappa_0 \sigma -  \kappa_\infty \sigma^2) \sigma_t,
\end{align}
and we need to control the terms on the \rhs. First of all, 
by elementary interpolation, it is clear that
\begin{equation}\label{r2:10} 
  (1+\lambda)\norma{\fhi_t}^2
   \leq \frac 18 \norma{\fhi_t}_{V}^2
   + c\norma{ \fhi_t}_*^2.
\end{equation}
Next, by the global Lipschitz continuity of $S$ (cf.~{\bf (A2)})
and a well-known chain rule formula for Lipschitz functions in Sobolev 
spaces, we obtain
\begin{equation*}
  - \io \partial_t (S(\fhi,\sigma)) \mu
   \le c \io (| \fhi_t | + |\sigma_t|) | \mu |
   \le \frac18 \| \fhi_t \|^2 + \frac18 \| \sigma_t \|^2 + c \| \mu \|^2.
\end{equation*}
Analogously, owing to interpolation once again, we deduce
\begin{equation*}
   \chi \io \sigma_t \fhi_t
     =  \chi\duav{\sigma_t,\fhi_t}_{V} 
      \leq \frac18 \norma{\fhi_t}_{V}^2+ c \norma{  \sigma_t}_*^2.
\end{equation*}
The estimation of the remaining terms is just a bit more involved. Firstly
we notice that, by standard Sobolev's embeddings holding both in the two- 
and in the three-dimensional case, we have
\begin{align*} \no
  - \chi\io \big( \nabla \sigma \cdot \nabla \fhi + \sigma \Delta \fhi\big) \sigma_t
   & \leq c \big( \norma{\nabla \sigma} \norma{\nabla \fhi}_{{\infty}} \norma{\sigma_t} 
      + \norma{\sigma}_{{4}} \norma{ \Delta \fhi}_{{4}} \norma{ \sigma_t} \big)\\ 
% \label{r2:13}
  & \leq \frac 18 \norma{\sigma_t}^2
   + c \norma{\sigma}^2_V\norma{\fhi}^2_{W^{2,4}(\Omega)}.
\end{align*}
Next, using \eqref{beta:bou}, we infer
\begin{align} \no
  \io  \beta(\fhi) (  \kappa_0 \sigma -  \kappa_\infty \sigma^2) \sigma_t
   & \leq B \big(  \kappa_0 \| \sigma \| +  \kappa_\infty \| \sigma \|_{4}^2 \big) \norma{\sigma_t} \\ 
 \label{r2:14}
  & \leq \frac 18 \norma{\sigma_t}^2
   + c \big( 1 + \norma{\sigma}^4_V \big).
\end{align}
Replacing the outcome of \eqref{r2:10}-\eqref{r2:14} into \eqref{key:11},
we then arrive at
\begin{align}\no
  & \ddt \Big( \frac 12\norma{\nabla \mu}^2 
   - \io S(\fhi,\sigma) \mu 
   + \frac 12\norma{\nabla \sigma}^2 \Big)
   + \frac58 \norma{\fhi_t}^2_V
   + \frac58 \norma{\sigma_t}^2\\
 \label{key:11b}
 & \quad 
  \leq c + c \| \fhi_t \|_*^2
   + c \| \sigma_t \|_*^2
   + c \| \mu \|^2
   + c \big( 1 + \norma{\sigma}^2_V + \norma{\fhi}^2_{W^{2,4}(\Omega)} \big) \norma{\sigma}^2_V.
\end{align}
To close the estimate, we go back to \eqref{co:seg:15}. Neglecting some nonnegative terms
on the \lhs\ and performing standard manipulations, that relation implies
\begin{equation}\label{co:seg:15:2}
  \frac12 \ddt \| \sigma \|^2
   \le c \| \sigma \|^2
    + c\| \sigma \|_{{6}} \| \nabla \sigma \| \| \nabla \fhi \|_{{3}}
   \le c \big( 1 + \| \fhi \|_{H^2(\Omega)}^2 \big) \| \sigma \|^2_{V}.
\end{equation}
Adding this relation to \eqref{key:11b}, we recover the full $V$-norm of $\sigma$ 
on the \lhs. Next, we prove that the functional we get under time derivative is coercive.
In this direction, we employ \eqref{CH2}, Young's inequality and the uniform
boundedness of $S$ (in particular we can now use that $-1 \le \fhi \le 1$ 
almost everywhere) to infer that 
\begin{align}\notag
  - \io S(\fhi,\sigma) \mu 
   & = - \io S(\mu - \mu_\Omega) - \io S\mu_\Omega
  \geq - c \norma{\mu - \mu_\Omega}
  - c |\mu_\Omega|\\ 
 \no 
  & \geq - c \norma{\nabla \mu}- c |\mu_\Omega|
   \geq -\frac 18 \norma{\nabla \mu}^2 - c |\mu_\Omega| - c \\ 
 \no
  & \geq - \frac 18 \norma{\nabla \mu}^2  
  - c \big( \| F_1'(\fhi) \|_{1} 
     + \| \fhi \|_{1} 
     + \| \sigma \|_{1} \big)\\
 \label{proof:Sdt}
  & \geq - \frac 18 \norma{\nabla \mu}^2  
  - c \| F_1'(\fhi) \|_{1} 
  - c.
\end{align}
Next, testing once more \eqref{CH2b} by $\fhi-\fhi\OO$ and proceeding similarly with
\eqref{co:16}-\eqref{co:18}, we arrive at
\begin{equation}\label{co:18xx}   
  \alpha \| F_1'(\fhi) \|_{{1}} 
   \le c \big( 1 + \| \nabla \mu \| \big),
\end{equation}
where ${\alpha>0}$ is as in \eqref{co:17} and
$c>0$ on the \rhs\ also depends on other quantities that have
already been controlled uniformly with respect to time.
As a consequence, from \eqref{key:11b}-\eqref{co:seg:15:2}, we deduce the 
differential inequality
\begin{align} \no
  & \ddt \underbrace{\Big( \frac 12\norma{\nabla \mu}^2 
   - \io S(\fhi,\sigma) \mu 
   + \frac 12\norma{\sigma}_V^2 \Big)}_{=: {\cal J}}
   + \frac58 \norma{\sigma_t}^2
   \leq c 
   + c \| \fhi_t \|_*^2
   + c \| \sigma_t \|_*^2
   + c \| \mu \|^2\\
  \label{key:11c}
  & \quad\qquad
   + c \big( 1 + \norma{\sigma}^2_V + \norma{\fhi}^2_{W^{2,4}(\Omega)}
   + \norma{\fhi}^2_{H^{2}(\Omega)}\big) \norma{\sigma}^2_V,
\end{align}
where the functional $\calJ$, thanks to \eqref{proof:Sdt}-\eqref{co:18xx}, satisfies
\begin{align*}\no
  \calJ 
  & {\geq \frac 38 \norma{\nabla \mu}^2 
   - c \| \nabla \mu \|
   - c
   + \frac 12\norma{\sigma}_V^2 }
%   \\
% \label{key:11c2}
%  & 
  \geq \frac 14 \norma{\nabla \mu}^2  
   + \frac 12\norma{\sigma}_V^2 - C,
\end{align*}
and $C>0$ depends only on quantities that have already been controlled
uniformly in time. Hence, for $C>0$ as above, \eqref{key:11c} can be 
rewritten in the form
\begin{align*} \no
  & \ddt ( \calJ + C) 
   + \frac58 \norma{\fhi_t}^2_V
   + \frac58 \norma{\sigma_t}^2\\
% \label{key:11d}  
  & \quad \quad
   \leq c 
   + c \| \fhi_t \|_*^2
   + c \| \sigma_t \|_*^2
   + c \| \mu \|^2
   + c \big( 1 + \norma{\sigma}^2_V + \norma{\fhi}^2_{W^{2,4}(\Omega)}
   + \norma{\fhi}^2_{H^{2}(\Omega)}\big) (\calJ + C).
\end{align*}
Then, recalling \eqref{reg:4} and \eqref{reg:2d:1}-\eqref{reg:2d:2},
an application of Gr\"onwall's lemma yields the estimate
\begin{align*}
 \norma{\fhi}_{\H1 V }
  + \norma{\nabla \mu}_{\L\infty H}
  + \norma{\sigma}_{\H1 H \cap \L\infty V}
  \leq c,
\end{align*} 
provided that the functional $\calJ$ is finite at the initial time,
and we actually note that this follows, at least formally, from~\eqref{init:reg}. 

Next, using the control on the mean value of $\mu$ resulting from \eqref{co:18xx}, 
it is a standard matter to infer that
\begin{align*}
	\norma{\mu }_{\L\infty V}\leq c.
\end{align*}
In turn, this also allows us to improve the regularity of $\fhi$. Indeed,
we may go back to relation \eqref{CH2:ell} and notice that, now, the 
\rhs\ lies in the space $\LIV$. Then, arguing once more as in~\cite[Lemmas~7.3 and 7.4]{GGW},
we deduce that
\begin{align}
\label{fhi:Linftyq}
  \norma{F_1'(\fhi)}_{L^\infty(0,T;L^q(\Omega))} 
   + \norma{\fhi}_{\L\infty {W^{2,q}(\Omega)}} \leq c,
\end{align}
where $q=6$ if $d=3$ and $q\in [1,\infty)$ if $d=2$.
Finally, by a comparison of terms in \eqref{CH1b}, it is easy to check
that
\begin{align*}
%\label{proof:2d:fhiq}
  \norma{ \fhi_t}_{\L\infty \Vp} \leq c,
\end{align*}
whereas, applying elliptic regularity in \eqref{nutrb}, one can 
easily deduce
\begin{align*}
  \norma{ \sigma }_{L^2(0,T;H^2(\Omega))} \leq c.
\end{align*}
Noting that the continuity property in \eqref{reg:strong:2d:3} is, once more, a consequence of
standard regularity results, this concludes the proof of the theorem.
\end{proof}
\begin{proof}[Proof of Theorem~\ref{THM:REG:2}]
First of all, proceeding as in \cite[Lemmas~7.3 and 7.4]{GGW} and 
using the growth condition \eqref{F:growth} with the 
Trudinger--Moser inequality (see also \cite{NSY}),
we deduce
\begin{align}\label{proof:2d:F}
  \norma{F''_1(\fhi)}_{\L\infty {\Lx{q}}} \leq c,
\end{align}
for any $q \in [1,\infty)$.
This acts as a starting point to prove the additional regularity in the statement. 
As before, we proceed formally to avoid unnecessary technicalities, noting that rigorous estimates 
could be performed, e.g., by working on a time discrete level as done in \cite{GGW}.
In this direction, we differentiate in time \eqref{CH1b} 
(where, we recall, $\mobm\equiv 1$), to find
\begin{equation*}
%\label{re2:11}   
   \fhi_{tt} = \Delta \mu_t + (S(\fhi,\sigma))_t = \Delta \mu_t - m \fhi_t 
     + \partial_\fhi h(\fhi,\sigma) \fhi_t
     + \partial_\sigma h(\fhi,\sigma) \sigma_t.
\end{equation*}
Then, we test the above equation by $\fhi_t$. Integrating by parts 
and using the Lipschitz continuity of $h$, we deduce
\begin{equation}\label{re2:12}
  \frac 12 \ddt \norma{\fhi_t}^2 
%   + m \norma{\fhi_t}^2
   + \io \nabla \mu_t \cdot \nabla \fhi_t 
   \le c \big( \| \fhi_t \|^2 + \| \sigma_t \|^2 \big).
\end{equation}
Next, differentiating \eqref{CH2b} in time and testing 
the result by $-\Delta \fhi_t$, we get
\begin{equation}\label{re2:13}
  \io \nabla \mu_t \cdot \nabla \fhi_t 
   = \norma{\Delta \fhi_t}^2
   - \io F_1''(\fhi)\fhi_t \Delta \fhi_t
   - \lambda \norma{\nabla \fhi_t}^2
   + \chi \int_\Omega \sigma_t \Delta \fhi_t.
\end{equation}
Combining \eqref{re2:12} with \eqref{re2:13} and performing standard
manipulations, it is easy to get
\begin{equation}\label{re2:14}
  \frac 12 \ddt \norma{\fhi_t}^2 
   + \frac12 \norma{\Delta \fhi_t}^2
   \le c \big( \| \fhi_t \|_V^2 + \| \sigma_t \|^2 \big)
   + \io F_1''(\fhi)\fhi_t \Delta \fhi_t
\end{equation}
and the last term can be controlled as follows:
\begin{align}\no
  \io F_1''(\fhi)\fhi_t \Delta \fhi_t
   & \le \| F_1''(\fhi) \|_{4}
    \| \fhi_t \|_{4} \| \Delta \fhi_t \| \\
 \label{re2:15}
  & \le c \| F_1''(\fhi) \|_{{4}}^2 \| \fhi_t \|_{V}^2 
     + \frac14 \| \Delta \fhi_t \|^2
     \le c \| \fhi_t \|_{V}^2 
     + \frac14 \| \Delta \fhi_t \|^2, 
\end{align}
the last inequality following from \eqref{proof:2d:F}. Hence, replacing 
\eqref{re2:15} into \eqref{re2:14}, using the known regularity properties
\eqref{reg:strong:2d:1} and \eqref{reg:strong:2d:3}, we deduce
\begin{equation}\label{re2:16}
  \| \fhi_t \|_{\LIH} + \| \fhi_t \|_{\LDHD} \le c,
\end{equation}
provided $\fhi_t$ lies in $H$ at the initial time. As before, 
this property has to be read by formally evaluating \eqref{CH1b}
at the time $t=0$. Then, by a direct check one can verify that 
this corresponds exactly to the condition on 
$\mu_0$ postulated in \eqref{init:reg:2}.

Then, viewing \eqref{CH1b} as a family of time-dependent 
elliptic equations whose \rhs s  lie in $\LIH \cap L^2(0,T;V)$
due to \eqref{re2:16}, \eqref{reg:strong:2d:3}
and the Lipschitz continuity of $h$, we deduce
\begin{equation*}
%\label{re2:17}
  \| \mu \|_{L^\infty(0,T;H^2(\Omega))} 
     + \| \mu \|_{L^2(0,T;H^3(\Omega))} \le c.
\end{equation*}
Note that the above, by Sobolev's embeddings, also gives
\begin{align}\label{mu:infty}
   \norma{\mu}_{L^\infty (Q)} \leq c.
\end{align}
Next, to improve the regularity of $\fhi$, we rewrite
\eqref{CH2b} as 
\begin{equation*}
%\label{ellip:11}
  - \Delta \fhi = \mu - F_1'(\fhi) + \lambda\fhi + \chi\sigma.
\end{equation*}
Then, recalling \eqref{fhi:Linftyq} and \eqref{proof:2d:F}, a simple check permits us to
verify that the above \rhs\ lies (at least) in $\LIV$. Hence, by
elliptic regularity we deduce also 
\begin{align}\label{fhiinH3}
   \norma{\fhi}_{L^\infty(0,T;H^3(\Omega))} \leq c.
\end{align}
Next, to get the $L^\infty$-bound of $\sigma$, we
come back to \eqref{nutrb}, which, rearranging, can be rewritten as
\begin{equation*}
%\label{re2:21}
  \sigma_t - \Delta \sigma 
    = - \chi (\nabla \sigma \cdot \nabla \fhi + \sigma \Delta \fhi ) 
      + \beta(\fhi)( \kappa_0\sigma - \kappa_\infty \sigma^2 ) =: G.
\end{equation*}
We now claim that $G \in \L\infty  H$. To check this, we consider only the 
two cross-diffusion terms, the other ones being simpler to deal with.
Indeed, we first observe that
\begin{align*}
   \norma{\nabla \sigma \cdot \nabla \fhi}_{\L\infty H} 
   \leq c \norma{\nabla \sigma }_{\L\infty H} \norma{\nabla \fhi }_{L^\infty(Q)}
   \leq c,
\end{align*}
thanks to \eqref{fhiinH3}, \eqref{reg:strong:2d:3} and Sobolev's embeddings.
Analogously,
\begin{align*}
  \norma{\sigma \Delta \fhi}_{\L\infty H} 
   \leq c \norma{\sigma }_{\L\infty {\Lx4}} \norma{\Delta \fhi }_{\L\infty {\Lx4}}
   \leq c.
\end{align*}
Then, recalling the assumption \eqref{init:sigmareg} on the initial datum,
by an application of \cite[Thm.~7.1, p.~181]{LSU} we readily obtain \eqref{sig:infty}.
Finally, the above regularity allows us to obtain the separation property \eqref{separation}.
To this aim, we go back to the expression \eqref{CH2c} 
and notice that the \rhs, thanks to \eqref{mu:infty}, \eqref{fhiinH3} and \eqref{sig:infty}, is 
now bounded in the $L^\infty(Q)$-norm. Hence, \eqref{separation} 
can be obtained by reasoning exactly as in the proof of \cite[Thm.~2.2]{CSS}.
This concludes the proof.
\end{proof}
\beos\label{su:rego}
With the separation property \eqref{separation} at disposal, the singular character
of $F_1'$ at $\pm 1$ is essentially lost and the term $F_1'(\fhi)$ in \eqref{CH2b} behaves 
like a smooth function of $\fhi$ with controlled growth. Thanks to this fact,
the regularity of solutions may be further improved at least as far as 
the nonlinear terms (like $h$, or $F_1$ itself) satisfy additional
regularity properties (e.g., $C^k$ for large $k$).
\eddos

%%%%%%%%%%%%%%%%%%%%%%%%%%%%%%%%%%%%%%%%%%%%%%%%%%%%%%%%%%%%%%%%%%%%%%%%%%%%%%%%%%%%%%%%%%%%55

\section{Uniqueness of strong solutions}
\label{SEC:UQ}

This section is devoted to the proof of Theorem~\ref{THM:UNIQ:2d}.
We first recall that in Subsection~\ref{SEC:NOT} was introduced the 
operator $\N:V_0'\to V_0$ representing, in a suitable weak sense, 
the inverse of (minus) the Neumann Laplacian acting on the 
functions with zero spatial average. 
Moreover, as anticipated in Remark~\ref{oss:altuniq},
we just consider the case $d=3$, noting that the conditions may be relaxed
in the two-dimensional setting.

\begin{proof}[Proof of Theorem \ref{THM:UNIQ:2d}]
Let us assume to have a couple of solutions $(\fhi_1,\mu_1,\sigma_1)$ 
and $(\fhi_2,\mu_2,\sigma_2)$ fulfilling the assumptions of the theorem and 
let us correspondingly set
\begin{align}\notag
  \fhi & := \fhi_1-\fhi_2, \qquad 
   \mu := \mu_1-\mu_2, \qquad 
   \sigma  := \sigma_1-\sigma_2,\\
 \label{not:diff}
   S_i & := S(\fhi_i ,\sigma_i)~~\text{for $i=1,2,$}
 \qquad \fhi_0  := \fhi_{0,1}-\fhi_{0,2},
 \qquad \sigma_0  := \sigma_{0,1}-\sigma_{0,2}.
\end{align}
Then, under the assumptions of the 
theorem the triplet $(\fhi,\mu,\sigma)$ turns out to solve the system
\begin{alignat}{2}\label{CH1:cd}
  & \fhi_t = \Delta \mu + (S_1-S_2) && \qquad \text{ in }\, Q,\\
 \label{CH2:cd}
  & \mu = - \Delta \fhi + \big(f(\fhi_1)-f(\fhi_2)\big) - \chi\sigma&& \qquad \text{ in }\, Q,\\
 \label{nutr:cd}
  & \sigma_t - \Delta \sigma + \chi \dive (\sigma \nabla \fhi_1 + \sigma_2 \nabla \fhi )
    = \kappa_0 \sigma - \kappa_\infty \sigma (\sigma_1 + \sigma_2) && \qquad \text{ in }\, Q,\\
 \label{BC:cd}
  & \dn \fhi = \dn \mu = \dn \sigma = 0&& \qquad \text{ on } \Sigma,\\
 \label{init:cd}
  & \fhi|_{t=0} = \fhi_0, \quad
 \sigma|_{t=0} = \sigma_0 
  &&\qquad \text{ in } \Omega.
\end{alignat}
We then start by integrating \eqref{CH1:cd} over $\Omega$ to find that
\begin{align}\label{uq:1}
  \fhi_\Omega' = ( S_1 - S_2 )\OO
   = \frac 1 {|\Omega|} \io \big(S(\fhi_1,\sigma_1)-S(\fhi_2,\sigma_2)\big).
\end{align}
Testing the above by $\fhi_\Omega$ and using Young's inequality 
with the Lipschitz continuity of $S$, we easily deduce
\begin{align}\label{uq:2}
  \frac 12 \ddt |\fhi_\Omega|^2 \leq |\fhi_\Omega|^2  + c (\norma{\fhi}^2+\norma{\sigma}^2).
\end{align}
Next, we subtract \eqref{uq:1} from \eqref{CH1:cd} and test the resulting 
equality by $\N(\fhi - \fhi_\Omega)$ obtaining that
\begin{align} \notag
  & \frac 12 \ddt \norma{\fhi - \fhi_\Omega}^2_*
    + \io (\fhi- \fhi_\Omega)(\mu -\mu_\Omega)
   = \io \big( (S_1-S_2) - (S_1 - S_2)_\Omega \big) \, \N(\fhi-\fhi_\Omega)\\ 
 \label{uq:3}
  & \quad \leq c \norma{\fhi}^2
   + c \norma{\sigma}^2 
   + c \norma{\fhi - \fhi_\Omega}^2_*.
\end{align}
Let us now point out that, by the Poincar\'e--Wirtinger inequality and some
elementary interpolation, 
\begin{align} \no
  c \| \fhi \|^2 
   & \le c \big( \| \fhi- \fhi\oo \|^2 + |\fhi\oo|^2 \big)
   \le c \| \fhi- \fhi\oo \|_V \| \fhi- \fhi\oo \|_* + c|\fhi\oo|^2 \\
 \label{uq:3s}
   & \le \delta \| \nabla \fhi\|^2 + c_\delta \| \fhi- \fhi\oo \|_*^2 + c|\fhi\oo|^2 
\end{align}
for ``small'' $\delta>0$ and correspondingly ``large'' $c_\delta>0$. 

Next, noting that $\io \mu_\Omega(\fhi - \fhi_\Omega) = 0$, we may
use \eqref{CH2:cd} to obtain that
\begin{align} \label{uq:4}
  \io (\fhi- \fhi_\Omega)(\mu -\mu_\Omega)
   = \norma{\nabla \fhi}^2
    + \io (\fhi - \fhi_\Omega)(f(\fhi_1) - f (\fhi_2)) 
    - \chi \io (\fhi - \fhi_\Omega) \sigma.
\end{align}
Using also \eqref{uq:3s}, we deduce 
\begin{equation}\label{uniq:11}
  \chi \bigg| \io (\fhi - \fhi_\Omega) \sigma \bigg|
   \leq c \norma{\sigma}^2 + c \norma{\fhi-\fhi_\Omega}^2  
   \leq \d \norma{\nabla \fhi}^2 + \cd \norma{\fhi-\fhi_\Omega}_*^2 +c \norma{\sigma}^2,
\end{equation}
for $\delta>0$ and $c_\delta>0$ as above. 

Summing \eqref{uq:2} with \eqref{uq:3} and using \eqref{uq:4}, \eqref{uniq:11}
and \eqref{uq:3s} again, we deduce
\begin{align*} \notag
  & \frac 12 \ddt \big( \norma{\fhi - \fhi_\Omega}^2_* + | \fhi\OO |^2 \big)
    + \io (\fhi - \fhi_\Omega) (f(\fhi_1) - f (\fhi_2)) 
    + (1 - 2\delta) \| \nabla \fhi \|^2\\
% \label{uq:3b}
  & \quad\quad \leq c_\delta \big( \norma{\fhi - \fhi_\Omega}^2_* 
    + \norma{\sigma}^2 + |\fhi\OO|^2 \big).
\end{align*}
Next, decomposing $f$ into its monotone and remainder parts, 
it is not difficult to get
\begin{align} \notag
  &    \frac 12 \ddt \big( \norma{\fhi - \fhi_\Omega}^2_* + | \fhi\OO |^2 \big)
    + (1 - 3 \delta) \| \nabla \fhi \|^2\\
 \label{uq:3c}
  & \quad\quad \leq c_\delta \big( \norma{\fhi - \fhi_\Omega}^2_* + \norma{\sigma}^2 + |\fhi\OO|^2 \big)
     + \io (F_1'(\fhi_1) - F_1'(\fhi_2)) \fhi_\Omega.
\end{align}

\smallskip

\noindent%
We now move to the estimation of $\sigma$. 
Integrating \eqref{nutr:cd} over $\Omega$ we obtain
\begin{equation} \label{uniq:a1}
  (\sigma\oo)_t = \kappa_0 \sigma\oo - \kappa_\infty \big( \sigma_1^2 - \sigma_2^2 \big)\oo.
\end{equation}
Subtracting the above from \eqref{nutr:cd}, we then get
\begin{equation} \label{uniq:a2}
  (\sigma - \sigma\oo)_t 
   - \Delta \sigma 
    + \chi \dive (\sigma \nabla \fhi_1 + \sigma_2 \nabla \fhi )
    = \kappa_0 ( \sigma - \sigma\oo)
    - \kappa_\infty \big( \sigma_1^2 - \sigma_2^2 - (\sigma_1^2)\oo + (\sigma_2^2)\oo \big).
\end{equation}
Testing \eqref{uniq:a1} by $\sigma\oo$, it is not difficult to obtain
\begin{align} \notag
    \frac12 \ddt |\sigma\oo|^2  
    & \le \kappa_0 |\sigma\oo|^2
     + \frac{\kappa_\infty}{|\Omega|} | \sigma\oo | \io | \sigma | |\sigma_1 + \sigma_2|\\
 \no
    & \le \kappa_0 |\sigma\oo|^2 
     + c| \sigma\oo | \big( \| \sigma - \sigma\oo \| + | \sigma\oo | \big) \| \sigma_1 + \sigma_2 \| \\
 \label{uniq:a3}
   & \le \eta \| \sigma - \sigma\oo \|^2 
     + c_\eta \big( 1 + \|\sigma_1 \|^2 + \| \sigma_2 \|^2 \big) | \sigma\oo |^2,
\end{align}
where $\eta>0$ denotes a positive constant whose value will be 
fixed at the end.
Next, we test \eqref{uniq:a2} by $\calN ( \sigma - \sigma\oo )$ to deduce that
\begin{align} \notag
  & \frac 12 \ddt \norma{\sigma - \sigma_\Omega}^2_* 
   + \| \sigma - \sigma_\Omega \|^2 
   \leq \chi \io \sigma \nabla \fhi_1 \cdot \nabla \calN (\sigma - \sigma_\Omega)
   + \chi \io \sigma_2 \nabla \fhi \cdot \nabla \calN (\sigma - \sigma_\Omega)\\
 \label{uniq:a4}
  & \quad\quad + \kappa_0 \norma{\sigma - \sigma_\Omega}^2_*
  - \kappa_\infty \io \big( \sigma_1^2 - \sigma_2^2 - (\sigma_1^2)\oo + (\sigma_2^2)\oo \big) 
   \calN (\sigma - \sigma\oo).
\end{align}
As for the \rhs, we first notice that 
\begin{align} \no
  & \chi \io \sigma \nabla \fhi_1 \cdot \nabla \calN (\sigma - \sigma_\Omega)
   \le c \| \sigma \| \| \nabla \fhi_1 \|_{\infty} \| \nabla \calN (\sigma - \sigma_\Omega) \|\\
  \no
   & \qquad\quad
    \le c \| \sigma \| \| \nabla \fhi_1 \|_{\infty} \| \sigma - \sigma_\Omega \|_*
    \le \eta \| \sigma \|^2
   +  c_\eta \| \nabla \fhi_1 \|^2_{\infty} \| \sigma - \sigma_\Omega \|_*^2\\
  \label{uniq:a5} 
   & \qquad\quad
    \le 2\eta \| \sigma - \sigma\oo \|^2 + c\eta |\sigma\oo|^2
    + c_\eta \| \fhi_1 \|^2_{W^{2,6}(\Omega)} \| \sigma - \sigma_\Omega \|_*^2.
\end{align}
To control the second integral, several strategies are possible,
leading to different assumptions on $\sigma_2$. Under the conditions in the statement,
we may proceed by noting that
\begin{align} \no
  & \chi \io \sigma_2 \nabla \fhi \cdot \nabla \calN (\sigma - \sigma_\Omega)
   \le c \| \sigma_2 \|_{6} \| \nabla \fhi \| 
       \| \nabla \calN (\sigma - \sigma_\Omega) \|_{{3}} \\
 \no
  & \qquad\quad \le c \| \sigma_2 \|_{6} \| \nabla \fhi \| 
       \| \calN (\sigma - \sigma_\Omega) \|_{V}^{1/2}
           \| \calN (\sigma - \sigma_\Omega) \|_{H^2(\Omega)}^{1/2}\\
 \no
  & \qquad\quad \le c \| \sigma_2 \|_{6} \| \nabla \fhi \| 
       \| \sigma - \sigma_\Omega \|_{*}^{1/2}
           \| \sigma - \sigma_\Omega \|^{1/2}\\
 \label{uniq:a6b} 
  & \qquad\quad \le \eta \| \nabla \fhi \|^2
   + \eta \| \sigma - \sigma_\Omega \|^2
  + c_\eta \| \sigma_2 \|_{6}^4 \| \sigma - \sigma_\Omega \|_{*}^2.
\end{align}
Next, we move to the last term in \eqref{uniq:a4}, which can be treated as 
follows:
\begin{align} \no
  & - \kappa_\infty \io \big( \sigma_1^2 - \sigma_2^2 - (\sigma_1^2)\oo + (\sigma_2^2)\oo \big) 
   \calN (\sigma - \sigma\oo)\\
 \no  
  & \qquad\quad
    \le c \big\| \sigma_1^2 - \sigma_2^2 - (\sigma_1^2)\oo + (\sigma_2^2)\oo \|_{1}
    \| \calN (\sigma - \sigma\oo) \|_{\infty}\\
  \no
   & \qquad\quad
    \le c \| \sigma \| \big( \| \sigma_1 \| + \| \sigma_2 \| \big)
    \| \calN(\sigma - \sigma\oo) \|_{V}^{1/2} \| \calN(\sigma - \sigma\oo) \|_{H^2(\Omega)}^{1/2}\\
  \no
   & \qquad\quad
    \le c \big( \| \sigma - \sigma\oo \| + | \sigma\oo | \big)
       \big( \| \sigma_1 \| + \| \sigma_2 \| \big)
    \| \sigma - \sigma\oo \|_{*}^{1/2} \| \sigma - \sigma\oo \|^{1/2}\\
  \label{uniq:a7} 
   & \qquad\quad
    \le c_\eta \big( 1 + \| \sigma_1 \|^4 + \| \sigma_2 \|^4 \big)
    \| \sigma - \sigma\oo \|_{*}^2 + \eta \| \sigma - \sigma\oo \|^2
      + c | \sigma\oo |^2.
\end{align}
Next, we replace \eqref{uniq:a5}, \eqref{uniq:a6b} and
\eqref{uniq:a7} into \eqref{uniq:a4} to deduce that
\begin{align*} \notag
  & \frac 12 \ddt \norma{\sigma - \sigma_\Omega}^2_* 
   + ( 1 - 4 \eta )\| \sigma - \sigma_\Omega \|^2 \\
% \label{uniq:b1a}
  & \quad\quad  
   \le c_\eta \big( 1 + \| \sigma_1 \|^4 + \| \sigma_2 \|_{6}^4 
           + \| \fhi_1 \|^2_{W^{2,6}(\Omega)} \big)\| \sigma - \sigma_\Omega \|_*^2
    + \eta \| \nabla \fhi \|^2 + c | \sigma\oo |^2.
\end{align*}
Adding \eqref{uniq:a3} to the above relation gives
\begin{align} \notag
  & \frac12 \ddt \big( \norma{\sigma - \sigma_\Omega}^2_* + |\sigma\oo|^2  \big)
   + ( 1 - 5 \eta )\| \sigma - \sigma_\Omega \|^2
   \le c_\eta \big( 1 + \|\sigma_1 \|^2 + \| \sigma_2 \|^2 \big) | \sigma\oo |^2 \\
 \label{uniq:b2b}
  & \quad\quad  
   + c_\eta \big( 1 +  \| \sigma_1 \|^4 + \| \sigma_2 \|_{6}^4
           + \| \fhi_1 \|^2_{W^{2,6}(\Omega)} \big)\| \sigma - \sigma_\Omega \|_*^2
    + \eta \| \nabla \fhi \|^2.
\end{align}
We then take $\delta=1/6$ in \eqref{uq:3c} and multiply that relation by $\zeta>0$
to be chosen below. Finally, we add the result to \eqref{uniq:b2b}. This yields
\begin{align*} \notag
  & \frac12 \ddt \big( \norma{\sigma - \sigma_\Omega}^2_* + |\sigma\oo|^2
   + \zeta \norma{\fhi - \fhi_\Omega}^2_* + \zeta | \fhi\OO |^2 \big)
   + \frac\zeta2 \| \nabla \fhi \|^2
   + ( 1 - 5 \eta )\| \sigma - \sigma_\Omega \|^2\\
 \no
  & \mbox{}~~~~~ \le c_\eta \big( 1 + \|\sigma_1 \|^2 + \| \sigma_2 \|^2 \big) | \sigma\oo |^2 \\
 \no
  & \mbox{}~~~~~\quad + c_\eta \big( 1  + \| \sigma_1 \|^4 + \| \sigma_2 \|_{6}^4 
           + \| \fhi_1 \|^2_{W^{2,6}(\Omega)} \big)\| \sigma - \sigma_\Omega \|_*^2
    + \eta \| \nabla \fhi \|^2\\
% \label{uniq:b3b}
  & \mbox{}~~~~~\quad + c \zeta \norma{\fhi - \fhi_\Omega}^2_* 
    + c_1 \zeta \norma{\sigma - \sigma\oo}^2 + c \zeta | \sigma\oo |^2
    + c \zeta |\fhi\OO|^2
     + \zeta \bigg| \io (F_1'(\fhi_1) - F_1'(\fhi_2)) \fhi_\Omega \bigg|,
\end{align*}
where $c_1>0$ is a computable constant independent of $\zeta$ and $\eta$.
Then, choosing first $\zeta\le \min\{1,1/2c_1\}$, we arrive at
\begin{align*} \notag
  & \frac12 \ddt \big( \norma{\sigma - \sigma_\Omega}^2_* + |\sigma\oo|^2
   + \zeta \norma{\fhi - \fhi_\Omega}^2_* + \zeta | \fhi\OO |^2 \big)
   + \frac\zeta2 \| \nabla \fhi \|^2
   + \Big( \frac12 - 5 \eta \Big )\| \sigma - \sigma_\Omega \|^2\\
 \no
  & \mbox{}~~~~~ \le c_\eta \big( 1 + \|\sigma_1 \|^2 + \| \sigma_2 \|^2 \big) | \sigma\oo |^2 \\
 \no
  & \mbox{}~~~~~\quad
    + c_\eta \big( 1 + \| \sigma_1 \|^4 + \| \sigma_2 \|_{6}^4
           + \| \fhi_1 \|^2_{W^{2,6}(\Omega)} \big)\| \sigma - \sigma_\Omega \|_*^2
    + \eta \| \nabla \fhi \|^2\\
% \label{uniq:b4b}
  & \mbox{}~~~~~\quad
    + c \norma{\fhi - \fhi_\Omega}^2_* 
    + c | \sigma\oo |^2
    + c |\fhi\OO|^2
    + \bigg| \io (F_1'(\fhi_1) - F_1'(\fhi_2)) \fhi_\Omega \bigg|.
\end{align*}
Next, choosing $\eta\le \min\{1/20,\zeta/4\}$, we get
\begin{align} \notag
  & \frac12 \ddt \big(\norma{\sigma - \sigma_\Omega}^2_* + |\sigma\oo|^2
   + \zeta \norma{\fhi - \fhi_\Omega}^2_* + \zeta | \fhi\OO |^2 \big)
   + \frac\zeta4 \| \nabla \fhi \|^2
   + \frac14 \| \sigma - \sigma_\Omega \|^2\\
 \no
  & \mbox{}~~~~~
   \le c \big( 1 + \|\sigma_1 \|^2 + \| \sigma_2 \|^2 \big) | \sigma\oo |^2
    + c \big( 1 + \| \sigma_1 \|^4 + \| \sigma_2 \|_{6}^4 
           + \| \fhi_1 \|^2_{W^{2,6}(\Omega)} \big)\| \sigma - \sigma_\Omega \|_*^2\\
 \label{uniq:b5b}
  &  \mbox{}~~~~~~~~~~
    + c \norma{\fhi - \fhi_\Omega}^2_* 
    + c | \sigma\oo |^2
    + c |\fhi\OO|^2
    + \bigg| \io (F_1'(\fhi_1) - F_1'(\fhi_2)) \fhi_\Omega \bigg|.
\end{align}
To obtain a contractive estimate, we need to manage the last term. This is
treated in two different ways depending on the assumption on $h$. Indeed,
if $h$ is a constant, we may proceed as in \cite{GGM} since in that
case the ODE relation \eqref{uq:1} reduces to 
\begin{align}\label{uq:1b}
  \fhi_\Omega' + m \fhi\OO = 0.
\end{align}
For constant $h$ we may then proceed by noting that 
\begin{equation} \label{uniq:22}
  \bigg| \io (F_1'(\fhi_1) - F_1'(\fhi_2)) \fhi_\Omega \bigg|
   \leq \big( \norma{F_1'(\fhi_1)}_{1} + \norma{F_1'(\fhi_2)}_{1} \big)
    |\fhi_\Omega|.
\end{equation}
Then, testing \eqref{uq:1b} by the sign of $\fhi\OO$
(this standard procedure may be also justified by approximation), 
summing the result to \eqref{uniq:b5b},
and using \eqref{uniq:22}, we deduce 
\begin{align*} \notag
  & \frac12 \ddt \big( \norma{\sigma - \sigma_\Omega}^2_* + |\sigma\oo|^2
   + \zeta \norma{\fhi - \fhi_\Omega}^2_* + \zeta | \fhi\OO |^2 + |\fhi\oo| \big)
   + \frac\zeta4 \| \nabla \fhi \|^2
   + \frac14 \| \sigma - \sigma_\Omega \|^2\\
 \no
  & \mbox{}~~~~~
   \le c \big( 1 + \|\sigma_1 \|^2 + \| \sigma_2 \|^2 \big) | \sigma\oo |^2
    + c \big( 1 + \| \sigma_1 \|^4 + \| \sigma_2 \|_{6}^4 
           + \| \fhi_1 \|^2_{W^{2,6}(\Omega)} \big)\| \sigma - \sigma_\Omega \|_*^2\\
% \label{uniq:b6b}
  &  \mbox{}~~~~~~~~~~
    + c \norma{\fhi - \fhi_\Omega}^2_* 
    + c |\fhi\OO|^2
    + \big( \norma{F_1'(\fhi_1)}_{1} + \norma{F_1'(\fhi_1)}_{1} \big)|\fhi_\Omega|,
\end{align*}
where we recall that $\zeta$ is a positive constant whose value has already been fixed. 
Next, we observe that, in view of \eqref{reg:5} and \eqref{cond:un1}-\eqref{cond:un3},
we may apply Gr\"onwall's lemma, which gives the statement (and, more 
generally, the continuous dependence estimate \eqref{cont:dep1}).
This concludes the analysis of the first case.

On the other hand, when $h$ is nonlinear, it does not seem to 
be possible to proceed as above. For this reason, we need to 
provide a different control of the last integral term in \eqref{uniq:b5b}.
Namely, we may first notice that a simple computation shows, as we are assuming $F \in C^2(-1,1)$, that
\begin{equation*}
% \label{uniq:31}
  F_1'(\fhi_1) - F_1'(\fhi_2)
   = \ell \fhi, \quext{with}~~
  \ell = \int_0^1 F_1''( s\fhi_1 + (1-s) \fhi_2)\,\dis.
\end{equation*}
Consequently, by the Young, H\"older and Poincar\'e--Wirtinger inequalities we infer 
\begin{align*} \no
  & \bigg| \io (F_1'(\fhi_1) - F_1'(\fhi_2)) \fhi_\Omega \bigg| 
  = \bigg| \fhi\OO \io \ell \fhi \bigg| 
   \le | \fhi\OO | \| \fhi \| \| \ell \|\\
 \no
  & \quad\quad \le | \fhi\OO | \big( \| \fhi - \fhi\OO \| + |\fhi\OO| \big)  \| \ell \|
    \le c | \fhi\OO | \big( \| \nabla \fhi \| + |\fhi\OO| \big) \| \ell \|\\
  & \quad\quad\le \frac\zeta8 \| \nabla \fhi \|^2  
   + c \big( 1 + \| \ell \|^2 \big) |\fhi\OO|^2 
  \le \frac\zeta8 \| \nabla \fhi \|^2  + c |\fhi\OO|^2 
   \big( 1 + \| F_1''(\fhi_1) \|^2 + \| F_1''(\fhi_2) \|^2 \big), 
\end{align*}
where the value of $\zeta$ was assigned before (and the last constants
$c>0$ also depend on it).
Replacing this into \eqref{uniq:b5b}, we then get 
\begin{align*} \notag
  & \frac12 \ddt \big(\norma{\sigma - \sigma_\Omega}^2_* + |\sigma\oo|^2
   + \zeta \norma{\fhi - \fhi_\Omega}^2_* + \zeta | \fhi\OO |^2 \big)
   + \frac\zeta8 \| \nabla \fhi \|^2
   + \frac14 \| \sigma - \sigma_\Omega \|^2\\
 \no
  & \mbox{}~~~~~
   \le c \big( 1 + \|\sigma_1 \|^2 + \| \sigma_2 \|^2 \big) | \sigma\oo |^2
    + c \big( 1 + \| \sigma_1 \|^4 + \| \sigma_2 \|_{6}^4 
           + \| \fhi_1 \|^2_{W^{2,6}(\Omega)} \big)\| \sigma - \sigma_\Omega \|_*^2\\
% \label{uniq:b7b}
  &  \mbox{}~~~~~~~~~~
    + c \norma{\fhi - \fhi_\Omega}^2_* 
    + c |\fhi\OO|^2 
       \big( 1 + \| F_1''(\fhi_1) \|^2 + \| F_1''(\fhi_2) \|^2 \big).
\end{align*}
Once again, using also the additional assumption \eqref{cond:un4}, 
Gr\"onwall's lemma gives the thesis.

Finally, with reference to Remark~\ref{oss:altuniq}, we
notice that the second integral on the \rhs\ of \eqref{uniq:a4} 
can alternatively managed in the following way employing H\"older's inequality:
\begin{align*} \no
  & \chi \io \sigma_2 \nabla \fhi \cdot \nabla \calN (\sigma - \sigma_\Omega)
   \le c \| \sigma_2 \|_{3+\eps} \| \nabla \fhi \| 
       \| \nabla \calN (\sigma - \sigma_\Omega) \|_{\frac{6+2\eps}{1+\eps}}\\
% \label{uniq:a6a} 
  & \qquad \le \delta \| \nabla \fhi \|^2 
    + c_\delta \| \sigma_2 \|_{L^\infty(0,T;L^{3+\eps}(\Omega))}^2
          \| \nabla \calN (\sigma - \sigma_\Omega) \|_{\frac{6+2\eps}{1+\eps}}^2,
\end{align*}
where $\eps>0$ is arbitrarily small (but fixed)
and the \rhs\ can be managed by using interpolation and accordingly adjusting the 
magnitude of the occurring constants as the interested reader may verify. 
\end{proof}

\bigskip
\noindent
\noindent
{\bf Acknowledgments.}~
This research has been performed in the framework of the MIUR-PRIN Grant 2020F3NCPX 
``Mathematics for industry 4.0 (Math4I4)''. The present paper also benefits from the support of 
the GNAMPA (Gruppo Nazionale per l'Analisi Matematica, la Probabilit\`a e le loro Applicazioni)
of INdAM (Istituto Nazionale di Alta Matematica).

\bigskip

%%%%%%%%%%%%%%%%%%%%%%%%%%%%%%%%%%%%%%%%%%%%%%%%%%%%%%%%%%%%%%%%%%%%%%%%%%%%%%%%%%
%%%%%%%%%%%%%%%%%%%%%%%%%%%%%%%%%%% References %%%%%%%%%%%%%%%%%%%%%%%%%%%%%%%%%%%
%%%%%%%%%%%%%%%%%%%%%%%%%%%%%%%%%%%%%%%%%%%%%%%%%%%%%%%%%%%%%%%%%%%%%%%%%%%%%%%%%%

%%%%%%%%%%%%%%%%%%%%%%%%%%%%%%%%%%%%%%%%%%%%%%%%%%%%%%%%%%%%%%%%%%%%%%%%%%%%%%%%%%%%%%%%%%%%%

%
%
%\vspace{15mm}
%
%\noindent%
%{\bf First author's address:}\\[1mm]
%Elisabetta Rocca\\
%Dipartimento di Matematica, Universit\`a degli Studi di Pavia\\
%Via Ferrata, 5,~~I-27100 Pavia,~~Italy\\
%E-mail:~~{\tt elisabetta.rocca@unipv.it} 
%
%\vspace{4mm}
%
%\noindent%
%{\bf Second author's address:}\\[1mm]
%Giulio Schimperna\\
%Dipartimento di Matematica, Universit\`a degli Studi di Pavia\\
%Via Ferrata, 5,~~I-27100 Pavia,~~Italy\\
%E-mail:~~{\tt giusch04@unipv.it}
%
%\vspace{4mm}
%
%\noindent%
%{\bf Third author's address:}\\[1mm]
%Andrea Signori\\
%Dipartimento di Matematica, Universit\`a degli Studi di Pavia\\
%Via Ferrata, 5,~~I-27100 Pavia,~~Italy\\
%E-mail:~~{\tt andrea.signori01@unipv.it} 


\begin{thebibliography}{99}
 
 
\bibitem{Ciarletta}
A.~Agosti, P.F.~Antonietti, P.~Ciarletta, M.~Grasselli, and  M.~Verani, 
{\sl A Cahn--Hilliard-type equation with application to tumor growth dynamics}, 
Math.\ Methods Appl.\ Sci., 
{\bf 40} (2017), 7598--7626.

 
\bibitem{At}
 H.~Attouch,
 ``Variational Convergence for Functions and Operators'',
 Pitman,
 London,
 1984.
 
\bibitem{Ba}
 V.~Barbu,
 ``Nonlinear Semigroups and Differential Equations in Banach Spaces''.
 Noordhoff, Leiden, 1976.

 
\bibitem{Br}
 H.~Br\'ezis,
 ``Op\'erateurs Maximaux Monotones et S\'emi-groupes de Contractions
    dans les Espaces de Hilbert''.
 North-Holland Math.\ Studies {\bf 5},
 North-Holland,
 Amsterdam,
 1973.
 
 \bibitem{BPCCZ} 
 F. Bubba, B. Perthame, D. Cerroni, P. Ciarletta, and P. Zunino, 
 {\sl A coupled 3D-1D multiscale Keller--Segel model of chemotaxis and its application to cancer invasion},
 {Discrete Contin.\ Dyn.\ Syst.\ Ser.~S.}, 
 to appear (2022).


\bibitem{CH}
 J.W.~Cahn and J.E.~Hilliard, 
 {\sl Free energy of a nonuniform system. I.~Interfacial free energy},
 J.~Chem.\ Phys., 
 {\bf 28} (1958),
 258--267.

\bibitem{CGH}
P.~Colli, G.~Gilardi,  and D.~Hilhorst, 
{\sl On a Cahn--Hilliard type phase field model related to tumor growth}, 
{Discrete Contin.\ Dyn.\ Syst.}, 
{\bf 35} (2015), 2423--2442.

\bibitem{CGRS1}
P.~Colli, G.~Gilardi, E.~Rocca, and J.~Sprekels, 
{\sl Vanishing viscosities and error estimate for a Cahn--Hilliard type phase-field system related to tumor growth}, 
{Nonlinear Anal.\ Real World Appl.}, {\bf 26} (2015), 93--108.

\bibitem{CGRS2}
P.~Colli, G.~Gilardi, E.~Rocca, and  J.~Sprekels, 
{\sl Asymptotic analyses and error estimates for a Cahn--Hilliard type phase field system modelling tumor growth}, 
{Discrete Contin.\ Dyn.\ Syst.\ Ser.~S.}, 
{\bf 10} (2017), 37--54.
 
\bibitem{CL10}
V.~Cristini and  J.~Lowengrub, 
``Multiscale modeling of cancer. An integrated experimental and mathematical modeling approach",
Cambridge Univ.~Press, 2010.
 

\bibitem{CLLW09}
V.~Cristini, X.~Li, J.S.~Lowengrub, and  S.M.~Wise, 
{\sl Nonlinear simulations of solid tumor growth using a mixture model: {invasion} and branching},
{J.~Math.\ Biol.}, 
{\bf 58} (2009), 723--763. 
 
\bibitem{DFRSS17}
M.~Dai, E.~Feireisl, E.~Rocca, G.~Schimperna, and  M.~Schonbek, 
{\sl Analysis of a diffuse interface model for multispecies tumor growth}, 
{Nonlinearity}, 
{\bf 30} (2017), 1639--1658.

\bibitem{FLP}
E. Feireisl, P. Lauren\c cot, and  H. Petzeltov\'a,
{\sl On convergence to equilibria for the Keller–Segel chemotaxis model},
{J. Differential Equations},
{\bf 236} (2007), 551--569.

\bibitem{FGR}
S.~Frigeri, M.~Grasselli,  and E.~Rocca, 
{\sl On a diffuse interface model of tumor growth}, 
{European J.~Appl.\ Math.}, 
{\bf 26} (2015), 215--243.

\bibitem{FLR}
S.~Frigeri, K.F.~Lam, and  E.~Rocca, 
{\sl On a diffuse interface model for tumour growth 
with non-local interactions and degenerate mobilities}. 
In: P.~Colli, A.~Favini, E.~Rocca, G.~Schimperna, J.~Sprekels~(eds.), Solvability,
Regularity, Optimal Control of Boundary Value Problems for PDEs,
pp.~217--254, Springer INdAM Series, Springer, Milan, 2017.

\bibitem{CSS}
P. Colli, A. Signori, and J. Sprekels,
{\sl Optimal control of a phase field system modelling tumor growth with chemotaxis and singular potentials},
{Appl. Math. Optim.}, {\bf 83} (2021), 2017--2049.

\bibitem{EPP}
C. Elbar, B. Perthame, and A. Poulain,
{\sl Degenerate Cahn--Hilliard and incompressible limit of a Keller--Segel model},
arXiv:2112.10394, 2021.

\bibitem{FG}
S. Frigeri and M. Grasselli, 
{\sl Nonlocal Cahn--Hilliard--Navier--Stokes systems with singular potential}, 
{ {Dyn. Partial Differ. Equ.}, {\bf 9}} (2012), 273--304.

\bibitem{FLRS}
S. Frigeri, K.F. Lam, E. Rocca, and G. Schimperna,
{\sl On a multi-species Cahn--Hilliard--Darcy tumor growth model with singular potentials},
{Commun. Math. Sci.}, {\bf 16} (2018), 821--856.



\bibitem{GZ}
H. Gajewski and K. Zacharias,
{\sl Global behaviour of a reaction–diffusion system modelling chemotaxis},
Math. Nachr., {\bf 195} (1998), 77--114.


\bibitem{GLDirichlet}
H.~Garcke and K.F.~Lam, 
{\sl Analysis of a Cahn--Hilliard system with non zero Dirichlet conditions modelling tumour growth with chemotaxis},
{Discrete Contin.\ Dyn.\ Syst.}, 
{\bf 37} (2017), 4277--4308.

\bibitem{GLNeumann}
H.~Garcke and K.F.~Lam, 
{\sl Well-posedness of a Cahn--Hilliard system modelling tumour growth with chemotaxis and active transport}, 
{European J.~Appl.\ Math.}, 
{\bf 28} (2017), 284--316.

 

\bibitem{GAR}
H. Garcke, K.F. Lam, R. N\"urnberg, and E. Sitka,
{\sl A multiphase Cahn--Hilliard--Darcy model for tumour growth with necrosis},
Math. Models Methods Appl. Sci., {\bf 28} (2018), 525--577.

\bibitem{GLSS}
H.~Garcke, K.F. Lam, E.~Sitka, and V.~Styles,
{\sl A {C}ahn--{H}illiard--{D}arcy model for tumour growth with chemotaxis and active transport},
Math. Models Methods in Appl. Sci., {\bf 26} (2016), 1095--1148.


\bibitem{GLR}
H.~Garcke, K.F.~Lam, and  E.~Rocca,
{\sl Optimal control of treatment time in a diffuse interface model for tumour growth},
Appl.\ Math.\ Optim., {\bf 78} (2018), 495--544.


\bibitem{GLS}
H.~Garcke, K.F.~Lam, and A.~Signori,
{\sl On a phase field model of Cahn--Hilliard type for tumour growth with mechanical effects},
{Nonlinear Anal. Real World Appl.}  {\bf 57} (2021), 103192, https://doi.org/10.1016/j.nonrwa.2020.103192.

\bibitem{GLRS}
A. Giorgini, K.F. Lam, E. Rocca, and G. Schimperna, 
{\sl On the existence of strong solutions to the Cahn--Hilliard--Darcy system with mass source},
{\it SIAM J. Math. Anal.}, {\bf 54} (2022), 737--767. 

\bibitem{GGM}
A. Giorgini, M. Grasselli, and A. Miranville,
{\sl The Cahn--Hilliard--Oono equation with singular potential}, Math. Models Methods Appl. Sci., {\bf 27} (2017), %no. 13, 
2485--2510.

\bibitem{GGW}
A. Giorgini, M. Grasselli, and H. Wu, 
{\sl The Cahn--Hilliard--Hele--Shaw system with singular potential}, 
Ann. Inst. H. Poincar\'e Anal. Non Lin\'eaire, {\bf 35} (2018), 1079--1118.

\bibitem{HH}
M.H. Hashim  and A.J. Harfash, 
{\sl Finite element analysis of a Keller--Segel model with additional cross-diffusion and logistic source.
  Part I: Space convergence},
Comput. Math. Appl., {\bf 89} (2021), 44--56.

\bibitem{HZO} 
A.~Hawkins-Daarud, K.G.~van der Zee, and  J.T.~Oden,
{\sl Numerical simulation of a thermodynamically consistent four-species tumor growth model}, 
Int.\ J.~Numer.\ Meth.\ Biomed.\ Engng., 
\textbf{28} (2011), 3--24.

\bibitem{HPZO} 
A.~Hawkins-Daarud, S.~Prudhomme, K.G.~van der Zee, and  J.T.~Oden,
{\sl Bayesian calibration, validation, and uncertainty 
quantification of diffuse interface models of tumor growth}, 
J.~Math.\ Biol., 
\textbf{67} (2013), 
1457--1485.


\bibitem{HMV}
M.A. Herrero, E. Medina, and J.J.L. Vel\'azquez,
{\sl Finite-time aggregation into a single point in a reaction–diffusion system},
Nonlinearity, {\bf 10} (1997), 1739--1754.


\bibitem{HKNZ15} 
D.~Hilhorst, J.~Kampmann, T.N.~Nguyen, and K.G.~van der Zee, 
{\sl Formal asymptotic limit of a diffuse-interface tumor-growth model}, 
{Math.\ Models Methods Appl.\ Sci.}, 
{\bf 25} (2015), 1011--1043. 


\bibitem{H1}
D. Horstmann,
{\sl On the existence of radially symmetric blow-up solutions for the Keller--Segel model},
J. Math. Biol., {\bf 44} (2002), 463--478.

\bibitem{H2}
D. Horstmann,
{\sl From 1970 until now: the Keller--Segel model in chemotaxis and its consequences},
Jahresber. Deutsch. Math.-Verei., {\bf 106} (2004), 51--69.


\bibitem{KS2}
P. Knopf and A. Signori, 
{\sl Existence of weak solutions to multiphase Cahn--Hilliard--Darcy and 
Cahn--Hilliard--Brinkman models for stratified tumor growth with chemotaxis and general source terms},
{Comm. Partial Differential Equations}, (2021). doi.org/10.1080/03605302.2021.1966803. 


\bibitem{Ipo}
E. Ipocoana, 
{\sl On a non-isothermal Cahn--Hilliard model for tumor growth},
J.~Math.\ Anal.\ Appl.,
{\bf 506} (2022), Paper No. 125665, 18 pp.


\bibitem{JL}
W. J\"ager and S. Luckhaus, 
{\sl On explosions of solutions to a system of partial differential equations modelling chemotaxis}, Trans. Amer. Math.
Soc., {\bf 329} (1992), 819--824.


\bibitem{JiangWuZheng14}
J.~Jiang, H.~Wu, and S.~Zheng,
{\sl Well-posedness and long-time behavior of a non-autonomous {C}ahn--{H}illiard--{D}arcy 
system with mass source modeling tumor growth},
{J.~Differential Equations},
{\bf 259} (2015), 
{3032--3077}.


\bibitem{KS}
E.F. Keller and L.A. Segel,
{\sl Model for chemotaxis},
J. Theoret. Biol.,
{\bf 30} (1971), 225--234. 


\bibitem{LSU}
O.\,A. Lady\v{z}enskaja, V.\,A. Solonnikov, and N.\,N. Uralceva,
``Linear and Quasilinear Equations of Parabolic
Type'', Mathematical Monographs, Vol.~{\bf 23}, 
American Mathematical Society, Providence, Rhode Island, 1968.


\bibitem{LowengrubTitiZhao13}
J.S.~Lowengrub, E.~Titi, and K.~Zhao,
{\sl Analysis of a mixture model of tumor growth},
{European J.~Appl.\ Math.},
{\bf 24} (2013),
{691--734}. 


\bibitem{MZ}
A. Miranville and  S. Zelik, 
{\sl Robust exponential attractors for Cahn--Hilliard type equations with singular potentials}, 
{Math. Methods Appl. Sci.}, {\bf 27} (2004), 545--582.

\bibitem{NSY}
 T.\ Nagai, T.\ Senba, and K.\ Yoshida, 
 {\sl Application of the Trudinger--Moser inequality to a parabolic system of chemotaxis},
 Funkcial.\ Ekvac., {\bf 40} (1997), 
 411--433.

  
 \bibitem{KNP}
 N.\ Kenmochi, M.\ Niezg\'odka, and I.\ Pawlow, 
 {\sl Subdifferential operator approach to the
Cahn--Hilliard equation with constraint},
J. Differential Equations, {\bf 117} (1995), 320--356.

\bibitem{Mi}
 A.\ Miranville, 
 {\sl Asymptotic behavior of the Cahn--Hilliard--Oono equation}, 
 J.~Appl.\ Anal.\ Comput., 
 {\bf 1} (2011),
 523--536.
 
\bibitem{MAIMS}
A.~Miranville,
{\sl The Cahn--Hilliard equation and some of its variants},
{AIMS Mathematics},
{\bf 2} (2017), 479--544.

 
\bibitem{OP1}
Y.\ Oono and S.\ Puri, 
 {\sl Study of phase-separation dynamics by use of cell dynamical systems. I. Modeling}, 
 Phys.\ Rev.~A, {\bf 38} (1988), 
 434--463.

\bibitem{OP2}
Y.\ Oono and S.\ Puri, 
 {\sl Study of phase-separation dynamics by use of cell dynamical systems. II. Two-dimensional demonstrations}, 
 Phys.\ Rev.~A, {\bf 38} (1988), 
 1542--1573.
 
 
\bibitem{Si}
 J.~Simon,
 {\sl Compact sets in the space {$L^p(0,T;B)$}},
 Ann.\ Mat.\ Pura Appl.\ (4),
 {\bf 146} (1987),
 65--96.
 
 
\bibitem{Sc}
G.~Schimperna,
{\sl On the Cahn--Hilliard--Darcy system with mass source and strongly
 separating potential}, 
Discrete Contin. Dyn. Syst. Series S, 
to appear (2022).


\bibitem{SS}
L.~Scarpa and A.~Signori,
{\sl On a class of non-local phase-field models for tumor growth with possibly singular potentials, chemotaxis, and active transport},
{Nonlinearity}  {\bf 34} (2021), 319--3250.
 
\bibitem{Vi}
 G.~Vitali, 
 {\sl Sull'integrazione per serie}\/ (Italian),
 Rend.\ Circ.\ Mat.\ Palermo,
 {\bf 23} (1907), 
 137--155.
 
 
\bibitem{Wri1}
 M. Winkler, {\sl Boundedness in the higher-dimensional 
   parabolic-parabolic chemotaxis system with logistic source}, 
  {Comm. Partial Differential Equations}, {\bf 35} (2010), 1516--1537. 

\bibitem{Wri2}
M. Winkler, {\sl Emergence of large population densities despite logistic 
growth restrictions in fully parabolic chemotaxis systems},
Discrete Contin. Dyn. Syst. Ser. B, {\bf 22} (2017), 2777--2793.

\bibitem{Wri3}
M. Winkler,
{\sl Finite-time blow-up in the higher-dimensional parabolic–parabolic Keller–Segel system},
J. Math. Pures Appl.,
{\bf 100} (2013), 748--767.
  


  
\end{thebibliography}
\end{document}